\documentclass{article}
\usepackage{amsmath}
\usepackage{amsthm}
\usepackage{amssymb}
\usepackage{hyperref}
\usepackage{verbatim}
\usepackage{graphicx}

\usepackage[usenames,dvipsnames,svgnames,table]{xcolor}
\usepackage{geometry}
\def\a{\alpha}
\def\b{\beta}

\def\Ga{\Gamma}
\def\de{\delta}
\def\De{\Delta}
\def\ep{\epsilon}

\def\la{\lambda}

\def\si{\sigma}

\def\Om{\Omega}

\def\nab{\nabla}
\def\varep{\varepsilon}

\def\DD{{\cal D}}

\def\DD{{\cal D}}

\newcommand{\N}[0]{\mathbb{N}}

\newcommand{\R}[0]{\mathbb{R}}
\newcommand{\Z}[0]{\mathbb{Z}}

\newcommand{\T}[0]{\mathbb{T}}
\newcommand{\bbE}[0]{{\mathbb E}}

\newcommand{\fr}[2]{\frac{#1}{#2}}

\newcommand{\ALI}[1]{\begin{align*} #1 \end{align*}}
\newcommand{\tx}[1]{\mbox{#1}}
\newcommand{\leqc}[0]{\lesssim}
\newcommand{\geqc}[0]{\gtrsim}

\newcommand{\pr}[0]{\partial}

\newcommand{\co}[1]{\| #1 \|_{C^0}}

\newcommand{\ctdcxa}[1]{\| #1 \|_{C_t {\dot C}_x^\a} }
\newcommand{\kotn}[0]{k_0(\T^n)}
\newcommand{\Dkdt}[1]{\fr{D_{\lesssim k}^{#1}}{\pr t^{#1}}}

\newcommand{\pleqkvcn}[0]{ P_{\leq k} v \cdot \nab }
\newcommand{\plkv}[0]{ P_{\leq k} v }
\newcommand{\lk}[0]{\leq k}
\newcommand{\Dlkdt}[0]{\fr{D_{\lk}}{\pr t}}
	
\newcommand{\ali}[1]{ \begin{align} #1 \end{align} }

\def\XXint#1#2#3{{\setbox0=\hbox{$#1{#2#3}{\int}$}
     \vcenter{\hbox{$#2#3$}}\kern-.5\wd0}}

\newtheorem{thm}{Theorem}[section]
\newtheorem{lem}{Lemma}[section]

\newtheorem{prop}{Proposition}[section]
\newtheorem{cor}{Corollary}[section]

\newtheorem{remk}{Remark}[section]
\newtheorem{ques}{Question}[section]

\theoremstyle{definition}
\newtheorem{defn}{Definition}[section]

\theoremstyle{remark}

\title{ Regularity in time along the coarse scale flow for the incompressible Euler equations }
\author{ Philip Isett\thanks{Department of Mathematics, Caltech, Pasadena, CA, USA. (\href{mailto:isett@caltech.edu}{isett@caltech.edu}).} }
\date{}

\begin{document}

\maketitle

\begin{abstract}
One of the most remarkable features of known nonstationary solutions to the incompressible Euler equations is the phenomenon known as the Taylor hypothesis, which predicts that fine scale features of the flow are advected by the mean velocity.  In this work, we develop an extensive theory of time regularity for Euler weak solutions in any dimension based on quantitative realizations of this idea.  

Our work provides the key estimates for showing that the particle trajectories of any Euler flow that is $C^\alpha$ in the spatial variables uniformly in time are of class $C^{1/(1-\alpha)}$ when $1/(1-\alpha)$ is not an integer, whether or not the trajectories or solutions are unique.  In particular, we prove the smoothness of trajectories in borderline spaces such as $v \in C^1$ or bounded vorticity in any dimension.   An essential point is the existence and improved regularity of advective derivatives of high order.
\end{abstract}

\begin{comment}
\begin{abstract}
One of the most remarkable features of known nonstationary solutions to the incompressible Euler equations is the phenomenon known as the Taylor hypothesis, which predicts that fine scale features of the flow are advected by the mean velocity.  In this work, we develop a time regularity theory for Euler weak solutions based on quantitative expressions of this idea. 

  We assume only that our velocity field is H\"{o}lder continuous in the spatial variables, which is well-motivated by considerations related to turbulence, but prohibits the application of methods based on a flow map or well-posedness theory. Despite the dramatic lack of well-posedness, we obtain a rich theory of regularity in time for solutions, especially concerning advective derivatives. In particular, any Euler flow that is $C^\alpha$ in the spatial variables uniformly in time has continuous advective derivatives of any order less than $\frac{\alpha}{1-\alpha}$, and every point has a trajectory passing through it that is of class $C^r$ for all $r < \frac{1}{1-\alpha}$, and one that is $C^\infty$ if $v$ is $C^1$ or $v \in \bigcap_{\alpha < 1} L_t^\infty C_x^\alpha$ has borderline regularity.  
In a follow up work, we show that all trajectories are of class $C^{1/(1-\alpha)}$ in time if $1/(1-\alpha) \notin {\mathbb Z}$, whether or not the trajectories are unique.
\end{abstract}
\end{comment}

\section{Introduction}

The present paper concerns weak solutions to the incompressible Euler equations
\begin{align} \label{eq:euSystem} \tag{E}
\left\{
 \begin{aligned} \pr_t v + \tx{div } v \otimes v + \nab p = 0  \\
 \tx{div } v = 0,
 \end{aligned}
\right.
\end{align}
which describe the motion of an ideal, incompressible fluid with velocity given by the vector field $v(t,x)$ and pressure given by the scalar function $p(t,x)$.  We assume only that our velocity field $v$ is H\"{o}lder continuous in the spatial variables uniformly in time ($v \in L_t^\infty C_x^\a$, $0 < \a < 1$), and therefore interpret the system \eqref{eq:euSystem} in the sense of distributions, although our results will also provide new estimates for $C^1$ solutions as well.  We will work in the periodic setting, so that $v : I \times \T^n \to \R^n$ and $p : I \times \T^n \to \R$.  Taking the divergence of \eqref{eq:euSystem}, one sees that the pressure is determined up to the addition of a scalar depending on time $C(t)$, which we will normalize so that $\int_{\T^n} p(t,x) dx = 0$ and hence $p = - \De^{-1} \pr_l \pr_j(v^j v^l)$.  We use the summation convention to indicate a sum over repeated indices.  Our assumption on $v$ implies continuity, in which case a weak solution is equivalently a continuous vector field and pressure $(v,p)$ that satisfy, for any any region $\Om \subseteq \T^n$ with smooth boundary $\pr \Om$ and interior unit normal vector $\vec{n}$, the conservation laws
\ALI{
\fr{d}{dt} \int_\Om v(t,x) dx = \int_{\pr \Om} p(t,x) \vec{n} d\si + \int_{\pr \Om} v(t,x) v\cdot \vec{n} d\si, \qquad \int_{\pr \Om} v(t,x) \cdot \vec{n} d\si = 0,
}
which express the balance of momentum and balance of mass in $\Om$ at any time $t \in I$.% for any region $\Om \subseteq \T^n$ with smooth boundary $\pr \Om$ and interior unit normal vector $n$.

In this work and a subsequent paper \cite{isettTime2} we aim to prove the following theorem:
\begin{thm} \label{thm:trajectReg} Suppose $0 < \a < 1$ with $1/(1-\a) \notin \Z$.  If $v$ is a weak solution to Euler of class $v \in L_t^\infty C_x^\a$, then every particle trajectory of $v$ is of class $C_t^{1/(1-\a)}$ in time.
\end{thm}
This theorem improves on the classical result of Chemin \cite{cheminReg} that particle trajectories are smooth when the solution is $C^{1,\alpha}$.  The main difference is that the equations are not well-posed in the regime of regularity we consider.

We now explain the background and motivation for this result and for further results we obtain below that are of relevance to turbulent flow.

\subsection{Background and Motivation}
Theories and experimental observations of turbulence in fluids, including the phenomenon of anomalous dissipation of energy in the zero viscosity limit and scaling laws predicted by the foundational theory of Kolmogorov \cite{K41}, motivate the study of H\"{o}lder continuous weak solutions to the Euler equations that may arise in the inviscid limit.  
%The study of H\"{o}lder continuous weak solutions to the Euler equations is motivated by theories and experimental observations of  turbulence in fluids, including the phenomenon of anomalous dissipation of energy in the zero viscosity limit and scaling laws predicted by the foundational theory of Kolmogorov \cite{K41}.  
More precisely, turbulent flows are modeled as solutions to the $3D$ Navier Stokes equations at high Reynolds number, meaning that the viscosity parameter $\nu$ is small relative to the characteristic velocity $V$ and length scale $L$ of the fluid.

%relative to the characteristic velocity and length scale of the fluid.  One can take units of time and space in which these characteristic lengths and velocities are of unit size, in which case the viscosity $\nu$ is small.

A famous prediction of Kolmogorov's theory of hydrodynamic turbulence, sometimes called the Kolmogorov-Obukhov law, 
%A celebrated prediction of Kolmogorov's theory of hydrodynamic turbulence (``Kolmogorov and Obukhov's law'') 
states that the differences in velocity for nearby particles in turbulent flows obey, when suitably averaged, a universal scaling law corresponding to the H\"{o}lder exponent $1/3$:% when taken ``on average'' in some appropriate sense\footnote{Here the brackets $\langle \cdot \rangle$ indicate some relevant type of averaging.}:
\begin{align}
 {\langle} |v(x + \De x) - v(x)|^p {\rangle}^{\fr{1}{p}} \sim C_p ~\varep^{\fr{1}{3}} |\De x|^{\fr{1}{3}}.  \tag{KO} \label{law:oneThird}
\end{align}
Another celebrated prediction of the theory is that the energy spectrum $E(\la)$, defined so that the energy between wavenumbers $\la_1 < \la_2$ equals $\int_{\la_1}^{\la_2} E(\la) d\la$, scales on average as  $E(\la) \sim \varep^{2/3} \la^{-5/3}$.

These laws are both derived by dimensional analysis from the basic principles of Kolmogorov's theory.  The theory asserts that the statistical properties of turbulent flows in the ``inertial range'' of length scales are independent of viscosity and are governed by the 
%viscosity parameter $\nu$ and the 
rate of energy dissipation (averaged over the measured ensemble), which for freely decaying turbulence equals $\varep = - \fr{d}{dt} \fr{1}{2} \int |v|^2 dx$. %, and that the coarse scale properties of the flow should be independent of viscosity \cite{K41}.  
A central postulate in this theory --- the hypothesis of anomalous dissipation, also known as the ``zeroth law of turbulence'' --- states that the rate of energy dissipation $\varep$ remains uniformly strictly positive in the zero viscosity limit $\nu \to 0$ of Navier-Stokes.  
%A central postulate in this theory is the hypothesis of anomalous dissipation, otherwise known as the zeroth law of turbulence, which states in one form that the rate of energy dissipation $\varep$ remains uniformly strictly positive in the zero viscosity limit $\nu \to 0$ of Navier-Stokes.  
The inertial range of length scales in which the Euler equations are expected to govern the evolution and \eqref{law:oneThird} is purported to hold ranges between the characteristic length scale and the small, Kolmogorov length scale, 
%$L \geqc | \De x | \geqc \left(\nu^3/\varep \right)^4$.  
$L \gg | \De x | \gg \left(\nu^3/\varep \right)^4$.  
Similarly, the $5/3$rds law is predicted to hold in the inertial range of wavenumbers 
%$L^{-1} \lesssim \la \lesssim \left(\nu^3/\varep \right)^{-4}$.  
$L^{-1} \ll \la \ll \left(\nu^3/\varep \right)^{-4}$.  
Thus \eqref{law:oneThird} and the 5/3rds law extend to all small length scales and large wavenumbers in the limit $\nu \to 0$, or equivalently, if one takes nondimensionalized units of space and time in which large characteristic velocities and length scales have unit size.

If laws such as \eqref{law:oneThird} or the 5/3rds law %for the energy spectrum\footnote{The energy spectrum $E(\la)$ is defined so that the energy contained between wavenumbers $\la_1 < \la_2$ is $\int_{\la_1}^{\la_2}E(\la) d\la$.  Kolmogorov's theory \cite{K41} predicts $E(\la) \sim \varep^{2/3} \la^{-5/3}$ in the inertial range of $\la$, which extends from the characteristic inverse length of the flow to $\la \lesssim \left(\nu^3/\varep \right)^{-4}$.} 
hold even as upper bounds for velocity fluctuations of solutions, they imply the existence of weak solutions to the Euler equations that arise as limits of convergent subsequences in the zero viscosity limit \cite{chenGlimm,drivasEyinksingLeray}.  These solutions would be guaranteed to have only fractional smoothness
%\footnote{Since the initial data is not assumed to be fixed in the zero-viscosity limit in this context, the existence of singular subsequential limits is entirely unrelated to the blowup question for classical Euler flows.} 
if they satisfy estimates similar to \eqref{law:oneThird} or the 5/3rds law in the inertial range.
%and  they must dissipate rather than conserve kinetic energy if energy dissipation is indeed uniform in viscosity. 
(The singularity of these subsequential limits is unrelated to the problem of blowup for classical solutions, since the initial data is viscosity dependent, being drawn from different realizations of fully developed turbulence at high Reynolds number.)  Furthermore, the limits must dissipate rather than conserve kinetic energy if energy dissipation is indeed uniform in viscosity.  
%The limiting solutions must dissipate rather than conserve kinetic energy if energy dissipation is indeed uniform in viscosity, and they would be guaranteed to have only fractional smoothness if they satisfy estimates similar to \eqref{law:oneThird} in the inertial range.  
%when suitably interpreted as an upper bound on absolute structure functions, laws such as \eqref{law:oneThird} or Kolmogorov's 5/3 law for the energy spectrum imply convergence along subsequences in the zero viscosity limit to weak Euler flows \cite{chenGlimm,constantin2018remarks,drivasEyinksingLeray}, and the limits must dissipate rather than conserve kinetic energy if energy dissipation is indeed independent of viscosity. 
  %Thus, if these laws were to persist in the $0$ viscosity limit, the theory suggests the existence of energy-dissipating solutions to \eqref{eq:euSystem} obeying the $1/3$ law \eqref{law:oneThird}.  
Experimental measurements of turbulent fluid flows give strong evidence for the presence of anomalous dissipation\footnote{The evidence generally considers turbulence with an external force present to sustain the turbulence; see \cite{drivasEyinksingLeray,vassilicos2015dissipation} for recent reviews.}
 and suggest that the $1/3$ law \eqref{law:oneThird} for absolute structure functions may hold at least for %\footnote{For $p = 2$, the law \eqref{law:oneThird} carries a special significance as it gives a physical space expression of the Kolmogorov $5/3$-law for the energy spectrum. } 
$p = 3$, while moments of fourth order and higher %measurably 
tend to be even larger than (\ref{law:oneThird}) predicts due to a phenomenon known as intermittency \cite{frischIntm, van1972statistical}.  All these considerations give motivation to consider solutions to the Euler equations with fractional regularity, even if turbulent flows themselves are more exactly modeled by regular Navier-Stokes solutions with a small viscosity.  %Note that the initial data is not assumed to be fixed in the zero-viscosity limit in this context, but rather arises from different realizations of fully developed turbulence at high Reynolds number, so the existence of singular subsequential limits is unrelated to the question of singularity formation for classical Euler flows.  
%Although these remarks concern turbulence modeled by the Navier-Stokes equation, we note that the law \eqref{law:oneThird} (and similar predictions such as Kolmogorov's 5/3rds law for the energy spectrum of a turbulent flow) imply convergence along subsequences in the zero viscosity limit to weak Euler flows when suitably interpreted as an upper bound \cite{chenGlimm,constantin2018remarks,drivasEyinksingLeray}.  
%there is a phenomenon of intermittency leading to moments of fourth order and higher which are measurably larger than predicted by (\ref{law:oneThird}), see \cite{frischIntm, van1972statistical}.  
 For more general discussion of turbulence we refer to \cite{landau1959course,frisch}.

Our main results are related to an important and classical phenomenon in turbulence known as the Taylor hypothesis, which is the assumption that, in a turbulent flow, the small scale fluctuations in velocity are carried by the mean flow.  
%The hypothesis plays a role in experimental observations of turbulence since it implies that fluctuations in velocity in time at a fixed point in space are comparable to the size of spatial fluctuations in velocity because the time fluctuations capture small scale spatial fluctuations being carried by the mean flow through the point of measurement\footnote{See \cite{zaman1981taylor,moin2009revisiting} for empirical tests of the Taylor hypothesis.}.
This assumption has played an important role in many experimental studies\footnote{See \cite{zaman1981taylor,moin2009revisiting} for more recent empirical tests of the Taylor hypothesis.} of turbulence: It implies that a stationary measurement of variations in velocity will pick up ``frozen'' spatial oscillations in velocity that are carried through the measuring point by the mean velocity, implying that time variations in velocity at a fixed position are comparable to spatial variations.
%This assumption plays a role in experimental studies of turbulence since it implies that time variations in velocity at a fixed position are comparable to spatial variations -- namely, a stationary measurement of changes in velocity captures spatial fluctuations in velocity that are carried through the measurement point by the mean velocity\footnote{See \cite{zaman1981taylor,moin2009revisiting} for empirical tests of the Taylor hypothesis.}. 
Mathematically, this idea suggests that an Euler flow with a given H\"{o}lder regularity in space will have the same H\"{o}lder regularity in time, as pointed out by \cite{frischRide}.  The present paper gives the first proof of this result, namely that a solution that is $C^\a$ in space uniformly in time $v \in L_t^\infty C_x^\a$ is also of class $C^\a$ jointly in space and time $v \in C_{t,x}^\a$ (see Theorem~\ref{thm:timeRegBounds}).  This result turns out to be only the beginning of an extensive theory of time regularity with the Taylor hypothesis at the heart of the matter.  

The central observation driving our results is a type of improved regularity in the advective derivative $(\pr_t + v \cdot \nab)$, which in the classical picture would be the time derivative along the flow, as compared to the time derivative $\pr_t$.  In fact, we show in Theorem~\ref{thm:genThm} that solutions of class $v \in L_t^\infty C_x^\a$ possess continuous advective derivatives of a high order that approaches infinity as the H\"{o}lder exponent $\a$ increases to $1$.  In contrast, the pure time derivative $\pr_t v$ is not even continuous.  Prior to this work, improved regularity of the advective derivative has been obtained only for more regular (e.g. $C^{1,\a}$) solutions in connection to the celebrated theory of smoothness %and analyticity 
of particle trajectories initiated in \cite{chemin2d, cheminReg}, 
%,katoBdrySmooth,serfati2d, gamblin, sueur,serfatiAnalytic,gamblin, glassSueTak, sueur,shnAnalytic,frischRide,constantin2015analyticity}, 
which we review further below.  Moreover, even in the context of classical solutions, the existing results are far from obtaining sharp, quantitative bounds on the advective derivatives, which is a problem that we address in what follows.
%For classical solutions of class $C^{1,\a}$, the improved regularity of the advective derivative has been understood in connection the smoothness and analyticity of their particle trajectories, which is a celebrated subject that we review below.  

For solutions of class $C^\a$, the regularity theory is much more delicate due to the lack of well-posedness or even uniqueness in this regime and the lack of Lagrangian coordinates.  Our key insight is that every frequency level of the solution moves along local, coarse scale averages of the velocity field at a specific {\it time scale} that is dictated by the regularity of the velocity field.  Analytically, we capture this time scale using a Littlewood-Paley analysis that begins with a more robust proof of the commutator estimate of \cite{CET} in Section~\ref{sec:ReynStressMatDv1}, continues through a network of further commutator estimates, and culminates in the inductive algebraic framework of Section~\ref{sec:theGeneralCase}.  We apply our estimates to give the first results on the particle trajectories of $C^\a$ solutions (Theorem~\ref{thm:smoothTraject}), including the smoothness of particle trajectories in cases of borderline function spaces such as $C^1(\R \times \R^3)$ that were previously unknown.  
In a companion paper \cite{isettTime2}, we show that all trajectories of $L_t^\infty C_x^\a$ solutions are of class $C^{1/(1-\a)}$ whenever $1/(1-\a)$ is not an integer, regardless of possible nonuniqueness of the trajectories.

%we extend our methods to obtain $C^{1/(1-\a)}$ time regularity for all trajectories whenever $1/(1-\a)$ is not an integer, regardless of their possible nonuniqueness.

%Connected to anomalous dissipation, a part of our work relates also to a
%Related to anomalous dissipation, there is a 
Our work also connects to a 
well-known conjecture of Onsager regarding turbulent energy dissipation, which states that weak solutions to the Euler equations with spatial H\"{o}lder regularity below $1/3$ may exhibit decreasing energy \cite{onsag} despite the fact that regular solutions %to \eqref{eq:euSystem} 
must conserve energy.  %\footnote{Rather than being based on dimensional analysis, Onsager's derivation of the exponent $1/3$ was based on the dynamical notion of a ``frequency cascade'' in which the primary mechanism behind the energy dissipation is the movement of energy to arbitrarily high wavenumbers (or small scales), made possible by the nonlinear term in \eqref{eq:euSystem}.  We refer to \cite{eyinkSreen, deLSzeCtsSurv} for further discussion and a review of Onsager's computations.}.  
As we will discuss further below, significant progress towards this conjecture has been made involving the construction of $C^\a$ solutions that fail to conserve energy \cite{deLSzeHoldCts}, \cite{isett} and \cite{deLSzeBuck}.  These examples of H\"{o}lder continuous solutions exhibit special time regularity properties that were part of the initial motivation for several results in the present paper.  We show here that many of these properties are consequences of the Euler equations as opposed to biproducts of the constructions.  % significant progress towards this conjecture has been made recently using the method of convex integration to construct $C^\a$ solutions to Euler which fail to conserve energy, but Onsager's conjecture currently remains out of reach.  
%At the same time, part of the present work aims to clarify how this regularity in time poses an obstruction to improving the progress towards Onsager's conjecture for the current methods of construction.  This obstruction will be elaborated in the concluding remarks of the paper.  
Following the original preprint release of this paper, these methods most recently have culminated in a proof of Onsager's Conjecture in three dimensions \cite{isettOnsag} (see also \cite{BDLSVonsag,isettEndpt}), though there exist other variants of the conjecture such as the two dimensional case that remain unsolved.  %The proof in \cite{isettOnsag} uses analytical ideas developed in the present work.

%In the conclusion of the paper, we discuss a conjecture for the Euler equations related to the above ideas and to Theorem \ref{thm:energyReg} below, which offers an explanation as to why energy dissipation should be a nongeneric and unstable phenomenon for solutions with regularity below $1/3$.  This conjecture suggests that the slightest departure from the $1/3$ scaling law will generically lead to the failure of energy dissipation.

%\begin{comment}
Onsager also stated that an Euler flow with H\"{o}lder regularity $v \in L_t^\infty C_x^\a$ in space must conserve energy if $\a > 1/3$.  This statement has been proven in \cite{CET} after a slightly weaker result was obtained in \cite{eyink}.  In \cite{energyCons}, the proof of energy conservation was extended to the critical Besov space $L_t^3 B_{3, c(\N)}^{1/3}$.  The space $B_{3, c(\N)}^{1/3}$ corresponds to a regularity just slightly better than the law \eqref{law:oneThird} with $p = 3$, while the power law \eqref{law:oneThird} with $p = 3$ corresponds instead to the critical Besov space $B_{3, \infty}^{1/3}$, where the proof of energy conservation fails to go through.  An example of \cite{eyink} that is extended in \cite{energyCons} shows that the space $B_{3, \infty}^{1/3}$ admits divergence free vector fields whose instantaneous energy flux is nonzero.  For this reason, energy dissipating solutions to Euler with $B_{3,\infty}^{1/3}$ spatial regularity have been considered as a natural setting for exhibiting anomalous dissipation and for containing turbulent solutions that may arise in the $0$ viscosity limit (see for example \cite{cheskidov2014euler} for a recent Littlewood-Paley approach to the theory of intermittency; see also the proof of Proposition~\ref{prop:firstLPpieceBds} as well as Section~\ref{sec:timeRegEnergy} below for more discussion).  

In the conclusion of the paper, we discuss a conjecture for the Euler equations related to the above ideas and to Theorem \ref{thm:energyReg} below, which offers an explanation as to why energy dissipation should be a nongeneric and unstable phenomenon for solutions with regularity below $1/3$.  This conjecture suggests that the slightest departure from the $1/3$ scaling law \eqref{law:oneThird} (in an $L^3$ based topology) will generically lead to the failure of energy dissipation.  %In these concluding remarks, we also discuss a conjecture for the Euler equations related to the above ideas and Theorem \ref{thm:energyReg} below, which offers an explanation as to why energy dissipation should be an unstable phenomenon for solutions with regularity below $1/3$.

With this background and motivation in hand, we state the main results of the paper.

\subsection{Statement of Results}

%We state in this Section the main results of the paper, which concern the regularity in time of solutions to the incompressible Euler equations in the class $v \in L_t^\infty C_x^\a$, and the improved regularity of the advective derivative $\pr_t + v \cdot \nab$.  

Before stating the theorems, recall that a weak solution to Euler of class $v \in L_{t,x}^2$ has a uniquely defined restriction $v(t_0,\cdot) \in \DD'(\T^n)$ on any fixed time slice $t_0$, with $v(t_0)$ weakly continuous in $t_0$.  If $v \in L_t^\infty C_x^\a$, $0 < \a < 1$, the weak continuity implies that $v(t_0) \in C^\a$ not only for almost every $t_0$, but also for all $t_0$, and although we may not have continuity in time with values in $C^\a$, we obtain via weak continuity the equality of norms %and the restriction to every time slice belongs to $C^\a$ and satisfies the equality
\ALI{
\| v \|_{C_tC_x^\a} \equiv \sup_{t \in I} \| v(t,\cdot) \|_{C^\a} = \| v \|_{L_t^\infty C_x^\a},
}
Using compactness of the embedding $C^\a(\T^n) \hookrightarrow C^0(\T^n)$ and weak continuity again, one can show that $v$ is continuous jointly in $(t,x)$.  We will henceforth always assume $v$ is taken to equal its continuous representative.

Our first theorem states that a solution in the class $v \in L_t^\infty C_x^\a$ for $0 < \a < 1$ must also have the same H\"{o}lder regularity in time, and that the pressure must be essentially twice as regular in both space and time.   
%In the statement of the theorem, we use implicitly the fact that, by weak continuity in time, a solution of class $v \in L_t^\infty C_x^\a$ can be redefined on a set of times with measure $0$ to be continuous.
\begin{thm} \label{thm:timeRegBounds}
Let $0 < \b < \a < 1$ and suppose that $(v,p)$ solve the incompressible Euler equations \eqref{eq:euSystem} 
%\begin{align} \label{eq:euSystem}
%\left\{
% \begin{aligned} \pr_t v + \tx{div } v \otimes v + \nab p = 0  \\
% \tx{div } v = 0
% \end{aligned}
%\right.
%\end{align}
in the sense of distributions $I \times \T^n$ for some torus $\T^n, n \geq 2$ and some open interval $I$, with $v \in L_t^\infty C_x^\a$.  %Suppose also that the norm $\| v \|_{C_tC_x^\a} = \sup_{t \in I} \| v(t,\cdot) \|_{C^\a} = \| v \|_{L_t^\infty C_x^\a}$ is finite.  %Theorems \ref{thm:timeRegBounds}-\ref{thm:energyReg} 
%in any of the theorems of this paper and can just as well assume $v \in L_t^\infty C_x^\a$ in all of our estimates.} and that the pressure is normalized to have integral $0$.  
Then $v \in C_{t,x}^\a$ is jointly $C^\a$ in both space and time.  If $2\b < 2 \a < 1$,  we furthermore have
%Then if $\b \leq 1/2$, we have
\begin{align}
 p(t,x) \in C_{t,x}^{2 \b} \cap L_t^\infty C_x^{2\a},
\end{align}
while if $1 < 2\b < 2\a$, we have
\begin{align}
% v(t,x) \in C_{t,x}^\a, \qquad  (\pr_t p, \nab p) \in C_{t,x}^{2 \b - 1}
\pr_t p \in C_{t,x}^{2 \b - 1}, \qquad \nab p \in C_{t,x}^{2 \b - 1} \cap L_t^\infty C_x^{2\a - 1}.
\end{align}
%and $\nab p \in L_t^\infty C_x^{2\a - 1}$.
\end{thm}
The $C^{2\a}$ regularity of the pressure can be viewed as a rigorous version of a classical heuristic prediction of the Kolmogorov theory, which is that, based on dimensional analysis, fluctuations in the pressure should scale as the square of fluctuations in the velocity (see e.g. \cite{landau1959course}).  Meanwhile, the $C_{t,x}^\a$ regularity for $v$ verifies the prediction in \cite{frischRide} that, based on the Taylor hypothesis, H\"{o}lder continuity in space should imply the same H\"{o}lder continuity in time for the velocity field.

 %In the statement of the theorem, we implicitly use the fact that any Euler flow of class $v \in L_t^\infty C_x^\a$ has a canonical restriction to {\it every} time-slice $t_0$ that is of class $v(t_0, \cdot) \in C^\a$ and obeys the bound $\| v(t_0) \|_{C^\a} \leq \| v \|_{L_t^\infty C_x^\a}$, since $v$ is weakly continuous in time.  We do not need to assume that $v$ is continuous in $t$ with values in $C^\a$.

 %conjectured in \cite{frischRide}.

\begin{comment}
The statement that H\"{o}lder continuity in space implies the same H\"{o}lder continuity in time verifies an intuitive consequence of the Taylor hypothesis.  That is, if one assumes that small scale oscillations in space are carried along coarse scale averages of the flow, then one intuitively expects that the H\"{o}lder regularity in time should be the same as the H\"{o}lder regularity in space \cite{frischRide}.  %The fact that the pressure has twice as much regularity, both in space and time, may be regarded as less intuitive given that the pressure is initially determined only up to an arbitrary function of time.
\end{comment}

%The method of the paper can also be extended to higher regularity in a straightforward way (for example, one can show that $p \in C_{t,x}^{2,\b}$ for $v \in C_t C_x^{1,\a}$ and $0 < \b < \a < 1$) although we do not pursue such an extension here.
Assuming further regularity on the velocity field leads to further regularity of advective derivatives, as the following theorem illustrates.
\begin{thm} \label{thm:higherMatDv}

If $\a > 2/3$, then $- \fr{D^2}{\pr t^2}v = \fr{D }{\pr t} \nab p = \pr_t \nab p + \tx{{\normalfont div} }( v \otimes \nab p )$, which is well-defined as a distribution by Theorem (\ref{thm:timeRegBounds}), is also continuous.
\end{thm}

The time regularity properties for the velocity field stated in Theorems \ref{thm:timeRegBounds} and \ref{thm:higherMatDv} are particular examples of the following, more general theorem.
\begin{thm} \label{thm:genThm}
%If $\a > 1/2$ and $r$ is an integer satisfying $0 \leq r < \fr{2 \a - 1}{1-\a}$, then for $\si = 2 \a - 1 - r(1-\a)$
If $0 < \a < 1$ and $r$ is an integer satisfying $0 \leq r < \fr{\a}{1-\a}$, then for $\si = \a - r(1-\a)$
%\ali{
%\fr{D^{r-1}}{\pr t^{r-1}} \nab p = \fr{D^{r} v}{\pr t^{r}} &\in C_{t,x}^\si,  \qquad 
%}
\ali{
 \fr{D^{r} }{\pr t^{r}}v = - \fr{D^{r-1}}{\pr t^{r-1}} \nab p  &\in C_{t,x}^\si,  \qquad 
}
where the material derivative of a continuous tensor field $T$ is defined to be the distribution given by $\fr{D}{\pr t} T = \pr_t T + \pr_j( v^j T)$.  
\end{thm}
As a corollary, we obtain infinitely many advective derivatives when $v$ is borderline Lipschitz:
\begin{cor}  If $\ctdcxa{v}$ is bounded for every $\a < 1$, which includes the case where the vorticity is bounded, then $v$ has continuous advective derivatives of every order.
\end{cor}
%In particular, Theorem \ref{thm:genThm} shows that the velocity field has continuous advective derivatives of every order whenever the norm $\ctdcxa{v}$ is bounded for all $\a < 1$, which includes in particular the case where the vorticity is uniformly bounded.  
In contrast, $C^1$ regularity in space is apparently insufficient to even yield continuity of the first time derivative $\pr_tv$.  The proof of Theorem~\ref{thm:genThm} provides new and presumably sharp estimates on advective derivatives, which are of interest also in the case of classical solutions.

We prove furthermore that the pressure possesses better regularity in its advective derivatives, which turns out to be even more subtle than the results for the velocity field.  %To state our results, we use the notation $y_+ = \max \{y, 0\}$.
%\[ (1 - 2 \a)_+ = \begin{cases} (1 - 2 \a) &\tx{if } (1 - 2 \a) \geq 0 \\ 0 &\tx{if } (1 - 2 \a) < 0 \end{cases} \]
Regularity in time for the pressure may seem surprising given that, a priori, the pressure is only determined by the equations up to a scalar depending on time.  However, when the pressure is properly normalized to have integral $0$, we obtain Theorem~\ref{thm:matDvPCts} below.  %(Here we use the notation $(y)_+ = \max \{ y, 0\}$.)
\begin{comment}
\begin{thm} \label{thm:matDvPCts}
%, and if $\a > 1/2$, the distribution $\fr{D^2p}{\pr t^2} = \pr_t \fr{Dp}{\pr t} + \tx{{\normalfont div} } ( \fr{D p}{\pr t} v )$ is continuous.  
Under the same assumptions, for any non-negative integer $s$ satisfying $s(1-\a) - 1 + (1 - 2 \a)_+ < 0$ we have \[ \fr{D^s}{\pr t^s} p \in C_{t,x}^\b\] for all $\b < 1 - (1 - 2 \a)_+ - s(1-\a)$.  In particular, if $\a > 1/3$, the distribution $\fr{Dp}{\pr t} = \pr_t p + \tx{{\normalfont div} } ( p v )$ is continuous, and the pressure associated to a uniformly Lipschitz Euler flow will have continuous advective derivatives of all orders.
%
\end{thm}
\end{comment}
\begin{thm} \label{thm:matDvPCts}
%, and if $\a > 1/2$, the distribution $\fr{D^2p}{\pr t^2} = \pr_t \fr{Dp}{\pr t} + \tx{{\normalfont div} } ( \fr{D p}{\pr t} v )$ is continuous.  
Under the same assumptions, for any non-negative integer $s$ with $s < \fr{\min\{1,2\a\}}{(1-\a)}$
 \[ \fr{D^s}{\pr t^s} p \in C_{t,x}^\b\] for all $\b <\min\{2\a, 1\} - s(1-\a)$.  In particular, if $\a > 1/3$, the distribution $\fr{Dp}{\pr t} = \pr_t p + \tx{{\normalfont div} } ( p v )$ is continuous, and the pressure associated to a uniformly Lipschitz Euler flow will have continuous advective derivatives of all orders.
\end{thm}
In relation to the ideas of anomalous dissipation and Onsager's conjecture, we are further motivated to consider the regularity in time of the energy profile
\[ e(t) = \int \fr{|v|^2}{2}(t,x) dx \]
The following theorem shows that the energy profile turns out to possesses much better regularity in time than both the pressure and the solution itself.
\begin{thm} \label{thm:energyReg} Under the conditions of Theorem \ref{thm:timeRegBounds}, the energy profile $e(t)$ satisfies %a H\"{o}lder condition
\ali{
|e(t + \De t) - e(t)| &\leq C |\De t|^{\fr{2 \a}{1 - \a}} \label{incl:energyRegularity}
}
for some $C$ depending on $v$ and $\alpha$.  Furthermore, if the solution $v$ is of class $L_t^p B_{3,\infty}^{1/3}$ for some $p \geq 3$, then the energy profile $e(t)$ is of class $e(t) \in W^{1,p/3}$ in time.
%\begin{align}
%e(t) = \int \fr{|v|^2}{2}(t,x) dx &\in C_t^\fr{2 \a}{1 - \a} \label{incl:energyRegularity}
%\end{align}
%for $\a < 1/3$, and is Lipschitz in time for $\a = 1/3$.
\end{thm}
\noindent Theorem~\ref{thm:energyReg} and its proof should be compared with the results of \cite{CET, energyCons} concerning the positive direction of Onsager's conjecture on the conservation of energy at regularity above $1/3$ (see Section~\ref{sec:timeRegEnergy} below).  %In particular, we remark that the same argument in \cite{energyCons} proves estimates for the rate of variation of energy $\| \fr{d}{dt} \int |v|^2(t,x) dx \|_{L_t^1}$ in terms of the $L_t^3 B_{3, \infty}^{1/3}$ norm of the velocity, even though the proof does not give conservation of energy.  %(This observation gives the $\a = 1/3$ case of Theorem~\ref{thm:energyReg}.)  

We conjecture that Theorem~\ref{thm:energyReg} is sharp, and that moreover the energy profile of generic solutions with H\"{o}lder regularity below $1/3$ will possess no better regularity than the H\"{o}lder estimate \eqref{incl:energyRegularity}.  In particular, the energy profile of such flows should fail to have bounded variation (and thus fail to be monotonic) on every open interval.  In this sense, we expect that the property of energy dissipation will generically fail upon the slightest departure from the $1/3$ law (or more specifically the regularity $B_{3,\infty}^{1/3}$).  We will resume discussion of this conjecture in the concluding remarks.

%perturbation in the $C_t C_x^{1/3 - \ep}$ topology of Euler flows, so that only function spaces which obey the $1/3$ law exactly (for instance, the space $L_t^\infty B_{3, \infty}^{1/3}$) could be considered as a setting in which energy dissipation exists as a stable phenomenon.

%as the only possible regularity at which energy dissipation may be stable under perturbation.  

Our last demonstration of time regularity concerns the smoothness of trajectories of incompressible Euler flows.  
%Theorem \ref{thm:genThm} 
Theorem~\ref{thm:smoothTraject} that follows  
recovers in the periodic setting the well known result that the particle trajectories of classical solutions to the Euler equations (more precisely, solutions in the class $L_t^\infty C_x^\a$ with $\a > 1$) are smooth curves.  This fact was first proven by Chemin in the setting of $\R^d$ \cite{chemin2d, cheminReg} and generalized to bounded domains in \cite{katoBdrySmooth}.  In dimension $2$, the smoothness of trajectories is known in spaces with bounded vorticity in which the Euler equations are well-posed \cite{serfati2d, gamblin, sueur}.  In fact, starting with the investigations of \cite{serfatiAnalytic}, it has been shown for classical, $C^{1,\a}$ solutions that the particle trajectories and even the motion of a rigid body immersed in an incompressible fluid are analytic in time \cite{gamblin, glassSueTak, sueur}.  We refer to \cite{sueur} for a summary of this activity, and also note the more recent proofs of \cite{shnAnalytic} and \cite{frischRide} along with the work of \cite{constantin2015analyticity} on more general, well-posed incompressible equations.

The above works on the smoothness of trajectories all rely on the existence and local well-posedness theory for the Euler equations and in many cases proceed in Lagrangian coordinates.  None of these ingredients are available in our setting, as we consider $v \in L_t^\infty C_x^\a$ with fractional regularity $\a < 1$.  For velocity fields with this regularity, particle trajectories may fail to be unique and it is not known whether a sensible flow map can be defined or whether uniqueness holds for the Euler initial value problem.  %or whether any kind of uniqueness can be established for the initial value problem.  
In dimensions $3$ and higher, the existence and local well-posedness theory for Euler is restricted to velocity fields that have regularity in space that is strictly better than $C^1$ (see \cite{pak2004existence} for a critical result in this direction).  In fact, an example of \cite{bardTiti} shows that solutions to the $3$D Euler equations can instantaneously lose regularity in $C^\a$ for any $\a < 1$ (even when the zero viscosity limit is unique), and nonuniqueness of solutions to the initial value problem is known for %$\a < 1/5$ \cite{isett} and more recently 
$\a < 1/3$ \cite{isettOnsag}.% by exhibiting a solution $v(t,x)$ with initial data $v(0,x) \in C^\a$ for which $v(t,\cdot)$ fails to belong to $C^{\a^2 + \ep}$ for $\ep > 0$ and $t \neq 0$.

Our approach to Theorem \ref{thm:matDvPCts}, in contrast, is completely from an Eulerian point of view, and entirely avoids  %independent of 
the existence and local well-posedness theories.  Furthermore, the proof reveals a specific physical phenomenon that is responsible for this improved regularity: Namely, there is a pattern in the time scales of motion for the high frequency components of the solution, and this time scale becomes uniform over frequency once one approaches Lipschitz regularity.   %of the advective derivative 
%by revealing a particular pattern in the time scales of motion for the high frequency components of the solution.  %: namely, the time scale of motion exhibited by the high frequencies fluctuations becomes more or less rapid at a specific rate that is dictated by the regularity of the fluid.  
As an application of our approach to Theorem \ref{thm:matDvPCts}, we establish the smoothness of trajectories in a case of borderline regularity that lies below the well-posedness and uniqueness thresholds.  
\begin{thm}\label{thm:smoothTraject} [Existence of High-Regularity Trajectories]
Suppose that that the solution satisfies $v \in L_t^\infty C_x^\a$ for all $\a < 1$.  Let $t_0 \in I$ and $x_0\in \T^n$ be given.  Then there exists $X(t) : I \to \T^n$ satisfying %the initial value problem
\ali{
\label{eq:particleTrajectory}
 \fr{d}{dt}X(t) = v(t, X(t)), \qquad  X(t_0) = x_0 
}
%has a unique solution $X(t,x_0) : I \to \T^n$ for the givien initial point $x_0 \in \T^n$.  
such that $X(t)$ is of class $C^\infty(I)$ and satisfies $\fr{d^{r+1}}{dt^{r+1}} X(t,x_0) = \fr{D^r}{\pr t^r}v(t, X(t,x_0))$ for any $r \geq 0$.  

If one assumes only that $v \in L_t^\infty C_x^\alpha$ for some $\alpha < 1$, then through any initial $(t_0, x_0)$ there exists a trajectory $X(t)$ passing through it that is of class $C^r$ for all integers $r < \frac{1}{1-\alpha}$.
\end{thm}
Note that the trajectory in Theorem~\ref{thm:smoothTraject} is not necessarily unique, even in the case that $v \in \bigcap_{\a < 1} L_t^\infty C_x^\alpha$.  However, if the velocity field is say log-Lipschitz (which includes the case of bounded vorticity), trajectories are unique and the theorem implies that every trajectory is of class $C^\infty$.  Thus we recover and improve on the result of Chemin \cite{chemin2d, cheminReg} in the periodic setting.  In particular, this result proves smoothness of trajectories for $C^1$ solutions, which had been an open problem in dimension three.

  %On the other hand, for solutions of class $L_t^\infty C_x^\alpha$ with $\alpha < 1$, the same proof shows that $C^r$ trajectories exist through any point for all integers $r < \frac{1}{1-\alpha}$.   %Theorem~\ref{thm:smoothTraject} can also be extended to construct particle trajectories $X(t,x_0)$ that have improved $C^r$ regularity in time when we only assume $v \in C_tC_x^\a$ for some positive $\a$ strictly less than $1$.  We believe it is an interesting question to determine whether there exist particle trajectories failing to exhibit this improved regularity.

In a companion paper \cite{isettTime2}, we will improve on Theorem~\ref{thm:smoothTraject} to prove Theorem~\ref{thm:trajectReg}, which is a more precise result on the time regularity of {\it arbitrary} trajectories of Euler flows in the $C^\a$ class, whether or not the trajectories are unique.

%In a companion paper \cite{isettTime2}, we will improve on Theorem~\ref{thm:smoothTraject} to give the following, more precise result on the time regularity of {\it arbitrary} trajectories of Euler flows in the $C^\a$ class, whether or not the trajectories are unique.
\begin{comment}
\begin{thm} \label{thm:trajectReg} Suppose $0 < \a < 1$ with $1/(1-\a) \notin \Z$.  If $v$ is a weak solution to Euler of class $v \in L_t^\infty C_x^\a$, then every particle trajectory of $v$ is of class $C_t^{1/(1-\a)}$ in time.
\end{thm}
\end{comment}

The estimates established in the present paper form the groundwork for the proof of Theorem~\ref{thm:trajectReg}.

%The proof of Theorem~\ref{thm:trajectReg} relies heavily on the estimates established in the present paper.

\subsection{Strategy of the Proof}

Theorems \ref{thm:timeRegBounds}-\ref{thm:smoothTraject} all relate to the phenomenon that the material derivative in general has better regularity than the stationary time derivative.  One basic example that lies at the heart of this phenomenon is the following bound on the material derivative of a Littlewood-Paley projection of the solution
\begin{lem} \label{lem:coarseScaleMatDv}
There is a universal constant $C$ such that if $(v, p)$ solve the Euler equations and are of class $v \in L_t^\infty C_x^\a$ then each Littlewood-Paley projection $P_k v$ satisfies
\begin{align}
\co{ (\pr_t + v \cdot \nab) P_k v } &\leq C 2^{(1 - 2 \a) k } \| v \|^2_{C_t{\dot C}_x^\a} \label{bd:matDvOfLpPiece}
\end{align}
where the homogeneous H\"{o}lder seminorm of a vector field on $\T^n$ is defined by
\[ \| v \|_{C_t{\dot C}_x^\a} = \sup_t \sup_{h \in \R^n \setminus \{0\} } \fr{|v(t, x + h) - v(t,x)|}{|h|^\a}  \]
\end{lem}
The constant $C$ does not depend on the torus and the bound holds on $\R^n$ as well.  Actually, the bound we establish directly below is
\begin{align}
\co{ (\pr_t + P_{\leq k} v \cdot \nab) P_{k+1} v } &\leq C 2^{(1 - 2 \a) k } \| v \|^2_{C_t{\dot C}_x^\a} \label{bd:matDvOfLpPiece2}
\end{align}
which is basically equivalent to (\ref{bd:matDvOfLpPiece}) but is more robust.  Here $C^0 = C^0_{t,x}$ is a supremum over time and space.  %The estimates in our paper hold also with the supremum in $C_t$-based norms replaced by the essential supremum in  $L_t^\infty$-based norms, and admit many generalizations to other function spaces.  %We remark here that stating the bounds with a supremum over $t \in I$ is slightly wasteful.  In fact the bounds are completely local in time in the sense that none of the constants in the estimates depend on the interval $I$.%, and we can therefore just as well replace the quantity $\ctdcxa{v}$ with $\| v(t,\cdot) \|_{{\dot C}_x^\a}$ at any given time $t$.

Observe that the estimate (\ref{bd:matDvOfLpPiece}) is consistent with dimensional analysis of the Euler equations (both sides having dimensions of $\fr{\tx{length}}{\tx{time}^2}$ where $2^k$ is an inverse length) and also remains invariant under Galilean transformations.  In contrast, from the identity
\[ \pr_t = (\pr_t + P_{\leq k} v \cdot \nab) - P_{\leq k} v \cdot \nab, \] 
we can only expect a weaker estimate
\ALI{
 \co{ \pr_t P_k v } &\leq C ( 2^{(1 - 2 \a) k } \| v \|^2_{C_t{\dot C}_x^\a} + 2^{(1 - \a)k} \co{ v } \| v \|_{C_t{\dot C}_x^\a}  ) \\
\co{ \pr_t P_k v } &\leq C_{\T^n} 2^{(1 - \a)k} \| v \|^2_{C_tC_x^\a} 
}
for the (stationary) time derivative $\pr_t P_k v$ of a Littlewood-Paley piece (which obviously fails to be Galilean invariant).  One can therefore interpret \eqref{bd:matDvOfLpPiece} and its proof (along with many other bounds in the present paper) as a demonstration that the fine scale features of an ideal incompressible flow must move along the coarse scale flow.  

The idea that the high frequency oscillations of the velocity field should move along the coarse scale average flow has played an important role in other parts of the analysis of fluids and the Euler equations.  This idea can be viewed as a more precise version of the classical Taylor hypothesis of frozen turbulence, which asserts that microstructures in turbulent fluids are convected along the large-scale average velocity of the fluid.  Besides serving as a key assumption in many experimental measurements of turbulent flows, the Taylor hypothesis 
%The Taylor hypothesis is a key assumption in many experimental measurements of turbulent flows and 
has also played a role in the applied literature, where the assumption of ``convected fluid microstructure'' forms the basic hypothesis underlying multiscale analysis approaches to modelling fluid turbulence.  We refer to \cite{holm2005lans, holmTronci} and the references therein for more on these topics.  The transport of high frequency waves by a low frequency velocity field has also been a key idea in the construction of energy-dissipating Euler flows used in progress towards Onsager's conjecture \cite{isett, deLSzeBuck} and turned out to play a key role in the resolution of the conjecture \cite{isettOnsag}.

%in the recent progress towards Onsager's conjecture \cite{isett}, \cite{deLSzeBuck}.  
	
%	has also appeared in the applied literature,  where the assumption of ``convected fluid microstructure'' forms the basic hypothesis underlying multiscale analysis approaches to modelling fluid turbulence.  We refer to \cite{holmTronci} and the references therein for more on these developments.

The estimates for coarse scale advective derivatives we obtain obey a pattern that indicates that the motion of each frequency component takes place at a time scale that is naturally dictated by the spatial regularity and dimensional analysis.  This time scale agrees with the Kolmogorov theory of turbulence when $\a = 1/3$.  Specifically, when we compare \eqref{bd:matDvOfLpPiece2} to the bound $\co{ P_{k+1} v } \leq C 2^{-\a k} \ctdcxa{v}$, we see that the coarse scale advective derivative $(\pr_t + P_{\leq k} v \cdot \nab)$ costs a factor $2^{(1- \a) k} \ctdcxa{v}$ in the estimate.  This cost is a general feature in the structure of the estimates in the paper when they do hold.  One can interpret these bounds informally as expressing that the oscillations at wavenumber $\la$ change within a natural time scale of $\la^{-(1-\a)} \ctdcxa{v}^{-1}$ when they are observed along the coarse scale flow.  Being consistent with dimensional analysis, this time scale agrees in the case $\a = 1/3$ with the time scale $\varep^{-1/3} \la^{-2/3}$ that is predicted for the turnover time of turbulent eddies with length scale $\la^{-1}$ by the scaling considerations in Kolmogorov's theory.  (See also Section~\ref{sec:timeRegEnergy} and in particular inequality \eqref{eq:energyVariationIneq} below for more explicit comparisons between the energy variation $\varep$ and the homogeneous norms $\ctdcxa{v}$.)  In contrast, the time scale corresponding to a stationary frame of reference, as reflected in the bounds for the stationary time derivative $\pr_t$, is much more rapid, being of the order $\la^{-1}$.  %We emphasize, however, that this time scale for the high frequencies is only observed along the coarse scale flow, and not from a stationary point of view.

One motivation for the regularity statements
%The regularity statements 
in Theorems \ref{thm:timeRegBounds} and \ref{thm:energyReg} as well as Lemma \ref{lem:coarseScaleMatDv} 
%are motivated by 
comes from 
the properties demonstrated for the weak solutions to Euler constructed in \cite{deLSzeHoldCts}, \cite{isett} and \cite{deLSzeBuck} as well as a conjecture studied in \cite{isett}.   These papers make progress towards Onsager's conjecture by exhibiting various, different constructions of H\"{o}lder continuous, periodic Euler flows that fail to conserve energy.  %, although the exponent $1/3$ remains out of reach.  
These solutions are shown to have many of the properties stated in Theorems \ref{thm:timeRegBounds} and \ref{thm:energyReg}; however, these properties are all proven using the explicit form of the building blocks of the constructions (see in particular Remark 1.2 of \cite{deLSzeHoldCts} regarding the improved regularity in space for the pressure).  Theorems \ref{thm:timeRegBounds}-\ref{thm:matDvPCts} and inequality \ref{bd:matDvOfLpPiece} show that despite allowing for a great amount of flexibility, the Euler equations impose nontrivial constraints (beyond simply the conservation of momentum) even on low regularity solutions.  In particular, these bounds impose constraints on what kind of scheme one can use to attack Onsager's conjecture, as they demonstrate that an improvement in spatial regularity of solutions must be accompanied by improvements in the regularity in time along the coarse scale flow within the construction.  %(see Section \ref{sec:conclusion} for further discussion).  On the other hand, the above results are still consistent with the conjectural ``ideal case scenario'' studied in \cite{isett}, which would imply Onsager's conjecture if the convex integration scheme for Euler could be sufficiently improved.
Such constraints remain relevant for the pursuit of still open variants of Onsager's conjecture such as the two dimensional case and Onsager's conjecture for Euler flows exhibiting local dissipation of energy.

% Ill-posedness
\begin{comment}
We remark that one cannot expect to obtain results such as Theorems (\ref{thm:timeRegBounds}-(\ref{thm:matDvPCts}) if one only imposes assumptions on the initial velocity $v(0,x)$ because the Euler equations are ill-posed in $C_t C_x^\a$ if $\a < 1$.  Bardos and Titi have demonstrated this ill-posedness in \cite{bardTiti} by exhibiting for any $\a < 1$ a solution to three-dimensional Euler which has initial data $v(0,\cdot) \in C_x^\a$ such that $v(t,\cdot)$ fails to be in $C_x^\b$ for any $\b > \a^2$ when $t \neq 0$.  It is also known in dimension $3$ that weak solutions in the class $C_t C_x^\a$ are nonunique for arbitrary, smooth initial data when $\a < 1/5$, \cite{isett}.\footnote{This result can be generalized to two dimensions as well using the main idea of \cite{deLSzeCts2d}.}  However, Theorem (\ref{thm:genThm}) does give the following corollary in two dimensions.
\begin{cor}  Suppose $v(0,x) : \T^2 \to \R^2$ is divergence free and has bounded vorticity.  Let $v(t,x)$ be the unique weak solution the $2D$ Euler equations obtained by the vanishing viscosity limit.  Then $v(t,x)$ has continuous material derivatives of all orders in the sense of Theorem (\ref{thm:genThm}).
\end{cor}
\begin{proof}
By the Calder\'{o}n-Zygmund Theory and the conservation of momentum and vorticity, the norms $\| v(t,\cdot) \|_{W^{1,p}}$ remain bounded uniformly in time for all $p < \infty$.  By Sobolev embedding, the H\"{o}lder norms $\| v \|_{C_t C_x^\a}$ are bounded for all $\a < 1$, so Theorem (\ref{thm:genThm}) applies.
\end{proof}
\end{comment}

A key ingredient in the proofs of Theorems \ref{thm:timeRegBounds}-\ref{thm:matDvPCts} and Lemma \ref{bd:matDvOfLpPiece} is the fundamental commutator estimate for the Reynolds stress used in the proof of energy conservation in \cite{CET}.  We give a new proof of this estimate below that is more robust in that it allows us to take advective derivatives as well as spatial derivatives.  Some techniques in the proof below draw inspiration from analogous estimates for solutions constructed by convex integration in \cite{isett}, and we will comment on these analogies in the course of the proof.

\subsection{ Organization of the paper}

The main results of the paper are Theorems \ref{thm:genThm} and \ref{thm:matDvPCts}, which are proven by an induction on the number of material derivatives.  Sections \ref{sec:theProof} - \ref{sec:firstRegTimePress} focus on estimates and applications for the first two material derivatives, which form a base case for the overall induction and provide the clearest setting to demonstrate many of the main ideas.  Along the way in Section \ref{sec:timeRegEnergy} we prove Theorem \ref{thm:energyReg} on the regularity of the energy profile.  The notation of the paper is summarized in Sections \ref{sec:notation1} and \ref{sec:notationPrelims}.  The remainder of Section \ref{sec:theGeneralCase} is devoted to the full proof of Theorems \ref{thm:genThm} and \ref{thm:matDvPCts}.  In Section \ref{sec:smoothTraject}, we use the results of Section \ref{sec:theGeneralCase} to establish Theorem \ref{thm:smoothTraject}.
  We conclude the paper in Section \ref{sec:conclusion} by discussing the relationship of the present results to convex integration and some questions which are motivated by the present work, including a conjecture related to Theorem \ref{thm:energyReg} and the idea of anomalous dissipation.

\subsection{Acknowledgements}

The author thanks P. Constantin, C. De Lellis and V. Vicol for conversations related to Theorem~\ref{thm:timeRegBounds} and for discussions regarding an earlier draft of the paper.  The author also thanks S.-J. Oh for discussions that motivated the proof of the endpoint regularity for $v$.  %The author is grateful to the anonymous referee for several helpful suggestions for improving the paper.  
%The author acknowledges the support of the NSF Graduate Research Fellowship Grant DGE-1148900.  
This material is based upon work supported by the NSF under Awards DGE-1148900, DMS-1402370 and DMS-2055019.

\section{The proof} \label{sec:theProof}

We consider solutions to the Euler equations, written here using the summation convention
\begin{equation} \label{eq:euEqns}
\left\{
 \begin{aligned} \pr_t v^l + \pr_j(v^j v^l) + \pr^l p = 0  \\
 \pr_j v^j = 0,
 \end{aligned}
\right.
\end{equation}
for any torus $\T^n = \R^n / \Ga$.  Many of the estimates proven below have universal constants that do not depend on the torus.  In some cases when low frequencies play a large role in the estimates, the bounds depend on the torus, generally through the size of the lowest frequency on the dual of $\T^n$ (which can be regarded as a characteristic inverse length for the flow).  For the first part of the argument, we state along the way which estimates depend on the torus.

The starting point for our proof is the fact that any solution to the Euler equations \ref{eq:euSystem} has Littlewood-Paley components obeying the bound \ref{lem:coarseScaleMatDv} for their material derivatives.  The material derivative estimate in Lemma \ref{lem:coarseScaleMatDv} follows from the commutator estimate extending the one of \cite{CET}  %Constantin, E and Titi %\footnote{ More precisely, we require a similar estimate for the derivative of the commutator, which appears in \cite{deLSzeC1iso}. } 
together with the following estimate for the Littlewood-Paley components of the pressure.

\begin{lem} If $(v,p)$ solve the incompressible Euler equations and $0 < \a \leq 1$, then the Littlewood-Paley projections of the pressure satisfy
\begin{align}
\co{ \nab P_k p } &\leq C 2^{(1 - 2 \a)k} \ctdcxa{v}^2 \label{ineq:pressBound}
\end{align}
\end{lem}
Using that $P_k p = \De^{-1} \nab \cdot \nab P_k p$, this bound is equivalent to the $C^{2\a}$ spatial regularity of the pressure when $\a \neq 1/2$ (see Corollary~\ref{cor:pressReg}).

%% Give more description!

%Note that, by writing $P_k p = \Delta^{-1} \mbox{div } \nab P_k p$ and using the Littlewood-Paley characterization, this bound is equivalent to $\co{ P_k p } \leq C 2^{-2 \a k} \ctdcxa{v}^2$, which implies that the pressure is of class $p \in L_t^\infty C_x^{2 \a}$ when $\a \neq 1/2$.

We begin the proof of these bounds by fixing notation, much of which is standard.

\subsection{Notation} \label{sec:notation1}

In what follows we will always regard distributions on the torus as being periodic distributions on the whole space $\R^n$.  All convolutions will refer to convolutions in the spatial variables at a fixed time unless otherwise stated.

The norm $\| X \|_{C^0}$ refers to the $C^0 = C^0_{t,x}$ norm of $X$ in both time and space.  For an operator $T$ acting on $C^0(\T^n)$, we will denote by $\| T \|$ the operator norm as a bounded mapping on $C^0(\T^n)$.  Very often our operators will be of a convolution form $T v(x) = \int_{\R^n} v(x+h) K(h) dh$, in which case $\| T \| \leq \| K \|_{L^1(\R^n)}$.  

The $\dot{C}^\b$ H\"{o}lder semi-norm of a tensor $f$ is given by $\| f\|_{\dot{C}^\b} = \sup_{h \in \R^d} \fr{|f(x+h) - f(x)|}{|h|^\b}$.  Given a function space $X$ measuring integrability and/or regularity in space, the mixed norm of a tensor $f(t,x)$, written $\| f \|_{L_t^p X}$ or $\| f \|_{L_t^p X_x}$, refers to the $L^p$ in time norm of the function $\| f(t, \cdot) \|_X$.

%Rather than use the standard Littlewood-Paley projection operators which are defined via the Fourier transform, we define our Littlewood-Paley projections using mollifying kernels with spherical symmetry and compact support.  Using these projections allows our estimates to easily have a localized character; in particular, the bound for $P_k p(t,x)$ at a given point $(t,x)$ depends only on the bounds for relative velocities $|v(t,y) - v(t,y')|$ at a distance $|y - x|, |y' - x| \leqc 2^{-k}$ from the point $(t,x)$.  
 
We recall here the basics of Littlewood-Paley theory to fix notation.  Let $\eta : \R^n \to \R$ be a radially symmetric smooth function such that ${\hat \eta}$ is supported in the unit ball of ${\hat \R}^n$, and such that ${\hat \eta}(\xi) = 1$ for $|\xi| \leq \fr{1}{2}$.  In particular, ${\hat \eta}(0) = \int_{\R^n} \eta(h) dh = 1$.  Now set $\eta_{\leq k}(h) = 2^{nk} \eta(2^k h)$ so that ${\hat \eta}_{\leq k} \in C_c^\infty(B_{2^k}({\hat \R}^n) ) $ also has integral $\int_{\R^n} \eta_{\leq k}(h) dh = 1$ and define
\begin{defn}[Littlewood-Paley Projections]
For any continuous, vector valued function $v : \T^n \to \R^m$, set
\begin{align}
P_{\leq k} v(t,x) &= \int_{\R^n} v(t,x+h) \eta_{\leq k}(h) dh \quad \quad (t,x) \in \R \times \T^n \\
P_k v(t,x) &= P_{\leq k} v(t,x) - P_{\leq k-1} v(t,x) = \eta_{k} \ast v (t,x) \\
\eta_{k} &= \eta_{\leq k} - \eta_{\leq k-1}
\end{align}
We also define $P_{- \infty} v = \fr{1}{|\T^n|} \int v(h) dh$ to be the average value of $v$
\begin{align}
 P_{- \infty} v = \fr{1}{|\T^n|} \int_{\T^n} v dx, \qquad 
 P_{(- \infty, k]} v = P_{\leq k} v - P_{- \infty} v
\end{align}
For any integers $k_1 < k_2$, we likewise define
\begin{align}
P_{[k_1, k_2]} &= P_{\leq k_2} - P_{\leq k_1} = \eta_{[k_1, k_2]} \ast 
\end{align}
\end{defn}
Note that the Fourier transform of ${\hat \eta}_{k}$ has support in a band $2^{-(k-1)} \leq |\xi| \leq 2^{(k+1)}$.  Our analysis will use often the following (rather delicate) properties of Littlewood-Paley projections, which give motivation for using the terminology ``projection''
\ali{
 P_{\leq k} = P_{\leq k+2} P_{\leq k}, \quad P_k = P_{[k-2, k+2]} P_k .
}
As we are working in the context of periodic functions, observe that there exists a greatest $k_0 \in \Z$ depending on $\T^n = \R^n / \Ga$ such that
\[ P_k v = 0, \quad \quad \forall~k \leq k_0(\T^n) \]
We then have a decomposition
\begin{prop}[Littlewood-Paley decomposition]
For any continuous vector-valued function $v$ on $\T^n$, we have
\begin{align}
v &= \sum_{k =  k_0(\T^n)}^\infty P_k v + P_{- \infty} v \label{eq:LPDecomp}
%&= P_{\leq 0} v + \sum_{k =  0}^\infty P_k v \label{eq:LPafterHigh}
\end{align}
where the summation (\ref{eq:LPDecomp}) is interpreted in the sense of distributions.
\end{prop}

\subsection{Preliminary Bounds on LP Pieces and the Pressure} \label{sec:boundsLPpieces}

Now, assuming that $(v^l, p)$ satisfy the Euler equations, let us study the equation obeyed by the Littlewood-Paley projections of $v^l$, which we write using the Einstein summation convention as
\begin{align}
\label{eq:euLPproj}
\left\{
\begin{aligned}
 \pr_t P_{\leq k}v^l + \pr_j(P_{\leq k}v^j P_{\leq k}v^l) + \pr^l P_{\leq k} p &= \pr_j R_{\leq k}^{jl}   \\
 \pr_j P_{\leq k} v^j &= 0
\end{aligned}
\right.
\\
R_{\leq k}^{jl} = P_{\leq k} v^j P_{\leq k} v^l &- P_{\leq k}(v^j v^l)
\end{align}
Using the fact that $P_{\leq k} v^j$ is divergence free, and subtracting (\ref{eq:euLPproj}) for $P_{\leq k+1} v^l$ and $P_{\leq k} v^l$ gives
\begin{align}
\label{eq:euLPpiece}
\left\{
\begin{aligned}
\pr_t P_{k+1}v^l + P_{\leq k}v^j \pr_j P_{k+1}v^l + P_{k+1} v^j \pr_j P_{\leq k + 1} v^l + \pr^l P_{k + 1} p &= \pr_j ( R_{\leq k + 1}^{jl}  - R_{\leq k}^{jl} )   \\
\pr_j P_{k + 1} v^j &= 0
\end{aligned}
\right.
\end{align}
From equations (\ref{eq:euLPproj}) and (\ref{eq:euLPpiece}) we will be able to deduce the following bounds
\begin{prop}[Bounds on Littlewood-Paley pieces, 1]\label{prop:firstLPpieceBds}
If $(v,p)$ solve Euler on $I \times \T^d$ for some open interval $I$ and $0 < \a \leq 1$, then their Littlewood Paley pieces satisfy the bounds
\begin{align}
\co{ P_k p } &\leq C 2^{- 2 \a k} \| v\|_{C_t {\dot C}_x^\a}^2 \label{ineq:pBound1}\\
\co{ \nab P_k p } &\leq C 2^{(1 - 2 \a) k} \| v\|_{C_t {\dot C}_x^\a}^2 \label{ineq:pBound2} \\
\co{ \nab^2 P_k p } &\leq C 2^{(2 - 2 \a) k} \| v\|_{C_t {\dot C}_x^\a}^2 \label{ineq:pBound3} \\
\co{ P_{k+1} v } &\leq C 2^{- \a k } \| v\|_{C_t {\dot C}_x^\a} \label{ineq:coBdPkv} \\
\co{ \nab P_{\leq k} v } + \co{ \nab P_{k + 1} v } &\leq C 2^{(1 - \a) k } \| v\|_{C_t {\dot C}_x^\a} \label{ineq:nabPkv} \\
\co{ (\pr_t + P_{\leq k} v \cdot \nab) P_{k+1} v } &\leq C 2^{(1 - 2\a)k} \| v\|_{C_t {\dot C}_x^\a}^2 \label{ineq:mtDvBound1} \\
\co{ \nab (\pr_t + P_{\leq k} v \cdot \nab) P_{k+1} v } &\leq C 2^{(2 - 2 \a)k} \| v \|_{C_t {\dot C}_x^\a}^2  \label{ineq:mtDvBound2}
\end{align}
where $\nab$ represents any spatial derivative.
\end{prop}
\begin{proof}[Proof of Proposition (\ref{prop:firstLPpieceBds})]
First note that (\ref{ineq:coBdPkv}), being a standard estimate, is clear from the expression
\ali{
 P_k v^j(x) &= \int v^j(x + h) (\eta_{\leq k}(h) - \eta_{\leq k - 1}(h) ) dh \\
&= \int ( v^j(x+h) - v^j(x + 2 h) ) \eta_{\leq k}(h) dh \label{eq:PkAsRelVelocAvg} \\
\co{ P_k v^j } &\leq \ctdcxa{v}\int |h|^\a \eta_{\leq k}(h) dh
}
The estimate (\ref{ineq:nabPkv}), which is also standard, follows from the formula
\ali{
 \pr_i P_{\leq k} v(x) &= \int ( v(x + h) - v(x) ) \pr_i \eta_{\leq k}(h) dh  \label{eq:derivative}
}
which uses the identity $\int_{\R^n} \pr_i \eta_{\leq k}(h) dh = 0$.  The same estimate follows as well for the individual projections $P_{k+1} v = P_{\leq k+1} v - P_{\leq k} v$.

We recall the well-known calculations above to emphasize that the Littlewood-Paley projections we are employing are simply particular examples of averages\footnote{The expression in line \eqref{eq:PkAsRelVelocAvg} can be viewed as a type of average in the sense that the weight function satisfies $\int \eta_{\leq k}(h) dh = 1$, even though the function $\eta_{\leq k}$ takes on both positive and negative values for the Littlewood-Paley projections we employ here.} or linear combinations of velocity differences.  In particular, line \eqref{eq:PkAsRelVelocAvg}, which expresses $P_k v$ as an average of relative velocities concentrated at distance $|h| \leq 2^{-k}$, shows how bounds on Littlewood-Paley projections can be viewed as estimates on the average size of velocity differences at scale $2^{-k}$.  This calculation also %provides a starting point for comparing statements about statistically averaged velocity differences to statements within the mathematical formalism of Littlewood-Paley theory, and 
explains the sense in which the scaling law \eqref{law:oneThird} relates to a Besov-type regularity for the velocity field.  

Now, observe that, since $v$ and all of its Littlewood-Paley projections are divergence free, one can routinely justify the following computation by passing from a smooth approximation
\begin{align}
P_k p &= \pr_j \pr_l \De^{-1} P_k (v^j v^l) \label{eq:nonLinForPkp} \\
&= \pr_j \pr_l \De^{-1} P_k \left[ ( v^j - P_{\leq k} v^j) (v^l - P_{\leq k} v^l) + ( v^j - P_{\leq k} v^j) P_{\leq k} v^l \right.  \notag \\
&\left.+ P_{\leq k} v^j ( v^l- P_{\leq k} v^l) + P_{\leq k} v^j P_{\leq k} v^l \right] \\
&= \pr_j \pr_l \De^{-1} P_k \left[ ( v^j - P_{\leq k} v^j) (v^l - P_{\leq k} v^l) \right]  + \pr_l \De^{-1} P_k \left[ ( v^j - P_{\leq k} v^j) \pr_j P_{\leq k} v^l \right] \notag \\
&+\pr_j  \De^{-1} P_k \left[ \pr_l P_{\leq k} v^j ( v^l- P_{\leq k} v^l) \right] +\De^{-1} P_k \left[ \pr_l P_{\leq k} v^j \pr_j P_{\leq k} v^l \right]
\end{align}
The bounds  (\ref{ineq:pBound1}), (\ref{ineq:pBound2}) and (\ref{ineq:pBound3}) follow when we apply the bounds
\ali{
\co{( v - P_{\leq k} v) } &\leq 2^{-\a k } \| v\|_{C_t {\dot C}_x^\a} \\
\co{\nab P_{\leq k} v } &\leq 2^{(1 - \a)k} \| v\|_{C_t {\dot C}_x^\a}
}
and we observe, by scaling considerations, that
\[ \| \nab^D \De^{-1} P_k \| \leq C_D 2^{(D-2)k} \]
as a bounded operator on $C^0$, since the right hand side bounds the $L^1$ norm of the kernel.

To obtain (\ref{ineq:mtDvBound1}) and (\ref{ineq:mtDvBound2}), we use equation (\ref{eq:euLPpiece}), the bounds (\ref{ineq:pBound2}) and (\ref{ineq:pBound3}) for the pressure, and the Constantin, E, Titi commutator estimate \cite{CET}, which we state in the form
\ali{
\co{ \nab^D R_{\leq k}^{jl} } &\leq C_D 2^{(D - 2 \a) k}  \| v\|_{C_t {\dot C}_x^\a}^2 \label{eq:theCETbound}
}
The statement (\ref{eq:theCETbound}), which includes bounds on higher derivatives, is proven in \cite{deLSzeC1iso}.  Later on we will give a different proof of (\ref{eq:theCETbound}) which will enable us to prove higher regularity in time.

\end{proof}

The estimate \eqref{ineq:pBound1} immediately implies the $C^{2\a}$ spatial regularity of the pressure
\begin{cor}\label{cor:pressReg} If $0 < \a < 1/2$, then $p \in L_t^\infty C_x^{2\a}$, while if $1/2 < \a < 1$, $\nab p \in L_t^\infty C_x^{2\a - 1}$.
\end{cor}
\begin{proof}  For $\a < 1/2$ and $\b = 2 \a$, the corollary follows from \eqref{ineq:pBound1} and the standard Littlewood-Paley characterization of the $C^\b$ semi-norm, valid for $0 < \b < 1$:
\ali{
\| f \|_{\dot{C}^\b(\T^n)} \sim \sup_{k \geq 0} 2^{\b k} \| P_k f \|_{L^\infty(\T^n)}. \label{LPchar:ca}
}
For $1/2 < \a < 1$, we use that $\co{ \nab P_k p } \leq C 2^k \co{P_k p}$ and \eqref{LPchar:ca} with $\b = 2 \a - 1$ and $f = \nab p$.  A proof of the standard estimate \eqref{LPchar:ca} in the case of $\T^n$ can be found, for example, in \cite[Appendix]{isettOnsag}.
\end{proof}

Let us emphasize again that the bounds in Proposition \ref{prop:firstLPpieceBds} are all consistent with the dimensional analysis of the Euler equations and remain invariant under Galilean transformations.  That is, the velocity carries units $\fr{ \tx{length} }{ \tx{time} }$ and the pressure carries units $\fr{ \tx{length}^2 }{ \tx{time}^2 }$ while the factor $2^{k}$ has the units of an inverse length scale.  Furthermore the bounds involve only velocity differences or derivatives, giving them a Galilean invariance.  In contrast, the bound available for the stationary time derivative
\ali{
\co{ \pr_t P_k v } &\leq \co{(\pr_t + P_{\leq k - 1} v \cdot \nab) P_k v } + \co{ P_{\leq k - 1} v \cdot \nab P_k v } \\
&\leq C (2^{(1 - 2\a)k} \ctdcxa{v}^2 + 2^{(1-\a)k} \co{v} \ctdcxa{v} ) \\
&\leq C_{\T^n} 2^{(1 - \a)k} \ctdcxa{v} \| v \|_{C_t C_x^\a}  \label{bd:prtPkv}
}
(which can also be proven directly without the commutator estimate) is weaker and clearly fails to be Galilean invariant.\footnote{Here it seems natural to define the norm $\| v \|_{C_t C_x^\a} = \| v \|_{C^0} + 2^{-\a \kotn} \ctdcxa{v} $ with $\kotn$ as defined in Proposition (\ref{prop:matDvLowFreqs}) in order to obtain a dimensionally consistent estimate in \eqref{bd:prtPkv} with a constant independent of $\T^n$. }  Nonetheless, (\ref{bd:prtPkv}) is enough to imply the H\"{o}lder regularity in time of $v \in C_{t,x}^\a$ asserted in Theorem \ref{thm:timeRegBounds}.

As a prelude to establishing Theorem~\ref{thm:timeRegBounds}, we first observe that the regularity $v \in C_{t,x}^\b$ for all $\b < \a$ can be obtained from the following interpolation argument.
\begin{align}
\| P_k v \|_{C_{t,x}^\b} &\leq C \co{ P_k v }^{1-\b} \co{ \nab_{t,x} P_k v }^{\b} \\
&\leq C_{\T^n} \left(2^{- \a k } \| v \|_{C_t {\dot C}_x^\a}\right)^{1-\b}\left(2^{(1 - \a)k} \ctdcxa{v}(1 + \| v \|_{C_t C_x^\a}) \right)^\b \\
&\leq C_{\T^n} 2^{(\b - \a) k } \ctdcxa{v} (1 + \| v \|_{C_t C_x^\a})^\b \label{ineq:gotTheHolder}
\end{align}
%The same interpolation argument using (\ref{ineq:pBound1}) and (\ref{ineq:pBound2}) implies that $p \in L_t^\infty C_x^{2\b}$ for all $\b < \a \leq 1$.  
Theorem~\ref{thm:timeRegBounds} states also that $v \in C_{t,x}^\a$, which will be proven in Section~\ref{sec:endpointCase} below.  

For later use, we record the following estimates, which are related to the bounds in Proposition~\ref{prop:firstLPpieceBds}
\begin{prop} \label{prop:matDvLowFreqs} Under the assumptions of Proposition (\ref{prop:firstLPpieceBds}) and $0 < \a < 1$, we have
\ali{
\co{ \nab^{2 + D} P_{\leq k} p } &\leq C_{D} 2^{(D + 2 - 2 \a)k} \| v \|_{C_t {\dot C}_x^\a}^2 \label{ineq:lowFrePlqkp1} \\
\co{ \nab^D (\pr_t + P_{\leq k} v \cdot \nab) \nab P_{\leq k} v } &\leq C_{D} 2^{(D + 2 - 2 \a)k} \| v \|_{C_t {\dot C}_x^\a}^2  \label{ineq:mtDvBound3}
}
If $\a = 1$, we have instead
\ali{
\co{ \nab^{2 + D} P_{\leq k} p } &\leq C_{D} (1 + |k - k_0(\T^n)|) 2^{Dk} \| v \|_{C_t {\dot C}_x^1}^2 \label{ineq:lowFrePlqkp1al1} \\
\co{ \nab^D (\pr_t + P_{\leq k} v \cdot \nab) \nab P_{\leq k} v } &\leq C_{D} (1 + |k - k_0(\T^n)|) 2^{Dk} \| v \|_{C_t {\dot C}_x^1}^2  \label{ineq:mtDvBound3al1}
}
The factor $(1 + |k - k_0(\T^n)|)$ can be omitted for $D > 0$.
\end{prop}
Here we recall that $k_0(\T^n)$ is an integer such that $2^{k_0(\T^n)}$ is comparable to the lowest frequency on the dual of the torus $\T^n = \R^n / \Ga$, the latter of which can be regarded as a characteristic inverse length for the flow.  In particular, the difference $|k - k_0(\T^n)|$, being the logarithm of a ratio of inverse lengths, is dimensionless.
\begin{proof}
The bound (\ref{ineq:lowFrePlqkp1}) for $0 < \a < 1$ and $D = 0$ follows by estimating
\ALI{
 \co{ \nab^2 P_{\leq k} p } &\leq \sum_{I = k_0(\T^n)}^k \co{ \nab^2 P_I p } 
\leq C \sum_{I = -\infty}^k 2^{(2 - 2 \a)I} \ctdcxa{v}^2
}
In the case $\a = 1 $ we obtain instead the bound
\ALI{ 
 \co{ \nab^2 P_{\leq k} p }  &\leq  C (1 + |k - k_0(\T^n)|) \| v \|_{C_t {\dot C}_x^\a}^2  
}
%The loss of the factor $|k - k_0(\T^n)|$ occurs here only in the case $\a = 1$.  
The proof for $D > 0$ is identical, but does not involve the loss of a logarithmic factor $|k - k_0(\T^n)|$.

The bound (\ref{ineq:mtDvBound3}) can now be obtained from equation (\ref{eq:euLPproj}) by using the bound (\ref{eq:theCETbound}) together with (\ref{ineq:lowFrePlqkp1}) and the basic estimates for $P_{k+1} v \cdot \nab P_{\leq k+1} v$.
\end{proof}

\subsection{On the endpoint regularity} \label{sec:endpointCase}

The argument of line \eqref{ineq:gotTheHolder} shows that the sequence $P_{\leq k} v$ is Cauchy in the space $C_{t,x}^\b$ for $\b < \a < 1$; thus $v \in C_{t,x}^\b$ for $\b < \a$.  Here we show that a more careful argument (adapting the technique in the Littlewood-Paley theory characterization of H\"{o}lder spaces) establishes the endpoint regularity $v \in C_{t,x}^\a$.  

Letting $\De t \in \R$ be fixed, the idea is to estimate $|v(t + \De t, x) - v(t,x)|$ by summing the bound
\ali{
| P_k v(t + \De t, x) - P_k v(t,x) | &\leq \min \{ \co{ \pr_t P_k v } |\De t|, 2 \co{ P_k v } \}
}
The bound $\co{ \pr_t P_k v } |\De t|$ is more useful for low frequencies, which vary less rapidly, while high frequencies have smaller amplitude and the $C^0$ estimate is more useful in this case.  Using \eqref{eq:LPDecomp}, we proceed to estimate $|v(t + \De t, x) - v(t,x)|$ by decomposing the velocity difference into
\ALI{
|v(t + \De t, x) - v(t,x)| &\leq \sum_{k = \kotn}^K \co{ \pr_t P_k v } |\De t| + 2 \sum_{k > K} \co{P_k v} \\
&\leq C \left( \sum_{k = - \infty}^K |\De t| 2^{(1 - \a) k} \| v \|_{C_t C_x^\a}\ctdcxa{v} + \sum_{ k > K } 2^{- \a k} \ctdcxa{v} \right) \\
&\leq C ( 2^{(1 - \a) K } \| v \|_{C_t C_x^\a} |\De t| + 2^{- \a K} ) \ctdcxa{v}
}
Choosing $K$ so that both terms are roughly equal gives
\ali{
|v(t + \De t, x) - v(t,x)| &\leq C |\De t|^\a \| v \|_{C_t C_x^\a}^{\a} \ctdcxa{v} \label{eq:endpointHolder}
}
The estimate \eqref{eq:endpointHolder} is dimensionally correct with a universal constant provided we normalize the inhomogeneous norm $\| v \|_{C_t C_x^\a}$ to have dimensions of velocity as in the footnote after \eqref{bd:prtPkv}.  The same idea applied to the difference $|v(t + \De t, x + \De x) - v(t,x)|$ establishes H\"{o}lder regularity in space and time.  

Several endpoint cases for the regularity results in this paper will be established using the above argument, but there are some exceptions.  Specifically, we will see that there are logarithmic losses in the estimates for the pressure that prohibit us from obtaining $p \in C_{t,x}^{2 \a}$ regularity in time, and the time regularity of the energy profile in the case $\a = 1/3$ is slightly more subtle.

\section{Regularity of the Energy Profile} \label{sec:timeRegEnergy}

In this Section, we establish Theorem~\ref{thm:energyReg} on the H\"{o}lder regularity for the energy profile as a function of time, and we conclude with some discussion of this theorem's relationship to the discussion of turbulence and anomalous dissipation in the introduction.  %We only give the proof in the periodic setting, however the method can probably be adjusted to give similar results in the whole space if one is willing to use $L^p$-based norms.

Let us define
\begin{defn} The {\bf energy increment from frequency $2^{k+1}$} is defined as
\begin{align}
\de e_{(k)}(t) &= \int_{\T^n} ( | P_{\leq k+1}  v |^2 - | P_{\leq k}  v |^2 )dx
\end{align}
\end{defn}

Then we have
\[ \int_{\T^n} |v|^2 dx = \int_{\T^n} | P_{- \infty} v |^2 dx + \sum_{k =  k(\T^n)}^\infty \de e_{(k)}(t) \]
where the first term, which is essentially the square of the total momentum, is conserved in time.

\begin{comment}
It is clear from
\begin{align}
 \pr_t P_{\leq k}v^l + \pr_j(P_{\leq k}v^j P_{\leq k}v^l) + \pr^l P_{\leq k} p &= \pr_j R_{\leq k}^{jl}  
\end{align}
that $\int | P_{\leq k} v |^2 dx$ is $C^1$ in time for each $k$.  
\end{comment}

For exponents $\a < 1/3$, the H\"{o}lder regularity of $\int |v|^2(t,x) dx$ in time follows using the argument of Section \eqref{sec:endpointCase} from the following estimates
\begin{prop}
\ali{
\| \de e_{(k)}(t) \|_{C^0_t} &\leq C_{\T^n} 2^{-2 \a k} \ctdcxa{v}^2 \label{eq:enEst1}\\
\| \fr{d}{dt} \de e_{(k)}(t) \|_{C^0_t} &\leq C_{\T^n} 2^{(1 - 3 \a) k } \ctdcxa{v}^3 \label{eq:enEst2}
}
\end{prop}
We begin by proving (\ref{eq:enEst1}).
\begin{proof}[Proof of (\ref{eq:enEst1})]
The first term in
\ali{
\de e_{(k)}(t) &= \int_{\T^n} |P_{k+1} v|^2 dx + 2 \int_{\T^n} P_{k+1} v \cdot P_{\leq k} v dx \label{eq:eninc1} \\
&= \de e_{(k), 1} + 2 \de e_{(k), 2}
}
is easily bounded by
\ali{
\de e_{(k), 1} &\leq |\T^n| \co{ P_{k+1} v }^2 \\
&\leq C_{\T^n} 2^{-2 \a k} \ctdcxa{v}^2
}
For $\de e_{(k), 2}$, we have no control over the size of $P_{\leq k} v$, so we exploit the fact that the interaction between $P_{k+1} v$ and $P_{\leq k} v$ takes place on the common frequency $\sim 2^k$.
\ali{
\de e_{(k),2} &= \int_{\T^n} P_{k+1} v \cdot P_{[k-3, k]} v dx \label{eq:enIncLoc}
}
at which point (\ref{eq:enEst1}) is clear.

A natural approach to estimating $\fr{d}{dt} \de e_{(k)}(t)$ would be to take the time derivative of \eqref{eq:eninc1} and to observe that (since $P_{\leq k} v$ is divergence free) we have
\begin{align}
\fr{d}{dt} \de e_{(k)}(t) &= \int_{\T^n} (\pr_t + P_{\leq k} v \cdot \nab) [|P_{k+1} v|^2 + P_{k+1} v \cdot P_{\leq k} v ] dx
\end{align}
Ideally, the cost for a material derivative should always be $2^{(1 - \a) k} \ctdcxa{v}$, which would give the bound (\ref{eq:enEst2}).  However, one quickly sees that there is no control over the size of the term $ P_{\leq k} v $ that remains when the material derivative hits $P_{k+1} v$ (which is related to the fact that $ P_{\leq k} v $ is an average of absolute velocities rather than relative velocities).  Thus, before differentiating it is useful to first express this interaction term in terms of relative velocities, as in \eqref{eq:enIncLoc}% the proof of (\ref{eq:enEst1}).  

On the other hand, when we express the derivative of the energy increment as
 \begin{align}
\fr{d}{dt} \de e_{(k)}(t) &= \int_{\T^n} (\pr_t + P_{\leq k} v \cdot \nab) [|P_{k+1} v|^2 + P_{k+1} v \cdot P_{[k-3, k]} v ] dx
\end{align}
it follows from (\ref{ineq:mtDvBound1}) and (\ref{ineq:coBdPkv}) that the material derivative costs $2^{(1 - \a) k} \ctdcxa{v}$ in the estimate for each term, which is the desired bound and gives (\ref{eq:enEst2}).

We remark here that the divergence free property of $P_{\leq k} v$ is not important in this estimate.  Namely, even without incompressibility, the other term that could arise would take the form
\[ \int \pr_i  P_{\leq k} v^i [ |P_{k+1} v|^2 + P_{k+1} v \cdot P_{[k-3, k]} v ] dx  \]
after an integration by parts, and the cost of introducing the term $\tx{ div } P_{\leq k} v$ is exactly the factor
\[\| \pr_i  P_{\leq k} v^i \|_{C^0} \leqc 2^{(1 - \a)k} \ctdcxa{v} \]
that we desire.  

For the case $\a = 1/3$, we recall that \cite{CET} proves the following bounds at any fixed time:
\ALI{
\| \nab P_{\leq k} v \|_{L^3} \leq C 2^{(1-\a)k} \|v \|_{\dot{B}_{3,\infty}^{\a}}, \qquad  \| R_{\leq k} \|_{L^{3/2}} \leq C 2^{-2\a k} \| v \|_{\dot{B}_{3,\infty}^{\a}}^2, \qquad 0 < \a < 1.
}
Using the embedding of $L^\infty$ into $L^3$ on the torus and H\"{o}lder with $\fr{1}{3} + \fr{2}{3} = 1$, we have that
\ALI{
\fr{d}{dt} e_{\leq k}(t) \equiv  \fr{d}{dt} \int_{\T^n} |P_{\leq k} v|^2(t,x) dx  &= \int_{\T^n} (\pr_t + P_{\leq k} v \cdot \nab) |P_{\leq k} v|^2 dx = - \int_{\T^n} \pr_j (P_{\leq k} v)_l R_{\leq k}^{jl} dx
}
is bounded uniformly in $k$ and $t$ by
\ali{
\left|\fr{d}{dt} e_{\leq k}(t) \right| &\leq C 2^{(1 - 3 \a) k} \| v\|_{L_t^\infty \dot{B}_{3,\infty}^{\a}}^3 \leq C_{\T^n} 2^{(1 - 3 \a) k} \ctdcxa{v}^3.
}
Setting $\a = 1/3$ and using the weak convergence $\fr{d}{dt} e_{\leq k}(t) \rightharpoonup \fr{d}{dt} e(t)$ in $\DD'(I)$, we have that $\fr{d}{dt} e(t) \in L^\infty$, so that the energy profile is Lipschitz for $v \in L_t^\infty B_{3,\infty}^{1/3}$ on either $\T^n$ or $\R^n$, and also for $v \in L_t^\infty C_x^{1/3}$ if $v$ is on $I \times \T^n$.
% in fact the same holds for $v \in L_t^\inft B_{3,\infty}^{1/3}$ whether $v$ is on $\T^n$ or $\R^n$.  

We conclude the proof by addressing the case where the velocity field has lower integrability in time and Onsager critical regularity of Besov-type, i.e. $v \in L_t^p B_{3, \infty}^{1/3}$.  % and 
%some comments related to anomalous dissipation.  
First, observe that the same argument for $p \geq 3$ shows that $\| \fr{d}{dt} e_{\leq k} \|_{L^{p/3}_t}$ remains uniformly bounded in $k$ when $v \in L_t^p B_{3, \infty}^{1/3}$.  Assume $p > 3$ and let $1 \leq r < \infty$ satisfy $\fr{1}{r} + \fr{3}{p} = 1$.  Using this bound, the weak convergence $\fr{d}{dt} e_{\leq k} \rightharpoonup \fr{d}{dt} e_{\leq k}$ in $\DD'(I)$ and the density of test functions in $L^r$, we obtain that $\fr{d}{dt} e$ belongs to the dual of $L^r(I)$, which is $L^{p/3}(I)$.  That is, $e(t) \in W^{1,p/3}$.

For $p = 3$ this argument shows that $\fr{d}{dt} e$ is in the dual of $C^0$ and is thus a finite, signed Radon measure by the Riesz Representation Theorem.  To obtain that $\fr{d}{dt} e \in L^1(I)$ is absolutely continuous, note that for any open set $O \subseteq I \subseteq \R$, we can use duality to bound the total variation on $O$ by 
\ALI{
\left\| \fr{d}{dt} e(t) \right\|_{TV(O)} &\leq C \| v \|_{L_t^3(O; \dot{B}_{3,\infty}^{1/3})}^3 = C \int_O \| v(t) \|_{\dot{B}_{3,\infty}^{1/3}}^3 dt
}
since the left hand side can be realized as a supremum of integration against test functions supported in $O$ with sup norm bounded by $1$.  For any given $\ep > 0$, since $\| v(t) \|_{\dot{B}_{3,\infty}^{1/3}}^3 \in L^1(I)$, we can choose $\de > 0$ so that the rightmost integral is less then $\ep$ whenever the measure of $O$ is less than $\de$.  That is, this inequality shows that $\fr{d}{dt} e(t)$ must be absolutely continuous with respect to Lebesgue measure and hence that $e(t) \in W^{1,1}$ as desired.  This statement concludes the proof of Theorem~\ref{thm:energyReg}.
\end{proof}
Let us now highlight the case $p = \infty$ where the velocity field is of class $L_t^\infty \dot{B}_{3, \infty}^{1/3}$ and the stability of anomalous dissipation may be possible.  In this case, the energy profile is Lipschitz and the bound
\ali{
 \left\| \fr{d}{dt} e(t) \right\|_{L_t^\infty} &\leq C \| v \|_{L_t^\infty \dot{B}_{3, \infty}^{1/3}}^3 \label{eq:energyVariationIneq}
}
holds with a constant $C$ that is universal.  Note that inequality \eqref{eq:energyVariationIneq} bears a close formal resemblance to the $p = 3$ case of \eqref{law:oneThird}.  It is in the above class that the quality of energy dissipation for Euler flows could conceivably be stable under perturbation, as a simple generalization of the above proof shows that the energy variation $\fr{d}{dt} e(t)$ varies continuously in $L^\infty$ as $v$ varies among Euler flows in the $L_t^\infty B_{3, \infty}^{1/3}$ topology.  Consequently an inequality of the form $\fr{d}{dt} e(t) \leq - \varep < 0$ is an open condition in this space.    %, making the quality of strict energy dissipation an open condition.  
In contrast, we conjecture that energy dissipation or more generally the quality of having an energy profile with bounded variation should be unstable and nongeneric for solutions with lesser regularity (see Section~\ref{sec:conclusion} below).

\section{Material derivative estimates for the pressure increments and commutators} \label{sec:timeRegPress}

One might hope to prove that the pressure also has H\"{o}lder regularity $p \in C_{t,x}^{2a}$ in time by establishing the following bound
\begin{ques}  For $0 < \a \leq 1$, is there an estimate of the form
\begin{align}
\co{ (\pr_t + P_{\leq k} v \cdot \nab) P_k p } &\leq C 2^{(1 - 3 \a) k} \| v \|_{C_t {\dot C}_x^\a}^3 ? \quad \quad  \label{eq:wantBdForPressure}
\end{align}
\end{ques}
Such a bound would be consistent with (\ref{ineq:pBound1}), (\ref{ineq:coBdPkv}) and (\ref{ineq:mtDvBound1}), since the cost of a material derivative should be $2^{(1 - \a)k} \|v \|_{C_t {\dot C}_x^\a}$ (which has the dimensions of an inverse time).  However, it is not clear whether (\ref{eq:wantBdForPressure}) would be true when $\a \leq 1/3$.  In this case, the high frequencies $2^I > 2^k$ that contribute to the pressure through the nonlinearity move along the more violent flow of $(\pr_t + P_{\leq I} v \cdot \nab)$ rather than that of $(\pr_t + P_{\leq k} v \cdot \nab)$.  After using the Littlewood-Paley calculus to expand
\ali{
P_k [ ( v^j - P_{\leq k} v^j) (v^l - P_{\leq k} v^l) ] &\approx \sum_{I > k} P_k [ P_I v^j P_I v^l ] \label{eq:paradiff}
}
it appears that the most optimistic attempt to bound the derivative \[ (\pr_t + P_{\leq k} v \cdot \nab) = (\pr_t + P_{\leq I} v \cdot \nab) - P_{[k, I] } v \cdot \nab \]
of (\ref{eq:paradiff}) gives rise to a divergence when $\a \leq 1/3$.  

%%%%  Add a discussion of \a \leq 1/3  

We circumvent around this difficulty for $\a \leq 1/3$ by decomposing the pressure
\[ p = \sum_{k =  k_0(\T^n)}^\infty \de p_{(k)} \]
into increments $\de p_{(k)}$ that only involve interactions of velocity components with frequencies below $2^{k}$.  %The philosophy of assembling the solution one frequency shell at a time is also the guiding philosophy in the construction of weak solutions by convex integration, and the analysis here closely mirrors the estimates in the construction \cite{isett}.  This is also essentially the same method used to prove regularity for the energy in Section (\ref{sec:timeRegEnergy}) above.

To define our pressure increments, we start by defining
\ali{
 p_{(k)} &= \De^{-1} \pr_j \pr_l P_{(- \infty, k]}( P_{(- \infty, k]} v^j P_{(- \infty, k]} v^\ell ) \\
 &= \De^{-1} \pr_j \pr_l P_{\leq k}( P_{\leq k} v^j P_{\leq k} v^\ell )
}
Note in particular that $p = \lim_{k \to \infty} p_{(k)}$ as a distribution.
\begin{defn}\label{def:pressInc} We define {\bf the pressure increment from frequency $2^{k}$} to be
\begin{align}
\de p_{(k)} = p_{(k+1)} - p_{(k)}
\end{align}
\end{defn}
The decomposition above is motivated by analogy with the construction and estimation of the pressure that arises in the construction of solutions in \cite{isett}.  The underlying philosophy there is to build the solution starting with the low frequencies and incrementally adding in the higher frequencies.

Since we can write the pressure as
\[ p = \sum_{k =  k_0(\T^n)}^\infty \de p_{(k)} \]
the $C_{t,x}^\b$ H\"{o}lder regularity for $p$ when $\b \leq 1/2$ follows (following the argument of line \eqref{ineq:gotTheHolder}) by interpolating the following bounds
\begin{prop}[Bounds for the pressure increments]\label{prop:pressIncBds1} For every $0 < \a \leq 1$, there is a constant $C = C_\a$ such that
\begin{align}
\co{ \de p_{(k)} } &\leq C_\a (1 + |k - k_0(\T^n)|) 2^{-2 \a k } \ctdcxa{v}^2 \label{bd:pressIncCo} \\
%\co{ \nab \de p_{(k)} } &\leq C_\a (1 + |k - k_0(\T^n)|) 2^{(1 -2 \a) k } \ctdcxa{v}^2 \label{bd:nabPressIncCo} \\
\co{ \nab^D \de p_{(k)} } &\leq C_{\a,D} (1 + |k - k_0(\T^n)|) 2^{(D -2 \a) k } \ctdcxa{v}^2 \label{bd:nab2PressIncCo} \\
% \co{ \pr_t \de p_{(k)} } &\leq C_{\a,\T^n} (1 + |k - k_0(\T^n)|) 2^{(1 - 2 \a) k } \| v \|_{C_t C_x^\a}^3 \label{bd:dtPressInc} \\
%\co{ \nab \pr_t \de p_{(k)} } &\leq C_{\a,\T^n} (1 + |k - k_0(\T^n)|) 2^{(2 - 2 \a) k } \| v \|_{C_t C_x^\a}^3
\co{( \pr_t + P_{\leq k} v \cdot \nab) \de p_{(k)}} &\leq C_\a (1 + |k - k_0(\T^n)|) 2^{(1 - 3\a) k } \ctdcxa{v}^3 \label{bd:matDvPress} \\
\co{\nab^D ( \pr_t + P_{\leq k} v \cdot \nab) \de p_{(k)}} &\leq C_{\a,D} (1 + |k - k_0(\T^n)|) 2^{(D + 1 - 3\a) k } \ctdcxa{v}^3 \label{bd:nabMatDvPress}
\end{align}
%where the third and fourth constants depend also on $\T^n$.
\end{prop}
%The number $2^{k_0(\T^n)}$ is basically the absolute value of the lowest frequency in the dual of the torus, so the expression $k - k_0(\T^n)$ is dimensionless, being the logarithm of a ratio of inverse lengths.  
The factors of $(1 + |k - k_0(\T^n)|)$ make no difference in summing the series
\ALI{
 \sum_{k = k_0(\T^n)}^\infty \| \de p_{(k)} \|_{C_{t,x}^\b} &\leq C \sum_{k = k_0(\T^n)}^\infty \| \de p_{(k)} \|_{C^0}^{(1-\b)} \| \nab_{t,x} \de p_{(k)} \|_{C^0}^{\b}  \\
&\leq C_{\T^n} \sum_{k = k_0(\T^n)}^\infty (1 + |k - k_0(\T^n)|) 2^{(\b - 2 \a)k} \ctdcxa{v}^2(1 + \| v \|_{C_tC_x^\a})^\b
}
because the convergence for $\b < 2 \a$, $\b \leq 1$ is exponential, so in particular Proposition (\ref{prop:pressIncBds1}) implies the $\a \leq 1/2$ case of Theorem (\ref{thm:timeRegBounds}).  The logarithmic loss in the estimates prevents us from obtaining the endpoint regularity $p \in C_{t,x}^{2 \a}$, but we remark that the argument of Section \ref{sec:endpointCase} can be adapted to give some endpoint type regularity in time with a logarithmic loss.
 
We now begin the proof of Proposition~\ref{prop:pressIncBds1} by establishing the bounds (\ref{bd:pressIncCo})-(\ref{bd:nab2PressIncCo}).%, which hold with a completely universal constant $C$.
\begin{proof}[Proof of (\ref{bd:pressIncCo})-(\ref{bd:nab2PressIncCo})]
We start by expressing
\ali{
\de p_{(k)} &= \De^{-1} \pr_j \pr_l \left[ P_{\leq  k+1} ( P_{\leq  k+1} v^j P_{\leq  k+1} v^l ) - P_{\leq  k} ( P_{\leq  k} v^j P_{\leq  k} v^l ) \right] \notag \\
&= \De^{-1} \pr_j \pr_l P_{k+1} [ P_{\leq  k+1} v^j P_{\leq  k+1} v^l ] \notag \\
&+ \De^{-1} \pr_j \pr_l P_{\leq  k} \left[ P_{\leq  k+1} v^j P_{\leq  k+1} v^l - P_{\leq  k} v^j P_{\leq  k} v^l  \right] \notag \\
&= \De^{-1} \pr_j \pr_l P_{k+1} [ P_{\leq  k+1} v^j P_{\leq  k+1} v^l ] + \De^{-1} \pr_j \pr_l P_{\leq k} [ P_{k+1} v^j P_{k+1} v^l ] \notag \\
&+  \De^{-1} \pr_j \pr_l P_{\leq k} \left[ P_{\leq k} v^j P_{k+1} v^l + P_{k+1} v^j P_{\leq k} v^l \right] \notag \\
&= \de p_{(k), 1} + \de p_{(k), 2} + \de p_{(k), 3} \\
\de p_{(k), 1} &= \De^{-1} \pr_j \pr_l P_{k+1} [ P_{\leq  k+1} v^j P_{\leq  k+1} v^l ] \\
\de p_{(k), 2} &= \De^{-1} \pr_j \pr_l P_{\leq k} [ P_{k+1} v^j P_{k+1} v^l ] \\
\de p_{(k), 3} &= \De^{-1} \pr_j \pr_l P_{\leq k} \left[ P_{\leq k} v^j P_{k+1} v^l + P_{k+1} v^j P_{\leq k} v^l \right]
}
We then use the basic properties of Littlewood-Paley pieces to  further decompose $\de p_{(k),3}$ into High-High and High-Low interactions.
\ali{
\de p_{(k), 3} &= \De^{-1} \pr_j \pr_l P_{\leq k} \left[ P_{\leq k} v^j P_{k+1} v^l + P_{k+1} v^j P_{\leq k} v^l \right] \\
&= \de p_{(k), 3HH} + \de p_{(k), 3HL} \\
\de p_{(k), 3HH} &= \De^{-1} \pr_j \pr_l P_{\leq k} \left[ P_{[ k - 3, k]} v^j P_{k+1} v^l + P_{k+1} v^j P_{[k-3, k]} v^l \right] \label{eq:highHighTerm} \\
\de p_{(k), 3HL} &= \De^{-1} \pr_j \pr_l P_{[k-3, k]} \left[ P_{\leq k-4} v^j  P_{k+1} v^l + P_{k+1} v^j  P_{\leq k-4} v^l \right]
}
Here we have taken advantage of the representation of the product as a convolution in frequency space, which ensures a lower bound on the frequency support of the product $P_{\leq k-4} v^j  P_{k+1} v^l$.

Finally, we use the fact that the Littlewood-Paley pieces of $v^j$ are all divergence free to write
\ali{
\de p_{(k), 3HL} &= \De^{-1} \pr_j P_{[k-3, k]} [ \pr_l P_{\leq k-4} v^j  P_{k+1} v^l ] + \De^{-1}  \pr_l P_{[k-3, k]} [ P_{k+1} v^j  \pr_j P_{\leq k-4} v^l ] \\
&=  \De^{-1} \pr_j P_{[k-3, k]} [ \pr_l P_{\leq k-4} v^j  P_{k+1} v^l ] + \De^{-1}  \pr_l P_{[k-3, k]} [ P_{k+1} v^j  \pr_j P_{\leq k-4} v^l ]
}
and
\ali{
\de p_{(k), 1} &= \De^{-1}  P_{k+1} [ \pr_l P_{\leq k+1} v^j \pr_j P_{\leq k+1} v^l ]
}

The estimate (\ref{bd:pressIncCo}) almost follows from the elementary bounds on $\co{ P_k v }$ and $\co{\nab P_{\leq k} v }$, and the bound on the operator norm
\[ \| \nab^D \De^{-1} P_k \| \leq C 2^{(D - 2)k} \]
just as in the proof of the estimates for $P_k p$ in Proposition (\ref{prop:firstLPpieceBds}).  The only exceptions are the terms $\de p_{(k), 3HH}$ and $\de p_{(k), 2}$, which both involve the operator $\De^{-1} \pr_j \pr_l P_{\leq k}$.  For these terms we lose a logarithmic factor by estimating the $L^1$ norm of the kernel by
\ali{
\| \De^{-1} \pr_j \pr_l P_{\leq k} \| &\leq \sum_{I = k_0(\T^n)}^k \| \De^{-1} \pr_j \pr_l P_{I} \| \\
&\leq C ( 1 +  |k - k_0(\T^n)|)
}
It is straightforward to see that all of the above estimates worsen by a factor of $2^k$  upon taking a spatial derivative (since they are frequency-localized), leading to (\ref{bd:nab2PressIncCo}).
\end{proof}

% For simpler bounds, look at dtdeltapk.tex

The main task in the proof of (\ref{bd:matDvPress})-(\ref{bd:nabMatDvPress}) is to compute the commutator of the material derivative $\pr_t + P_{\leq k} v \cdot \nab$ with the convolution operators appearing in the expression for $\de p_{(k)}$.  In general, the commutator of a vector field and a convolution operator can be expressed nicely using the fundamental theorem of calculus:
\ali{
[ P_{\leq k} v \cdot \nab, K \ast] f &= -\int_{\R^n} ( P_{\leq k} v^i(x + h) - P_{\leq k} v^i(x) )\fr{\pr}{\pr x^i} f(x+h) K(h) dh \label{eq:commuteBeforeFundThm} \\
&= -\int_0^1 \int_{\R^n} \pr_a P_{\leq k} v^i(x + s h) \fr{\pr}{\pr x^i} f(x+h) K(h) h^a dh ds \label{eq:commuteFundThm}
}
By observing that $\fr{\pr}{\pr x^i} f(x+h) = \fr{\pr}{\pr h^i} f(x+h)$, one can obtain an alternative expression that does not involve the derivative of $f$ by integrating by parts in the $h$ variables, giving %In our case, matters improve because the vector fields at hand are always divergence free in the $h$ variables, so integration by parts in (\ref{eq:commuteBeforeFundThm}) leaves only
\ali{
[ P_{\leq k} v \cdot \nab, K \ast] f &= \int_{\R^n} ( P_{\leq k} v^i(x + h) - P_{\leq k} v^i(x) ) f(x+h)\pr_i K(h) dh \notag \\
&+ \int_{\R^n} \pr_i P_{\leq k} v^i(x + h) f(x+h) K(h) dh  \label{eq:commuteintByParts1long} \\
&= \int_0^1 \int_{\R^n} \pr_a P_{\leq k} v^i(x + s h)  f(x+h) \pr_i K(h) h^a dh ds \notag \\
&+ \int_{\R^n} \pr_i P_{\leq k} v^i(x + h) f(x+h) K(h) dh \label{eq:commuteintByParts2long}
}
We remark that since the vector fields involved are always divergence free, the latter term in the commutator is actually $0$, and we are left with only one of these terms.
\ali{
[ P_{\leq k} v \cdot \nab, K \ast] f &= \int_{\R^n} ( P_{\leq k} v^i(x + h) - P_{\leq k} v^i(x) ) f(x+h)\pr_i K(h) dh\label{eq:commuteintByParts1} \\
&= \int_0^1 \int_{\R^n} \pr_a P_{\leq k} v^i(x + s h)  f(x+h) \pr_i K(h) h^a dh ds \label{eq:commuteintByParts2}
}
However we will never actually have a need for this extra cancellation, as the other term would obey the same bounds even if the vector field were not divergence free.
\begin{proof}[Proof of (\ref{bd:matDvPress})-(\ref{bd:nabMatDvPress})]
From the proof of (\ref{bd:pressIncCo})-(\ref{bd:nab2PressIncCo}), we can express the pressure increment from frequency $2^{k}$ in terms of Low-Low, High-Low and High-High interactions as
\ali{
\de p_{(k)} &= \de p_{(k), LL} + \de p_{(k), HL} + \de p_{(k), HH} \label{eq:dePLLHLHH}\\
\de p_{(k), LL} &= \De^{-1} P_{k+1} ( \pr_l P_{\leq k+1} v^j \pr_j P_{\leq k+1} v^l ) \notag \\%\label{eq:dePLL}\\
&= \De^{-1} P_{k+1} f_{(k),LL} \label{de:dePLLshort} \\
\de p_{(k), HL}&= \De^{-1} \pr_j P_{[k-3, k]} \left[ \pr_l P_{\leq k-4} v^j  P_{k+1} v^l \right] + \De^{-1} \pr_l P_{[k-3, k]} \left[ P_{k+1} v^j  \pr_j P_{\leq k-4} v^l \right] \notag \\%\\
&= 2 \De^{-1} \pr_j P_{[k-3, k]} \left[ \pr_l P_{\leq k-4} v^j  P_{k+1} v^l \right] \notag \\%\label{eq:dePHL} \\
&= \De^{-1} \pr_j P_{[k-3,k]} f_{(k),HL}^j \label{eq:dePHLshort} \\
\de p_{(k), HH} &= \De^{-1} \pr_j \pr_l P_{\leq k} \left[ P_{k+1} v^j P_{k+1} v^l + P_{[ k - 3, k]} v^j P_{k+1} v^l + P_{k+1} v^j P_{[k-3, k]} v^l \right] \notag \\%\label{eq:dePHH} \\
&= \De^{-1} \pr_j \pr_l P_{\leq k} f_{(k),HH}^{jl} \label{eq:dePHHshort}
}
We now apply the operator $\pr_t + P_{\leq k} v \cdot \nab = \pr_t + P_{\leq k} v^i \pr_i $ to the expression (\ref{eq:dePLLHLHH}).  This differentiation generates several commutator terms, the most subtle of which is the commutator $[ P_{\leq k} v \cdot \nab, \De^{-1} \pr_j \pr_l P_{\leq k} ]$ in (\ref{eq:dePHHshort}).  The content of the bounds (\ref{bd:matDvPress})-(\ref{bd:nabMatDvPress}) is simply that this differentiation costs a factor $2^{(1- \a) k} \ctdcxa{v}$ in all the estimates (possibly with a logarithmic loss in some cases).    In what follows we will always neglect the difference between $\pr_t + P_{\leq k} v \cdot \nab$ and $\pr_t + P_{\leq k + 1} v \cdot \nab$ and similar terms which give rise to some harmless factors of the form $P_{ k + 1 } v \cdot \nab$, since we already know from the bounds on spatial derivatives that the operator $P_{ k + 1 } v \cdot \nab$ incurs the desired cost of $2^{(1- \a) k} \ctdcxa{v}$.

The low frequency terms such as $\de p_{(k), LL}$ are treated as follows. 
\ali{
( \pr_t + P_{\leq k} v \cdot \nab ) \de p_{(k), LL} &= \De^{-1} P_{k+1} ( \pr_t + P_{\leq k} v \cdot \nab ) f_{(k),LL} + [( \pr_t + P_{\leq k} v \cdot \nab ) , \De^{-1} P_{k+1} ] f_{(k),LL} 
}
For the function $f_{(k), LL}$, we have the bounds
\ali{
\co{\nab^D f_{(k), LL} } &\leq C_D 2^{(D + 2 - 2\a) k} \ctdcxa{v}^2 \\
\co{(\pr_t + P_{\leq k} v \cdot \nab)f_{(k), LL} } &\leq C_{\T^n} (1 + | k - k_0(\T^n) | ) 2^{(3 - 3 \a) k} \ctdcxa{v}^2
}
which come from Proposition (\ref{prop:firstLPpieceBds}) and the bound (\ref{ineq:mtDvBound3}) for $\co{ \nab (\pr_t + P_{\leq k} v \cdot \nab)P_{\leq k} v }$.  

From (\ref{eq:commuteFundThm}), the commutator can be written in the form
\ali{
[( \pr_t + P_{\leq k} v \cdot \nab ) , \De^{-1} P_{k+1} ] f_{(k),LL} &= \int_0^1 \int_{\R^n} \pr_a P_{\leq k} v^i(x + s h) \fr{\pr}{\pr x^i} f_{(k),LL}(x+h) \De^{-1} \eta_{k + 1}(h) h^a dh ds
}
which obeys the desired estimate without losing a factor $(1 + |k - k_0(\T^n)|)$, as we have the scaling bound
\[ \| \De^{-1} \eta_{k+1}(h) |h| \|_{L^1} \leq C 2^{-3k} \]

The term $\de p_{(k), HL}$ is treated similarly, with only a few differences such as the appearance of high frequency terms such as $( \pr_t + P_{\leq k} v \cdot \nab ) P_{k + 1} v$ and a different scaling for the operator. %the only difference being the appearance of terms such as $( \pr_t + P_{\leq k} v \cdot \nab ) P_{k + 1} v$ which obey the correct bounds and cosmetic differences such as the $k-4$ in $( \pr_t + P_{\leq k} v \cdot \nab ) P_{\leq k - 4} v$.

The term that requires a more subtle analysis is the term $\de p_{(k), HH}$, which contains the highly nonlocal operator 
\ali{
\De^{-1} \pr_j \pr_l P_{\leq k} &= \sum_{I = k_0(\T^n)}^k \De^{-1} \pr_j \pr_l P_{I} 
}
In this case, we use the expression (\ref{eq:commuteintByParts2}) to write the commutator term as
\ali{
[( \pr_t + P_{\leq k} v \cdot \nab ) , &\De^{-1} \pr_j \pr_l P_{\leq k} ] f_{(k), HH}^{jl}(x+h) = \notag \\
&= \sum_{I = k_0(\T^n)}^k  \int_0^1 \int_{\R^n} \pr_a P_{\leq k} v^i(x + s h)  f_{(k), HH}^{jl}(x+h) \pr_i \pr_j \pr_l \De^{-1} \eta_I(h) h^a dh ds \label{eq:nonLocCommutator}
}
which immediately gives an estimate on the operator norm
\ali{
\sup_t \| [( \pr_t + P_{\leq k} v \cdot \nab ) , \De^{-1} \pr_j \pr_l P_{\leq k} ] \| &\leq \sum_{I = k_0(\T^n)}^k \co{ \nab P_{\leq k} v } \| |h| \nab^3 \De^{-1} \eta_I(h) \|_{L^1_h} \notag \\
&\leq C \sum_{I = k_0(\T^n)}^k \co{ \nab P_{\leq k} v } \cdot 1 \notag \\
&\leq C_{\T^n} ( 1 + |k - k_0(\T^n)|) 2^{(1 - \a) k} \ctdcxa{v}
}
This bound is worse than the estimate we used for the operator 
\[ \| \De^{-1} \pr_j \pr_l P_{\leq k} \| \leq C_{\T^n} ( 1 + |k - k_0(\T^n)|) \]
by exactly the factor
\[ \Big| [ ( \pr_t + P_{\leq k} v \cdot \nab ), ~\cdot~ ] \Big| \leq C 2^{(1-\a)k} \ctdcxa{v} \]
we desire, and the estimates (\ref{bd:matDvPress})-(\ref{bd:nabMatDvPress}) follow from differentiating the above formulas in the $x$ variables at a cost of $2^k$ per derivative.

With the expression (\ref{eq:nonLocCommutator}) for the commutator in hand, the bounds (\ref{bd:matDvPress})-(\ref{bd:nabMatDvPress}) follow quickly from the estimates in Proposition~\ref{prop:firstLPpieceBds}.
\end{proof}

We are now able to prove Theorem~\ref{thm:matDvPCts}.
\begin{cor}\label{cor:thm13}

If $\a > 1/3$, the distribution $\pr_t p + \pr_j( p v^j)$ is continuous when $p$ is normalized to have integral $0$ at every time $t$.

If $\a > 2/3$, the distribution $\pr_t \pr^l p + \pr_j( v^j \pr^l p )$ is continuous.
\end{cor}
\begin{proof}
Set $p_{(k)} = \sum_{I = k_0(\T^n)}^k \de p_{(I)}$.  Then $p_{(k)} \to p$ uniformly in space and time, and as a consequence 
\ali{
 \pr_t p_{(k)} + \pr_j( p_{(k)} v^j ) &\rightharpoonup \pr_t p + \pr_j( p v^j ) \label{weakConvergeMatDvP}
}
weakly as distributions from the uniform continuity of $v$.  By regularizing in time, it can be shown that for each $k$, we have the identity
\ali{
\pr_t p_{(k)} + \pr_j( p_{(k)} v^j ) &= \pr_t p_{(k)} + v^j \pr_j p_{(k)} \label{eq:sillyDivIdentity}
}
using the fact that $\pr_j v^j = 0$ as a distribution.  From this identity we conclude that the convergence of (\ref{weakConvergeMatDvP}) is actually uniform in $(t,x)$ when $\a > 1/3$, because we have proven the bound
\ALI{
\co{ (\pr_t + v \cdot \nab) \de p_{(k)} } &\leq \co{ (\pr_t + P_{\leq k} v \cdot \nab) \de p_{(k)} } + \co{ (v - P_{\leq k} v) \cdot \nab \de p_{(k)} } \\
&\leq C_{\T^n} ( 1 + |k - k_0(\T^n)| ) 2^{(1- 3 \a) k } \ctdcxa{v}^3
}
which decays exponentially as $k \to \infty$ whenever $\a > 1/3$.

Assuming that $\a > 2/3$, the continuity of the distribution $\pr_t \pr^l p + \pr_j( v^j \pr^l p )$ follows similarly.  Namely, we see that 
\ali{
 \pr_t \pr^l p_{(k)} + \pr_j(  v^j \pr^l p_{(k)}) &\rightharpoonup \pr_t \pr^l p + \pr_j(  v^j \pr^l p ) \label{weakConvergeMatDvNabP}
}
as distributions, since $\nab p_{(k)} \to \nab p$ uniformly for $\a > 1/2$.  The uniform convergence for $\a > 2/3$ then follows from the bounds 
\ALI{
\co{ (\pr_t + v \cdot \nab) \nab \de p_{(k)} } &\leq \co{ (\pr_t + P_{\leq k} v \cdot \nab) [ \nab \de p_{(k)} ] } + \co{ (v - P_{\leq k} v) \cdot \nab [ \nab \de p_{(k)} ] } \\
&\leq \co{ \nab (\pr_t + P_{\leq k} v \cdot \nab) \de p_{(k)} } + \co{ \nab P_{\leq k} v } \co{\nab \de p_{(k)} } \notag \\
&+ \co{ (v - P_{\leq k} v) \cdot \nab [ \nab \de p_{(k)} ] } \\
&\leq C_{\T^n} ( 1 + |k - k_0(\T^n)| ) 2^{(2- 3 \a) k } \ctdcxa{v}^3
}
or alternatively by repeating the commutator estimates above for $\nab \de p_{(k)}$ directly.
\end{proof}

The proof above is not the most robust proof of Corollary (\ref{cor:thm13}), in particular because it does not give any H\"{o}lder regularity.  As we will see when we have built up the relevant preliminary estimates, it is better to use the approximation
\ALI{
 \pr_t p_{(k)} + \pr_j( p_{(k)} P_{\leq k} v^j ) &\rightharpoonup \pr_t p + \pr_j( p v^j ) 
}
which avoids regularizing in time as in (\ref{eq:sillyDivIdentity}) and actually converges in the appropriate H\"{o}lder spaces.

For now we record the following bounds to accompany Proposition~\ref{prop:pressIncBds1} for use in the later applications of Section~\ref{sec:secondMatDvPressIncs}.%which will be used in later applications that require further bounds on low frequencies.
\begin{prop}\label{cor:morePressIncEstimates} Under the assumptions of Proposition (\ref{prop:pressIncBds1}),
\ali{
\co{\nab^{2 + D} p_{(k)} } &\leq C_{D, \T^n} (1 + |k - k_0(\T^n)|) 2^{(D+2(1-\a)) k} \ctdcxa{v}^2 \label{ineq:nab2Dpk} \\
\co{ \nab^D (\pr_t + P_{\leq k} v \cdot \nab) \nab^2 p_{(k)} } &\leq C_{D,\T^n} (1 + |k - k_0(\T^n)|) 2^{(D+ 3(1-\a))k} \ctdcxa{v}^3 \label{ineq:matDvNab2Dpk}
}
Furthermore, if $\a < 2/3$, we have
\ali{
\co{ \nab^D (\pr_t + P_{\leq k} v \cdot \nab) \nab p_{(k)} } &\leq C_{D,\T^n} (1 + |k - k_0(\T^n)|) 2^{(D + 2 - 3 \a)k} \ctdcxa{v}^3 \label{eq:goodMatDvPressGrad}
}
\end{prop}
\begin{proof}
The bound (\ref{ineq:nab2Dpk}) is immediate from the representation
\ali{
\nab^2 p_{(k)} &= \nab^2 \pr_j \pr_l \De^{-1} P_{\leq k} ( \plkv^j \plkv^l ) \\
&= \nab^2 \De^{-1} P_{\leq k} \pr_l \plkv^j \pr_j \plkv^l
}
and the estimate
\[ \| \nab^2 \De^{-1} P_{\leq k} \| \leq C (1 + |k - k_0(\T^n)|) \]

The estimate for the material derivative of $\nab^2 p_{(k)}$ follows by commuting using the formula (\ref{eq:commuteintByParts2}).

To obtain the bound (\ref{eq:goodMatDvPressGrad}), set
\ali{ 
\fr{D_{\lk}}{\pr t} &= ( \pr_t + \plkv \cdot \nab)
}
and define
\ali{
 \de_{(k)}\left[\fr{D_{\lk}}{\pr t} \nab p_{(k)}\right] &=  \fr{ D_{\lk}}{\pr t} \nab \de p_{(k-1)} + P_{k} v \cdot \nab \nab p_{(k-1)} \label{eq:incOfDdtNabpk} \\
&= \fr{ D_{\lk}}{\pr t} \nab p_{(k)} - \fr{ D_{\leq k-1}}{\pr t} \nab p_{(k-1)} 
}
Then
\ali{
\fr{D_{\lk}}{\pr t} \nab p_{(k)} &= \sum_{I = k_0(\T^n)}^k \de_{(I)}\left[ \fr{D_{\leq I}}{\pr t} \nab p_{(I)}\right] 
}
so (\ref{eq:goodMatDvPressGrad}) follows from the formula (\ref{eq:incOfDdtNabpk}) with the bounds (\ref{ineq:nab2Dpk})-(\ref{ineq:matDvNab2Dpk}) as in
\ali{
\co{ (\pr_t + P_{\leq k} v \cdot \nab) \nab p_{(k)} } &\leq C_{\T^n} \sum_{I = k_0(\T^n)}^k (1 + | I - k_0(\T^n)| ) 2^{(2 - 3\a)I} \ctdcxa{v}^3 \\
&\leq C_{\T^n} (1 + | k - k_0(\T^n)| ) 2^{(2 - 3\a)k} \ctdcxa{v}^3 \label{eq:sumByParts}
}
Here we have used summation by parts in $I$ and the condition $\a < 2/3$ to bound the sum.  The bound on higher spatial derivatives follows similarly.
\end{proof}

\section{Material derivative estimates for the forcing terms} \label{sec:forcingTerms1}

So far we have established Theorems (\ref{thm:matDvPCts}) and (\ref{thm:energyReg}) as well as the case $\b \leq 1/2$ of Theorem (\ref{thm:timeRegBounds}) by drawing on the basic estimates on Littlewood Paley pieces of Proposition (\ref{prop:firstLPpieceBds}).  Proving the full strength of Theorem (\ref{thm:timeRegBounds}) requires going beyond the first time derivative of the pressure, so we will be interested in developing further estimates on second material derivatives for $P_k v$ and $\nab P_{\leq k} v$ as a preliminary step in this direction.

\subsection{Material derivative estimates on LP pieces for the pressure} \label{sec:matDvLPp}

As a first step, we will prove the following bounds for coarse scale material derivatives of Littlewood Paley pieces of the pressure.
\begin{prop}\label{prop:DdtnabDPkp}  Assume that $v \in C_t C_x^\a$ for some $1/3 < \a < 1$ is a solution to incompressible Euler with pressure $p$.  Then for any integer $D \geq 0$ we have the bound
\ali{
\co{ \nab^D (\pr_t + P_{\leq k} v \cdot \nab) P_k p } &\leq C_{D, \a} 2^{(D + 1 - 3 \a) k} \ctdcxa{v}^3 \label{ineq:bdForDdtPkp}
}
\end{prop}
\begin{proof}
We first consider the case $D = 0$.

In the proof of Proposition (\ref{prop:firstLPpieceBds}), we used the incompressibility of $v$ to obtain a decomposition
\begin{align}
P_k p &= HH_k p + HL_k p + LL_k p \label{eq:Pkplong} \\
HH_k p &= \pr_j \pr_l \De^{-1} P_k \left[ ( v^j - P_{\leq k} v^j) (v^l - P_{\leq k} v^l) \right]  \\
HL_k p &= \pr_l \De^{-1} P_k \left[ ( v^j - P_{\leq k} v^j) \pr_j P_{\leq k} v^l \right] + \pr_j  \De^{-1} P_k \left[ \pr_l P_{\leq k} v^j ( v^l- P_{\leq k} v^l) \right] \\
LL_k p &= \De^{-1} P_k \left[ \pr_l P_{\leq k} v^j \pr_j P_{\leq k} v^l \right] 
\end{align}
We would like to estimate $(\pr_t + P_{\leq k} v \cdot \nab) P_k p$ and its derivatives by commuting the advective derivative $(\pr_t + P_{\leq k} v \cdot \nab)$ with the various convolution operators appearing on the right hand side of (\ref{eq:Pkplong}).  The difficulty which restricts us to $\a > 1/3$ is that the high frequency components of $(v - P_{\leq k} v)$ do not have a good estimate for the material derivative at the scale $2^{-k}$.  For the High-Low terms in (\ref{eq:Pkplong}), we can escape this difficulty with the higher frequencies using the bandlimited property of Littlewood-Paley projections to write
\begin{align}
HL_k p &= 2 \pr_j  \De^{-1} P_k \left[ \pr_l P_{\leq k} v^j ( v^l- P_{\leq k} v^l) \right] \notag \\
&= 2 \pr_j  \De^{-1} P_k \left[ \pr_l P_{\leq k} v^j P_{[k, k+2]} v^l) \right].
\end{align}
However, it seems that the best we can do for the High-High interactions is to write
\ali{
HH_k p &= \pr_j \pr_l \De^{-1} P_k \left[ ( v^j - P_{\leq k} v^j) (v^l - P_{\leq k} v^l) \right] \notag \\
&= \sum_{I = k}^\infty \sum_{ \substack{J \geq k \\ |I - J| \leq 2}} \pr_j \pr_l \De^{-1} P_k \left[ P_I v^j P_J v^l \right] \label{eq:expandParaDiff}
}
and to bound the material derivative of this term by first writing
\ali{
(\pr_t + P_{\leq k} v \cdot \nab) HH_k p &= \sum_{I=k}^\infty [ (\pr_t + P_{\leq I} v \cdot \nab) - P_{[k, I]} v \cdot \nab ] \pr_j \pr_l \De^{-1} P_k\left[ \sum_{ \substack{J \geq k \\ |I - J| \leq 2}} P_I v^j P_J v^l \right] \label{eq:expandDdtHighHigh}\\
&= \sum_{I = k}^\infty (A_{(I)} - B_{(I)})
}
We can estimate the latter term by
\ali{
\co{ B_{(I)} } &\leq \sum_{ \substack{J \geq k \\ |I - J| \leq 2}} \co{  P_{[k, I]} v } \cdot \| \nab \pr_j \pr_l \De^{-1} P_k \| \co{ P_I v^j P_J v^l  }\notag\\
&\leq C 2^{(1-\a) k} 2^{-2 \a I} \ctdcxa{v}^3
}
which is acceptable for (\ref{ineq:bdForDdtPkp}) upon summing over $I \geq k$.

The term $A_{(I)}$ is more dangerous, and involves a commutator,
\ali{
A_{(I)} &= \sum_{\substack{J \geq k \\ |I - J| \leq 2}} \pr_j \pr_l \De^{-1} P_k \left[ (\pr_t + P_{\leq I} v \cdot \nab) \left[ P_I v^j P_J v^l \right] \right] \notag \\
&+ \sum_{\substack{J \geq k \\ |I - J| \leq 2}} [(\pr_t + P_{\leq I} v \cdot \nab) , \pr_j \pr_l \De^{-1} P_k] \left[ P_I v^j P_J v^l \right] \\
&= A_{(I), 1} + A_{(I), 2}
}
For the term $A_{(I), 1}$ we use the bounds
\ALI{
\| \pr_j \pr_l \De^{-1} P_k  \| &\leq C \\
\co{ (\pr_t + P_{\leq I} v \cdot \nab) \left[ P_I v^j P_J v^l \right] } &\leq C 2^{(1-3 \a) I} \ctdcxa{v}^3
}
For the commutator term $A_{(I), 2}$, we have an estimate for the operator
\ALI{
\| [(\pr_t + P_{\leq I} v \cdot \nab) , \pr_j \pr_l \De^{-1} P_k] \| &\leq C 2^{(1-\a) I} \ctdcxa{v}
}
which is proven by the same integration by parts used to bound the commutator (\ref{eq:nonLocCommutator}) in Section (\ref{sec:timeRegPress}).  This bound together with the estimate $\co{ P_I v^j P_J v^l } \leq C 2^{-2\a I}$ gives (\ref{ineq:bdForDdtPkp}) for $D = 0$ after summing over $I$.

To obtain the estimate (\ref{ineq:bdForDdtPkp}) for $D > 0$, it is not safe to differentiate the terms in (\ref{eq:expandDdtHighHigh}) since the sum over $I$ will diverge.  Instead, one can first differentiate $\nab^D P_k p$ in space, letting the derivatives fall on the Littlewood Paley projections.  The estimate (\ref{ineq:bdForDdtPkp}) follows by first repeating the proof above to obtain the desired bound for $\co{(\pr_t + P_{\leq k} v \cdot \nab) \nab^D P_k p }$, and then commuting the material and spatial derivatives to obtain (\ref{ineq:bdForDdtPkp}).
\end{proof}

The estimate (\ref{ineq:bdForDdtPkp}) can be used to give another proof of Theorem (\ref{thm:matDvPCts}) along the same lines as the proof in Corollary (\ref{cor:thm13}).  

The same method also gives the following estimate 
\begin{prop} Under the assumptions of Proposition (\ref{prop:DdtnabDPkp}) and $0 < \a < 1$,
\ali{
\co{ \nab^D (\pr_t + \pleqkvcn) \nab^2 P_{\leq k} p } &\leq C_{D, \a} 2^{(D + 3 - 3 \a)k} \ctdcxa{v}^3 \label{ineq:matdvnab2plqkp}
}
\end{prop}
Applying the method of proof from Proposition (\ref{prop:DdtnabDPkp}) leads to an extra loss of $(1 + |k - k_0(\T^n)|)$ that we prefer to avoid.
\begin{proof}
We write
\ali{
(\pr_t + P_{\lk} v \cdot \nab) \nab^2 P_{\leq k} p &= \sum_{I = \kotn}^k [ (\pr_t + P_{\leq I + 1} v \cdot \nab) \nab^2 P_{\leq I + 1} p - (\pr_t + P_{\leq I } v \cdot \nab) \nab^2 P_{\leq I } p ] \\
&= \sum_{I = \kotn}^k (\pr_t + P_{\leq I + 1} v \cdot \nab)\nab^2 P_{I + 1} p + P_{I+1} v \cdot \nab (\nab^2 P_{\leq I + 1} p )
}
The estimate (\ref{ineq:matdvnab2plqkp}) now follows from Propositions (\ref{prop:DdtnabDPkp}) and (\ref{prop:matDvLowFreqs}) by differentiating and summing.  
\ALI{
\co{ \nab^D (\pr_t + \pleqkvcn) \nab^2 P_{\leq k} p } &\leq C \sum_{I = - \infty}^k 2^{(D + 3 - 3 \a)I} \ctdcxa{v}^3
}
Since $\a < 1$, the last term controls the geometric series.
\end{proof}

%The bound (\ref{ineq:matdvnab2plqkp}) follows by commuting from the same estimate for $\co{ (\pr_t + \pleqkvcn) \nab^{D + 2} P_{\leq k} p }$.  To establish the estimate, we repeat the proof of Proposition (\ref{prop:DdtnabDPkp}), but the kernels involved all have the form
%\[ \nab^{D + i} \De^{-1} \nab^2 P_{\leq k} \]
%for $i = 0, 1, 2$, and therefore satisfy the bounds
%\ALI{
%\| \nab^{D + i} \De^{-1} \nab^2 P_{\leq k} \| &\leq C (1 + | k - k_0(\T^n)| ) 2^{(D + i)k} \\
%\| [ \pr_t + \pleqkvcn, \nab^{D + i} \De^{-1} \nab^2 P_{\leq k} ] \| &\leq C (1 + | k - k_0(\T^n)| ) 2^{(D + i + 1 - \a)k}\ctdcxa{v}
%}
%using the same methods we have now applied several times in Sections (\ref{sec:timeRegPress}) and (\ref{sec:matDvLPp}).

In Section (\ref{sec:ReynStressMatDv1}) below, we establish analogous estimates for the Reynolds stress.  The proofs are very similar, but we will improve on the treatment of the High-High terms to give a proof that turns out to be more robust for our later applications.

\subsection{Estimates for the Reynolds stress} \label{sec:ReynStressMatDv1}

Here we collect all the necessary bounds on the Reynolds stress
\ali{
R_{\leq k}^{jl} &= P_{\leq k} v^j P_{\leq k} v^l - P_{\leq k}(v^j v^l) \label{eq:Rleqkdef}
}
and its derivatives which are used in the proofs of Theorems~\ref{thm:timeRegBounds}-\ref{thm:higherMatDv}.  In the process, we also illustrate all the main technical ideas necessary for estimating $R_{\leq k}$ that will be used in the remainder of the paper.

We start by giving an alternative proof of the commutator estimate from \cite{CET} and the generalization in \cite{deLSzeC1iso} which includes bounds on spatial derivatives $\co{ \nab^D R_{\leq k} }$.  %This commutator estimate was also used to bound certain advective derivatives in the construction of energy dissipating $C_{t,x}^{1/5-\ep}$ Euler flows in \cite{deLSzeBuck}.  The analogous terms in \cite{isett} were estimated by studying a a different type of commutator, but the related estimates in \cite{isett} can also be obtained from the bilinear formulation of (\ref{ineq:generalCETbound}) stated in \cite{deLSzeC1iso}.  
The proof we give here will be more flexible in that we will be able to obtain the necessary bounds on the material derivative $(\pr_t + \pleqkvcn)R_{\leq k}^{jl}$ using the same method.

\begin{prop}
If $0 < \a \leq 1$, then for any $D \geq 0$
\ali{
\co{ \nab^D R_{\leq k} } &\leq C_D 2^{(D - 2 \a)k} \| v \|_{{\dot C}_x^\a}^2 \label{ineq:generalCETbound}
}
\end{prop}
The proof given here does not use any special properties of Littlewood-Paley projections and generalizes to other mollifiers in the Schwartz class as well.
\begin{proof}
We start by observing a ``Galilean invariance'' of the commutator (\ref{eq:Rleqkdef}).  Namely, the expression 
\ali{
R_{\leq k}^{jl} &= \int \int v^j(x+h_1) v^l(x+h_2) \eta_{\leq k}(h_1) \eta_{\leq k}(h_2) dh_1 dh_2 - \int v^j(x+h) v^l(x+h) \eta_{\leq k}(h) dh \label{eq:lookAtCommutator1}
}
for $x \in \T^n$ has a schematic form similar to a variance 
\[- R_{\leq k} (t,x) = \bbE[v^2](t,x) - \bbE[v]^2(t,x) \]
since $\int \eta_{\leq k}(h) dh = 1$.  For example, $R_{\leq k}^{jl}$ would be negative definite if $\eta_{\leq k}$ were positive definite, although this is not the case for Littlewood-Paley projections.

Just as the variance of a random variable remains invariant under the addition of a constant, we can observe that (\ref{eq:lookAtCommutator1}) remains invariant when we subtract from $v^j$ any vector $A^j(x)$ at each point
\begin{align}
R_{\leq k}^{jl} &= \int \int(  v^j(x+h_1) - A^j(x) ) (v^l(x+h_2) - A^l(x)) \eta_{\leq k}(h_1) \eta_{\leq k}(h_2) dh_1 dh_2  \label{eq:canShiftByConst1} \\
&- \int ( (v^j(x+h) - A^j(x) )(v^l(x+h) - A^l(x)) \eta_{\leq k}(h) dh \label{eq:canShiftByConst2}
\end{align}
By choosing $A^j(x) = v^j(x)$ and $A^l(x) = v^l(x)$ we immediately obtain the $C^0$ bound in \eqref{ineq:generalCETbound} (this choice leads also to the decomposition in the \cite{CET} argument).  However, it is even more natural to choose $A^j(x) = P_{\leq k} v^j(x)$ and $A^l(x) = P_{\leq k} v^l(x)$ in analogy with the expression $\bbE[( v - \bbE[v])^2]$ for a variance.

With this latter choice, the first term (\ref{eq:canShiftByConst1}) disappears, leaving the expression
\ali{
R_{\leq k}^{jl} &= \int ( v^j (x + h) - P_{\leq k} v^j(x) ) ( v^l(x+h) - P_{\leq k} v^l(x) ) \eta_{\leq k}(h) dh  \label{eq:nowRleqIsAVariance}
}
We now expand \eqref{eq:nowRleqIsAVariance} by adding and subtracting $P_{\leq k} v(x+h)$ to each term and obtain
\ali{
R_{\leq k}^{jl} &= R_{\leq k, HH}^{jl} + R_{\leq k, HL}^{jl} + R_{\leq k, LL}^{jl} \label{eq:RleqkTrichotomy} \\
R_{\leq k, HH}^{jl} &= \int ( v^j(x+h) - P_{\leq k} v^j(x+h) ) (v^l(x+h) - P_{\leq k}v^l(x+h) ) \eta_{\leq k}(h) dh \\
&= P_{\leq k} [( v^j - P_{\leq k} v^j ) (v^l - P_{\leq k}v^l ) ] \\
R_{\leq k, HL}^{jl} &= \int ( v^j(x+h) - P_{\leq k} v^j(x+h) ) (P_{\leq k} v^l(x+h) - P_{\leq k}v^l(x) ) \eta_{\leq k}(h) dh \notag \\
&+ \int ( P_{\leq k} v^j(x+h) - P_{\leq k} v^j(x) ) (v^l(x+h) - P_{\leq k}v^l(x+h) ) \eta_{\leq k}(h) dh \\
R_{\leq k, LL}^{jl} &= \int ( P_{\leq k} v^j(x+h) - P_{\leq k} v^j(x) ) (P_{\leq k} v^l(x+h) - P_{\leq k}v^l(x) ) \eta_{\leq k}(h) dh
}
The bound (\ref{ineq:generalCETbound}) now follows quickly from the expanded form (\ref{eq:RleqkTrichotomy}).  Namely, it is easy to see that 
\ALI{
\co{ \nab_x^D ( P_{\leq k} v^j(x+h) - P_{\leq k} v^j(x) ) } &\leq C_D 2^{(D - \a) k} \ctdcxa{v}
}
For example, every low frequency term can be written as 
\ali{ 
( P_{\leq k} v^j(x+h) - P_{\leq k} v^j(x) ) &= \int_0^1 \pr_a P_{\leq k} v^j(x+s h) h^a ds
}
and the factor of $h^a$ can be incorporated into the mollifier $\eta_{\leq k}(h)$, which we estimate in $L^1_h$.

For high frequency terms, it is always possible to integrate by parts to estimate the derivatives in (\ref{ineq:generalCETbound}).  For example, we have
\ali{
\fr{\pr}{\pr x^i} R_{\leq k, HH}^{jl}(x) &= - \int  ( v^j (x + h) - P_{\leq k} v^j(x+h) ) ( v^l(x+h) - P_{\leq k} v^l(x+h) ) \fr{\pr}{\pr h^i}[\eta_{\leq k}(h)] dh
}
and one can similarly integrate by parts when the derivative hits the high-frequency factor in the High-Low terms
\ali{
\int \fr{\pr}{\pr x^i}( v^j(x+h) - &P_{\leq k} v^j(x+h) ) (P_{\leq k} v^l(x+h) - P_{\leq k}v^l(x) ) \eta_{\leq k}(h) dh =\\
&= \int \fr{\pr}{\pr h^i}( v^j(x+h) - P_{\leq k} v^j(x+h) ) (P_{\leq k} v^l(x+h) - P_{\leq k}v^l(x) ) \eta_{\leq k}(h) dh \\
&= - \int ( v^j(x+h) - P_{\leq k} v^j(x+h) ) \fr{\pr}{\pr h^i}\left[ (P_{\leq k} v^l(x+h) - P_{\leq k}v^l(x) ) \eta_{\leq k}(h) \right] dh
}
In every case, each spatial derivative in the $x$ variable costs a factor $2^k$ in the estimate, so combining these observations gives (\ref{ineq:generalCETbound}).
\end{proof}

The proof above allows us to estimate the material derivative of the Reynolds stress arising from mollifying a solution to the Euler equations provided we assume sufficient regularity on $v$.
\begin{prop}\label{prop:firstMatdvRlk} As in Section (\ref{sec:matDvLPp}), assume that $v \in C_t C_x^\a$ for some $1/3 < \a < 1$ is a solution to incompressible Euler with pressure $p$.  Then for any integer $D \geq 0$ we have the bound
\ali{
\co{ \nab^D (\pr_t + P_{\leq k} v \cdot \nab) R_{\leq k} } &\leq C_{D, \a} 2^{(D + 1 - 3 \a) k} \ctdcxa{v}^3 \label{ineq:bdForDdtRleqk}
}
\end{prop}
We start the proof by considering the case $D = 0$.  The case $D > 0$ can be deduced by establishing the same bound for $\co{ (\pr_t + P_{\leq k} v \cdot \nab) \nab^D R_{\leq k}}$, and this latter estimate can be obtained by modifying the argument below.
%We will give the proof for $D = 0$.  The case $D > 0$ will follow from establishing the same bound for $\co{ (\pr_t + P_{\leq k} v \cdot \nab) \nab^D R_{\leq k}}$, which can be obtained by modifying the proof below.
\begin{proof}
The proof closely mirrors that of Proposition (\ref{prop:DdtnabDPkp}), but involves different types of commutator terms.  As in the proof of Proposition (\ref{prop:DdtnabDPkp}), the most dangerous terms arise from taking material derivatives of high frequency terms.  As in the proof of Proposition (\ref{prop:DdtnabDPkp}), we can also show that $R_{\leq k, HH}$ is the only term to which frequencies much larger than $2^k$ contribute.  Namely, due to the bandlimited property of Littlewood-Paley projections, we have that the High-Low term can be expressed as
\ali{
R_{\leq k, HL}^{jl} &= \int ( v^j(x+h) - P_{\leq k} v^j(x+h) ) (P_{\leq k} v^l(x+h) - P_{\leq k}v^l(x) ) \eta_{\leq k}(h) dh \notag \\
&+ \int ( P_{\leq k} v^j(x+h) - P_{\leq k} v^j(x) ) (v^l(x+h) - P_{\leq k}v^l(x+h) ) \eta_{\leq k}(h) dh \\
&= \int ( P_{\leq k+2} v^j(x+h) - P_{\leq k} v^j(x+h) ) (P_{\leq k} v^l(x+h) - P_{\leq k}v^l(x) ) \eta_{\leq k}(h) dh \notag \\
&+ \int ( P_{\leq k} v^j(x+h) - P_{\leq k} v^j(x) ) (P_{\leq k+2} v^l(x+h) - P_{\leq k}v^l(x+h) ) \eta_{\leq k}(h) dh \\
R_{\leq k, HL}^{jl} &= \int P_{[k, k+2]} v^j(x+h) (P_{\leq k} v^l(x+h) - P_{\leq k}v^l(x) ) \eta_{\leq k}(h) dh \notag \\
&+ \int ( P_{\leq k} v^j(x+h) - P_{\leq k} v^j(x) ) P_{[k, k+2]}v^l(x+h) \eta_{\leq k}(h) dh \label{eq:RHLgood}\\
&= R_{\leq k, HL1}^{jl} + R_{\leq k, HL2}^{jl} 
}
We now wish to compute and estimate the material derivatives of these terms, so to begin we compute
\ali{
(\pr_t + P_{\leq k} v^i(x) &\fr{\pr}{\pr x^i})R_{\leq k, HL1}^{jl}(x) =\\ 
&= \int  (\pr_t + P_{\leq k} v^i(x) \fr{\pr}{\pr x^i})P_{[k, k+2]} v^j(x+h) (P_{\leq k} v^l(x+h) - P_{\leq k}v^l(x) ) \eta_{\leq k}(h) dh \notag \\
&+ \int P_{[k, k+2]} v^j(x+h) (\pr_t + P_{\leq k} v^i(x) \fr{\pr}{\pr x^i})[ P_{\leq k} v^l(x+h) - P_{\leq k}v^l(x) ] \eta_{\leq k}(h) dh \\
&= A_{HL,I}^{jl} + A_{HL,II}^{jl}
}
The first of these terms can be expressed as
\ali{
A_{HL,I}^{jl} &= \int [ (\pr_t + P_{\leq k} v \cdot \nab) P_{[k, k+2]} ]v^j(x+h) (P_{\leq k} v^l(x+h) - P_{\leq k}v^l(x) ) \eta_{\leq k}(h) dh \notag \\
&- \int ( P_{\leq k} v^i(x + h) - P_{\leq k} v^i(x) ) \pr_i P_{[k, k+2]} v^j(x+h) (P_{\leq k} v^l(x+h) - P_{\leq k}v^l(x) ) \eta_{\leq k}(h) dh \\
&= \int_0^1 \int [ (\pr_t + P_{\leq k} v \cdot \nab) P_{[k, k+2]} ]v^j(x+h) \pr_a P_{\leq k} v^l(x+ sh) h^a \eta_{\leq k}(h) dh \notag \\
&- \int_0^1 \int_0^1 \int \pr_{a_1} P_{\leq k} v^i(x + s_1 h) \pr_i P_{[k, k+2]} v^j(x+h) \pr_{a_2} P_{\leq k} v^l(x+ s_2 h) h^{a_1} h^{a_2} \eta_{\leq k}(h) dh ds_1 ds_2
}
so that the bound (\ref{ineq:bdForDdtRleqk}) is visible for any $0 < \a \leq 1$.  The second term can likewise be written as
\ali{
A_{HL,II}^{jl} &= \int_0^1 \int P_{[k, k+2]} v^j(x+h) (\pr_t + P_{\leq k} v^i(x) \fr{\pr}{\pr x^i})[ \pr_b P_{\leq k}v^l(x + s h) ] h^b \eta_{\leq k}(h) dh ds \\
&= \int_0^1 \int P_{[k, k+2]} v^j(x+h) [ (\pr_t + P_{\leq k} v \cdot \nab)\pr_b P_{\leq k} ] v^l(x+sh) h^b \eta_{\leq k}(h) dh ds \notag \\
&- \int_0^1 \int P_{[k, k+2]} v^j(x+h) [ P_{\leq k} v^i(x+sh) - P_{\leq k} v^i(x) ] \pr_i\pr_b P_{\leq k} v^l(x+sh) h^b \eta_{\leq k}(h) dh ds \\
&= \int_0^1 \int P_{[k, k+2]} v^j(x+h) [ (\pr_t + P_{\leq k} v \cdot \nab)\pr_b P_{\leq k} ] v^l(x+sh) h^b \eta_{\leq k}(h) dh ds \notag \\
&- \int_0^1 \int_0^1 \int P_{[k, k+2]} v^j(x+h) \pr_a P_{\leq k} v^i(x+s_1 s h) \pr_i\pr_b P_{\leq k} v^l(x+sh) (s h^a) h^b \eta_{\leq k}(h) dh ds ds_1 
}
and can quickly be seen to obey (\ref{ineq:bdForDdtRleqk}) as well thanks to Proposition (\ref{prop:matDvLowFreqs}).

The Low-Low term $R_{\leq k, LL}^{jl}$ can be treated similarly after it has been represented in the form
\ali{
R_{\leq k, LL}^{jl} &= \int ( P_{\leq k} v^j(x+h) - P_{\leq k} v^j(x) ) (P_{\leq k} v^l(x+h) - P_{\leq k}v^l(x) ) \eta_{\leq k}(h) dh \label{eq:RLLgood}\\
&= \int_0^1 \int_0^1 \int \pr_{a_1} P_{\leq k} v^j(x+s_1 h) \pr_{a_2} P_{\leq k} v^l(x+s_2 h) h^{a_1} h^{a_2} \eta_{\leq k}(h) dh ds_1 ds_2
}

The only term remaining for the proof of (\ref{ineq:bdForDdtRleqk}) is the term
\ali{
R_{\leq k, HH}^{jl} &= P_{\leq k} [( v^j - P_{\leq k} v^j ) (v^l - P_{\leq k}v^l ) ]
}
which also limits the present method to $\a > 1/3$.  This term can be handled by the exact same technique as in the proof of Proposition (\ref{prop:DdtnabDPkp}), namely by expanding in Littlewood-Paley pieces as in (\ref{eq:expandParaDiff}).  

Here we improve on the approach in the proof of Proposition (\ref{prop:DdtnabDPkp}) to obtain the bound for the derivatives $D > 0$ in (\ref{ineq:bdForDdtRleqk}).  Namely, following the proof of Proposition (\ref{prop:DdtnabDPkp}) we could first differentiate in space to
\ali{
 \nab^D R_{\leq k, HH}^{jl} = \nab^D P_{\leq k} [( v^j - P_{\leq k} v^j ) (v^l - P_{\leq k}v^l ) ] 
}
and then expand the nonlinearity into pieces $P_I v^j P_{\approx I} v^l$ and commute with $(\pr_t + P_{\leq k} v \cdot \nab)$ to obtain (\ref{ineq:bdForDdtRleqk}).

Instead, we proceed by observing that we can write the operator $(\pr_t + P_{\leq k} v \cdot \nab) P_{\leq k}$ appearing in $(\pr_t + P_{\leq k} v \cdot \nab) R_{\leq k, HH}^{jl}$ as
\ali{ 
(\pr_t + P_{\leq k} v \cdot \nab) P_{\leq k} &= P_{\leq k + 2} (\pr_t + P_{\leq k} v \cdot \nab) P_{\leq k}
}
%for $I \geq k$ 
using the bandlimited property of Littlewood-Paley projections.  With this representation the bounds for the derivatives $D > 0$ in (\ref{ineq:bdForDdtRleqk}) follow from the $D = 0$ case as the spatial derivatives fall on the operator $P_{\leq k+2}$.  

In fact, as we will see later, it is helpful to observe some further cancellation using the bandlimited property.  Namely, following (\ref{eq:expandDdtHighHigh}) we expand
\ali{
(\pr_t + P_{\leq k} v \cdot \nab) R_{\leq k, HH}^{jl} &= P_{\leq k+2} (\pr_t + P_{\leq k} v \cdot \nab) P_{\leq k} [( v^j - P_{\leq k} v^j ) (v^l - P_{\leq k}v^l ) ] \\
&= \sum_{I \geq k} P_{\leq k+2} [ (\pr_t + P_{\leq I} v \cdot \nab) - P_{[k,I]} v \cdot \nab) ] P_{\leq k} ( P_I v^j P_{\approx I} v^l )
}
The bandlimited property of the Littlewood-Paley projections allows us to throw out the high frequency components in the term $P_{[k,I]} v \cdot \nab$ and obtain
\ali{
(\pr_t + P_{\leq k} v \cdot \nab) R_{\leq k, HH}^{jl} &= P_{\leq k+2} (\pr_t + P_{\leq k} v \cdot \nab) P_{\leq k} [( v^j - P_{\leq k} v^j ) (v^l - P_{\leq k}v^l ) ] \\
&= \sum_{I \geq k} P_{\leq k+2} [ (\pr_t + P_{\leq I} v \cdot \nab) - P_{[k,k+3]} v \cdot \nab) ] P_{\leq k} ( P_I v^j P_{\approx I} v^l ) \label{eq:formWeUseHighHigh}
}
The technique we have employed here to eliminate the high frequency components in the term $P_{[k,I]} v \cdot \nab$ is fairly subtle, but this cancellation will be crucial when estimating further advective derivatives of this term later on in Section~\ref{sec:highOrderAdvecForcing}. 

The bounds in (\ref{ineq:bdForDdtRleqk}) are now immediate from the form (\ref{eq:formWeUseHighHigh}).  As in the proof of Proposition (\ref{prop:DdtnabDPkp}), the most dangerous terms come from the operator $(\pr_t + P_{\leq I} v \cdot \nab)$, which always costs a factor of $2^{(1 - \a) I} \ctdcxa{v}$ in the estimate regardless of whether it falls on $P_{\approx I} v$, or whether it falls on the operator $P_{\leq k}$, giving rise to a commutator $[(\pr_t + P_{\leq I} v \cdot \nab), P_{\leq k}]$ of norm
\[ \sup_t || [(\pr_t + P_{\leq I} v \cdot \nab), P_{\leq k}] || \leq C 2^{(1 - \a) I} \ctdcxa{v}. \]
  The main advantage of the form (\ref{eq:formWeUseHighHigh}) is that this form facilitates commutator estimates for taking higher order material derivatives. 
\end{proof}

As an immediate corollary to the estimates (\ref{ineq:bdForDdtPkp}), (\ref{ineq:matdvnab2plqkp}) and (\ref{ineq:bdForDdtRleqk}), we are able to estimate second order material derivatives of Littlewood-Paley projections of the velocity.

\begin{prop} \label{prop:canTake2matDvsNow} If $1/3 < \a < 1$ and $(v,p)$ solve the incompressible Euler equations, then
\ali{
\co{ \nab^D (\pr_t + P_{\leq k} v \cdot \nab) P_{k+1} v } &\leq C_D 2^{(D + (1-\a) - \a) k} \ctdcxa{v}^2 \label{eq:basicFirstMatDvPkv}\\
\co{ \nab^D (\pr_t + P_{\leq k} v \cdot \nab)^2 P_{k+1} v } &\leq C_D 2^{(D + 2(1-\a) - \a) k} \ctdcxa{v}^3 \label{ineq:secondMatDvPkv} \\
\co{ \nab^D (\pr_t + P_{\leq k} v \cdot \nab) \nab P_{\leq k} v } &\leq C_{D,\T^n} 2^{(D + 2 (1-\a)) k} \ctdcxa{v}^2 \label{eq:basicFirstMatDvNabPlkv}\\
\co{ \nab^D (\pr_t + P_{\leq k} v \cdot \nab)^2 \nab P_{\leq k} v } &\leq C_{D,\T^n} 2^{(D + 3 (1-\a)) k} \ctdcxa{v}^3 \label{ineq:secondMatDvNabPlkv} \\
\co{ \nab^{D+1} (\pr_t + P_{\leq k} v \cdot \nab) P_{\leq k} v } &\leq C_{D,\T^n} 2^{(D + 2 (1-\a)) k} \ctdcxa{v}^2 \label{eq:basicFirstMatDvNabPlkvPrime}\\
\co{ \nab^{D+1} (\pr_t + P_{\leq k} v \cdot \nab)^2 P_{\leq k} v } &\leq C_{D,\T^n} 2^{(D + 3 (1-\a)) k} \ctdcxa{v}^3 \label{ineq:secondMatDvNabPlkvPrime}
}
\end{prop}
\begin{proof}
The bounds (\ref{eq:basicFirstMatDvPkv}), (\ref{eq:basicFirstMatDvNabPlkv}) and (\ref{eq:basicFirstMatDvNabPlkvPrime}) were already proven with no restriction on $\a$ by applying the basic estimates for Littlewood-Paley pieces of the pressure and the Reynolds stress to the equations
\begin{align}
\label{eq:euLPproj3}
 (\pr_t + P_{\leq k}v^j \pr_j)[P_{\leq k}v^l] + \pr^l P_{\leq k} p =& \pr_j R_{\leq k}^{jl}   \\
(\pr_t + P_{\leq k}v^j \pr_j)[\pr_a P_{\leq k}v^l] + \pr_a P_{\leq k}v^j \pr_j P_{\leq k}v^l + \pr_a \pr^l P_{\leq k} p =& \pr_a \pr_j R_{\leq k}^{jl} \label{eq:euNabLPproj3} \\
\label{eq:euLPpiece3}
(\pr_t + P_{\leq k}v^j \pr_j)[P_{k+1}v^l] + P_{k+1} v^j \pr_j P_{\leq k + 1} v^l + \pr^l P_{k + 1} p =& \pr_j ( R_{\leq k + 1}^{jl}  - R_{\leq k}^{jl} )   
\end{align}

The remaining estimates (\ref{ineq:secondMatDvPkv}) and (\ref{ineq:secondMatDvNabPlkv}) and the equivalent bounds (\ref{ineq:secondMatDvNabPlkv}) and (\ref{ineq:secondMatDvNabPlkvPrime}) follow by applying the operators $(\pr_t + P_{\leq k}v^j \pr_j)$ and $\nab^D$ to equations (\ref{eq:euLPproj3}) - (\ref{eq:euLPpiece3}).  The bounds (\ref{ineq:bdForDdtPkp}), (\ref{ineq:matdvnab2plqkp}) and (\ref{ineq:bdForDdtRleqk}) together with the basic estimates of Proposition (\ref{prop:firstLPpieceBds}) imply that the derivative $(\pr_t + P_{\leq k}v^j \pr_j)$ always costs a factor
\[ | (\pr_t + P_{\leq k} v \cdot \nab) | \leq C 2^{(1-\a)k} \ctdcxa{v} \]
in the estimates provided $\a > 1/3$.
\end{proof}

Later on we will generalize the estimates (\ref{ineq:bdForDdtPkp}), (\ref{ineq:matdvnab2plqkp}), (\ref{ineq:bdForDdtRleqk}) and Proposition (\ref{prop:canTake2matDvsNow}) to higher order material derivatives, but first we will study how the estimates obtained so far can be used to prove some higher order regularity in time for the pressure and the velocity.

\section{Second material derivatives of the pressure increments and applications} \label{sec:secondMatDvPressIncs}

Here we show that the estimates of Section (\ref{sec:forcingTerms1}) can be applied to give higher regularity in time for the pressure and the pressure gradient.  Our focus will be on proving the following theorems

\begin{thm} \label{thm:theMatRegWeProveNow}
For $1/2 < \a < 1$, we have that 
\ali{
 (\pr_t p, \nab p) &\in C_{t,x}^\b  \label{in:nabPHoldReg}
}
for all $\b < 2 \a - 1$.

If $2/3 < \a < 1$, then $\fr{D }{\pr t} \nab p = \pr_t \nab p + \tx{{\normalfont div} }( v \otimes \nab p )$, which is well-defined as a distribution by (\ref{in:nabPHoldReg}), is also H\"{o}lder continuous, with
\ali{
 \fr{D }{\pr t} \nab p &\in C_{t,x}^\b  \label{eq:DdtNbpctb}
}
for all $\b < 3 \a - 2$.  If $3/4 < \a < 1$, we have furthermore that
\ali{
 \fr{D^2}{\pr t^2} \nab p = \pr_t \fr{D }{\pr t} \nab p + \pr_j( v^j \fr{D }{\pr t} \nab p )   &\in C^0 \label{eq:Ddt2nbpc0}
}
\end{thm}
The methods we exhibit can be extended to establish \eqref{eq:DdtNbpctb} and \eqref{eq:Ddt2nbpc0} with $\pr_t p$ in place of $\nab p$, although this extension requires control over material derivative estimates of third order that are not proven in this section.

To provide a further application of the second order material derivative estimates, we will prove a regularity theorem for $\fr{D}{\pr t} p$ and $\fr{D^2}{\pr t^2} p$ in Theorem~(\ref{thm:2timeRegp}) below.

The proof of Theorem (\ref{thm:theMatRegWeProveNow}), which is contained in Section (\ref{sec:pressIncPressGrad}) below, proceeds by estimating first and second material derivatives of the pressure increments defined in Section (\ref{sec:timeRegPress}).  In order to establish the necessary estimates, we will start by proving some preliminary bounds for higher order commutators between material derivatives and the relevant convolution operators.

At this point we will no longer keep track of which constants depend on the torus $\T^n$. %, so for simplicity we will no longer keep track of this dependence.
%At this point many of the constants depend on the torus $\T^n$, so for simplicity we will no longer keep track of this dependence.

\subsection{Second order commutator estimates for material derivatives and convolution operators} \label{sec:commuteWithConv}

In this Section, we show how our bounds from the Euler equations can be used to bound the operator norms of second and third order commutators between coarse scale material derivatives and convolution operators.  

We start by introducing some further notation.  We will continue to use the notation
\[ \fr{D_{\lk}}{\pr t}  = (\pr_t + P_{\leq k} v \cdot \nab) \]
and
\[ \fr{D^2_{\lk}}{\pr t^2}  = (\pr_t + P_{\leq k} v \cdot \nab)^2. \]
The expression $[X,]^rT$ will denote the operator obtained by commuting $T$ with $X$ repeatedly $r$ times (e.g. $[X,]^2T =[X,[X,T]]$).

The following proposition describes the general estimate one has available for commutators of coarse scale material derivatives with operators of convolution form.  A main example to keep in mind is the operator $T = \De^{-1} \nab^2 P_{\leq k}$ which appears in the definition of the pressure increments introduced in Section (\ref{sec:timeRegPress}).

\begin{prop}[Commuting material derivatives and smoothing operators]\label{prop:commutingWithOperators}
Suppose that $(v,p)$ solve the incompressible Euler equations, $0 < \a < 1$.  Suppose that $T$ takes the form
\[ T f = \int_{\R^n} f(x+h) K(h) dh \]
and that the kernel $K$ satisfies the estimates
\ali{
\| \nab^D K \|_{L^1_h} + \| |h| \nab^{1+D} K \|_{L^1_h} + \| |h|^2 \nab^{2+D} K \|_{L^1_h} &\leq 2^{D k} \label{ineq:assumedKBounds}
}
for all $0 \leq D \leq M$.
Then we have the estimates
\ali{
\sup_t \left\| \nab^D \left[\fr{D_{\leq I}}{\pr t},\right]^r T \right\| &\leq C_D 2^{D \max \{ k, I \} } \cdot 2^{r(1-\a) I} \ctdcxa{v}^r \label{ineq:genCommuteOperator1}
}
for all $0 \leq D \leq M$ and all $0 \leq r \leq 2$.

If $\a > 1/3$ and we also assume for $0 \leq D \leq M$ that
\ali{
 \| |h|^3 \nab^{3+D} K \|_{L^1_h} &\leq 2^{D k}, \label{ineq:thirdMomentK}
}
then (\ref{ineq:genCommuteOperator1}) holds as well for $r = 3$.
\end{prop}
We remark that there is actually no need to invoke the Euler equations for the cases $r = 0,1$ in Proposition (\ref{prop:commutingWithOperators}).
\begin{proof}
%First note that the operators $T = \De^{-1} \nab^A P_{\approx k}$ and $\De^{-1} \nab^{2+A} P_{\leqc k}$ are all of convolution form
%\[ T f(x) = K \ast f(x) = \int_{\R^n} f(x - h) K(h) dh \]
%for the appropriate kernel $K$.  In every case we have considered, $K(h) = K(-h)$ is also even, so the minus can be changed to a plus.  In this case, we have shown the basic commutator formula
First recall the basic commutator formulas
\ali{
\left[\fr{D_{\leq I}}{\pr t}, K \ast \right] f(x) &= \int_{\R^n} ( P_{\leq I} v^i(x + h) - P_{\leq I} v^i(x) ) f(x+h)\pr_i K(h) dh\label{eq:commuteintByParts1prime} \\
&= \int_0^1 \int_{\R^n} \pr_a P_{\leq I} v^i(x + s h)  f(x+h) \pr_i K(h) h^a dh ds \label{eq:commuteintByParts2prime}
}
from (\ref{eq:commuteintByParts2}) obtained using integration by parts.  Also recall that these formulas have been simplified using the fact that $P_{\leq I} v$ is divergence free, but the extra term that would arise otherwise would in any case have a similar form and obey the same bounds.\footnote{However, this remark does allow us to remove the term $\|\nab^{D} K\|_{L^1_h}$ in the assumption (\ref{ineq:assumedKBounds}) when $r \geq 1$.}

The bounds (\ref{ineq:commutForPk})-(\ref{ineq:commutForPlqk}) for $r = 0,1$ are then immediate from the form (\ref{eq:commuteintByParts2prime}) without using the Euler equations.  When estimating $\| \nab^D \left[\fr{D_{\leq I}}{\pr t}, K \ast \right] \|$ we do not let the derivatives hit the function $f$, but rather integrate by parts using the form (\ref{eq:commuteintByParts2prime}).  %The point here is that the distance $h$ is essentially bounded for the kernels associated to (\ref{ineq:commutForPk})-(\ref{ineq:commutForPlqk}) acting on the torus.

%The most difficult case is the case $K = \sum_{I = k_0(\T^n)}^k $  (\ref{ineq:commutForPlqk})

We will now commute again with the operator $\fr{D_{\leq I}}{\pr t} = (\pr_t + P_{\leq I} v \cdot \nab)$.  Here it is important to use the form (\ref{eq:commuteintByParts1prime}) rather than (\ref{eq:commuteintByParts2prime}), since we do not always have the control we would require over $\| |h|^2 \nab K \|_{L^1}$ when $K$ is a long range kernel such as $\De^{-1} \nab^2 \eta_{\leq k}$.  On the other hand, it is safe to absorb extra powers of $h$ onto the kernel when $K$ is short range, such as $\De^{-1} \nab^2 \eta_{k}$.

We express the second commutator with $\fr{D_{\leq I}}{\pr t} = (\pr_t + P_{\leq I} v \cdot \nab)$ as follows
\ali{
\left[\fr{D_{\leq I}}{\pr t},  \right]^2 &K \ast[f](x) = (\pr_t + P_{\leq I} v^{i_2}(t,x) \fr{\pr~}{\pr x^{i_2}} ) \int_{\R^n} ( P_{\leq I} v^{i_1}(x + h) - P_{\leq I} v^{i_1}(x) ) f(t,x+h) \pr_{i_1} K(h) dh \label{eq:firstStageOfDdtComm2} \\
&- \int_{\R^n} ( P_{\leq I} v^{i_1}(x + h) - P_{\leq I} v^{i_1}(x) )  \left(\pr_t + P_{\leq I} v^{i_2}(t,x+h) \fr{\pr~}{\pr x^{i_2}}\right) f(t,x+h) \pr_{i_1} K(h) dh \label{eq:firstStageOfDdtComm3}
}
The resulting operator acts on $f$ only in the spatial variables, and only on a fixed time slice.  %There is, however, a term of the form $\fr{D_{\leq I}}{\pr t} \nab P_{\leq I}$ coming from (\ref{eq:firstStageOfDdtComm2}) which we will control using the estimates from the Euler equations.  
A full expansion of (\ref{eq:firstStageOfDdtComm2})-(\ref{eq:firstStageOfDdtComm3}) gives
\ali{
\left[\fr{D_{\leq I}}{\pr t},  \right]^2 &K \ast[f](x) = T_{(I)}[f] - T_{(II)}[f] - T_{(III)}[f] \label{eq:goodDecomposition} \\
T_{(I)}[f] &= \int_{\R^n} \left(\fr{D_{\leq I}}{\pr t}P_{\leq I} v^{i_1}(t,x + h) - \fr{D_{\leq I}}{\pr t}P_{\leq I} v^{i_1}(t,x)\right)   f(t,x+h) \pr_{i_1} K(h) dh \label{eq:TIOfThisCommutator}\\
T_{(II)}[f] &= \int_{\R^n} \left(P_{\leq I} v^{i_2}(t,x + h) -  P_{\leq I} v^{i_2}(t,x)\right)\pr_{i_2} P_{\leq I} v^{i_1}(t,x + h)  f(t,x+h) \pr_{i_1} K(h) dh \label{eq:TIIOfThisCommutator} \\
T_{(III)}[f] &= {\int_{\R^n}}\left( P_{\leq I} v^{i_1}(x + h) {-} P_{\leq I} v^{i_1}(x) \right)  \left(P_{\leq I} v^{i_2}(t,x+h) {-} P_{\leq I} v^{i_2}(t,x)\right) \pr_{i_2}f(t,x+h) \pr_{i_1} K(h) dh \label{eq:thirdLineSecondCommutator}
}
As we have seen before, the main observation here which confirms that the commutator is indeed a smoothing operator is that we can integrate by parts in the $h$ variables when the derivative hits the $f$ in (\ref{eq:thirdLineSecondCommutator}) by noticing that $\pr_{i_2}f(t,x+h) = \fr{\pr f}{\pr x^{i_2} } = \fr{\pr f}{\pr h^{i_2} }$.  The resulting expression can then be simplified by observing that
\[\fr{\pr ~}{\pr h^{i_2} } \left(P_{\leq I} v^{i_2}(t,x+h) - P_{\leq I} v^{i_2}(t,x)\right) = 0 \]
from the fact that $v$ is divergence free, but this observation is not important for the estimates, since a nonzero term of the same type will appear.  Performing this integration by parts gives
\ali{
T_{(III)}[f] &= - T_{(III, 1)}[f] - T_{(III,2)}[f] \label{eq:expandTIIIpieces}\\
T_{(III, 1)}[f] &= \int_{\R^n} \pr_{i_2}P_{\leq I} v^{i_1}(x + h)  \left(P_{\leq I} v^{i_2}(t,x+h) - P_{\leq I} v^{i_2}(t,x)\right) f(t,x+h) \pr_{i_1} K(h) dh \\
T_{(III,2)}[f] &= {\int_{\R^n}}\left( P_{\leq I} v^{i_1}(x + h) {-} P_{\leq I} v^{i_1}(x) \right)  \left(P_{\leq I} v^{i_2}(t,x+h) {-} P_{\leq I} v^{i_2}(t,x)\right) f(t,x+h) \pr_{i_2} \pr_{i_1} K(h) dh \label{eq:theQuadraticTermInCommut2}
}
We are now able to read off the bound (\ref{ineq:genCommuteOperator1}) using the estimates of Proposition (\ref{prop:canTake2matDvsNow}).  The main step here is to apply the fundamental theorem of calculus to every term that has the difference form
\ali{
 \de_h F(x) &= F(x+h) - F(x)  \label{eq:dehform} \\
&= \int_0^1 \pr_a F(x + s h) h^a ds \notag
}
and then absorb the factor of $h$ into the kernel.  For example, we have
\ali{
T_{(I)}[f] &= \int_{\R^n} \de_h \fr{D_{\leq I}}{\pr t}P_{\leq I} v^{i_1}(t,x) f(t,x+h) \pr_i K(h) dh \\
&= \int_0^1 \int_{\R^n} \pr_{a_1} \fr{D_{\leq I}}{\pr t}P_{\leq I} v^{i_1}(t,x + sh) f(t,x+h) \pr_i K(h) h^{a_1} dh ds
}
and
\ali{
T_{(III,2)}[f] &= {\int_{\R^n}}\de_h P_{\leq I} v^{i_1}(t,x) \de_h P_{\leq I} v^{i_2}(t,x) f(t,x+h) \pr_{i_2} \pr_{i_1} K(h) dh \\
&= {\int_0^1 \int_0^1 \int_{\R^n}} \pr_{a_1} P_{\leq I} v^{i_1}(t,x{+}s_1 h) \pr_{a_2} P_{\leq I} v^{i_2}(t,x{+}s_2 h) f(t,x{+}h) \pr_{i_2} \pr_{i_1} K(h) h^{a_1} h^{a_2} {dh ds_1 ds_2} \label{eq:alreadyHasTwoDeh}
}
The bounds (\ref{ineq:genCommuteOperator1}) now follow from Proposition (\ref{prop:canTake2matDvsNow}) and the bounds (\ref{ineq:assumedKBounds}) assumed for $K$.  As one would expect, to bound derivatives $\nab^D\left[\fr{D_{\leq I}}{\pr t},  \right]^2 K \ast[f]$ one must always integrate by parts in the $h$ variable when the derivative hits the function $f(t,x+h)$.

Now that we have obtained a good expansion
\ali{
\left[\fr{D_{\leq I}}{\pr t},  \right]^2 &K \ast[f](x) = T_{(I)} - T_{(II)} + T_{(III, 1)} + T_{(III,2)}
}
from (\ref{eq:TIOfThisCommutator}), (\ref{eq:TIIOfThisCommutator}) and (\ref{eq:expandTIIIpieces}), to prepare for further estimates it is worthwhile to observe that the structure of the commutator survives to allow estimates for higher order commutators after introducing one more trick.  %Perhaps surprisingly, we will actually need one more trick to handle the third commutator, but once this trick is illustrated it will be clear how these commutator estimates work in general.  

In what follows, we will often suppress the dependence in $t$ of all the terms; however, every tensor field that appears besides the kernel $K$ depends on time.

Most of the terms that arise in the expansion of $\left[\fr{D_{\leq I}}{\pr t},  \right]^3 K \ast[f](x)$ are estimated by techniques we have already used for the first two commutators.  We will focus on the term 
\[ T_{(II)} = \int_{\R^n} \de_h P_{\leq I} v^{i_2}(t,x)\pr_{i_2} P_{\leq I} v^{i_1}(t,x + h)  f(t,x+h) \pr_{i_1} K(h) dh \] from (\ref{eq:TIIOfThisCommutator}) since this term requires one additional trick to estimate, while the other terms are treated similarly.

To begin, we expand
\ali{
\left[\fr{D_{\leq I}}{\pr t},\right]&T_{(I)}[f](x){=} \left(\pr_t {+} P_{\leq I} v^{i_3}(t,x) \fr{\pr~}{\pr x^{i_3}} \right) \int_{\R^n} \de_h P_{\leq I} v^{i_2}(t,x)\pr_{i_2} P_{\leq I} v^{i_1}(t,x {+} h)  f(t,x{+}h) \pr_{i_1} K(h) dh \notag \\
&- \int_{\R^n} \de_h P_{\leq I} v^{i_2}(t,x)\pr_{i_2} P_{\leq I} v^{i_1}(t,x + h)  \fr{D_{\leq I}}{\pr t}f(t,x+h) \pr_{i_1} K(h) dh \\
&= T_{(II, A)} + T_{(II, B)} \\
T_{(II, A)} &= \int_{\R^n} \left[ \left(\pr_t {+} P_{\leq I} v^{i_3}(t,x) \fr{\pr~}{\pr x^{i_3}} \right)\de_h P_{\leq I} v^{i_2}(t,x) \right]\pr_{i_2} P_{\leq I} v^{i_1}(t,x {+} h)  f(t,x{+}h) \pr_{i_1} K(h) dh \label{eq:termTIIA}\\
T_{(II, B)} &= \int_{\R^n} \de_h P_{\leq I} v^{i_2}(t,x)(\pr_t + P_{\leq I} v^{i_3}(t,x) \fr{\pr~}{\pr x^{i_3}} ) \left[\pr_{i_2} P_{\leq I} v^{i_1}(t,x + h)  f(t,x+h)\right] \pr_{i_1} K(h) dh \notag \\
&- \int_{\R^n} \de_h P_{\leq I} v^{i_2}(t,x)\pr_{i_2} P_{\leq I} v^{i_1}(t,x + h)  \fr{D_{\leq I}}{\pr t}f(t,x+h) \pr_{i_1} K(h) dh  \label{eq:termTIIB}
}

Whenever we encounter a term of the form $\de_h F(x) = F(x+h) - F(x)$ as in (\ref{eq:termTIIA}), we always commute the material derivative and the difference operator $\de_h$ as in
\ali{
(\pr_t + P_{\leq I} v^i(t,x) \pr_i) \de_h F(x) &= \de_h[(\pr_t + P_{\leq I} v \cdot \nab) F](x) - \de_h P_{\leq I} v^i(x) \pr_i F(x+h)
}
For term (\ref{eq:termTIIA}), this operation gives rise to two more terms
\ali{
T_{(II, A)} &= T_{(II, A1)} - T_{(II, A2)} \\
T_{(II, A1)} &= \int_{\R^n} \de_h \fr{D_{\leq I}}{\pr t}P_{\leq I} v^{i_2}(t,x) \pr_{i_2} P_{\leq I} v^{i_1}(t,x {+} h)  f(t,x{+}h) \pr_{i_1} K(h) dh \\
T_{(II, A2)} &= \int_{\R^n} \de_h P_{\leq I} v^{i_3}(t,x) \pr_{i_3} P_{\leq I}v^{i_2}(t,x{+}h)  \pr_{i_2} P_{\leq I} v^{i_1}(t,x {+} h)  f(t,x{+}h) \pr_{i_1} K(h) dh
}
which are controlled by Proposition (\ref{prop:canTake2matDvsNow}).

For the term $T_{(II, B)}$ there are derivatives on the function $f(t,x+h)$, so it will be necessary to integrate by parts in order to control this term.  It is important to be careful how this integration by parts is executed, since a naive application of the product rule for the advective derivative in (\ref{eq:termTIIB}) will lead to terms such as
\[ \int_{\R^n} \de_h v^{i_2}(t,x) \de_h v^{i_3}(t,x) \pr_{i_3} \pr_{i_2} P_{\leq I} v^{i_1}(t,x{+}h)  f(t,x{+}h) \pr_{i_1} K(h) dh  \] 
which cannot be controlled for long-range kernels, since we have not assumed control of $\| |h|^2 \nab K \|_{L^1_h}$.

To avoid seeing such terms, %we first apply the observation that
%\[\pr_{i_2} P_{\leq I} v^{i_1}(t,x {+} h) \fr{D_{\leq I}}{\pr t} f(t,x{+}h) = \fr{D_{\leq I}}{\pr t}[\pr_{i_2} P_{\leq I} v^{i_1} f](t,x{+}h) -  \fr{D_{\leq I}}{\pr t}\pr_{i_2} P_{\leq I} v^{i_1}(t,x {+} h) f(t,x{+}h)  \]
we expand $T_{(II, B)}[f]$ in a way that keeps the product term \[ \pr_{i_2} P_{\leq I} v^{i_1}(t,x {+} h) f(t,x{+}h)\] in tact
\ali{
T_{(II,B)} &= - T_{(II,B1)} + T_{(II,B2)} \\
T_{(II,B1)} &= \int_{\R^n} \de_h P_{\leq I} v^{i_2}(x) \de_h P_{\leq I} v^{i_3}(x) \pr_{i_3}\left[\pr_{i_2} P_{\leq I} v^{i_1}(x + h)  f(x+h)\right] \pr_{i_1} K(h) dh  \label{eq:nowCanIntByParts}\\
T_{(II,B2)} &= \int_{\R^n} \de_h P_{\leq I} v^{i_2}(x) \left[ \fr{D_{\leq I}}{\pr t}[\pr_{i_2}P_{\leq I} v^{i_1}f](x{+}h) -  \pr_{i_2} P_{\leq I} v^{i_1}(x{+} h)  \fr{D_{\leq I}}{\pr t} f(x{+}h) \right] \pr_{i_1} K(h) dh  \\
&= \int_{\R^n} \de_h P_{\leq I} v^{i_2}(x) \left[\fr{D_{\leq I}}{\pr t}\pr_{i_2}P_{\leq I} v^{i_1}(x{+}h) f(x{+}h) \right] \pr_{i_1} K(h) dh 
\label{eq:termTIIB2OK}
}
The term (\ref{eq:termTIIB2OK}) is under control by the bounds of Proposition (\ref{prop:canTake2matDvsNow}).  For the term (\ref{eq:nowCanIntByParts}), we observe that the derivative $\pr_{i_3} = \fr{\pr~}{\pr x^{i_3}} = \fr{\pr~}{\pr h^{i_3}}$ can be viewed as a derivative in the $h$ variables, which allows us to integrate by parts to expand
\ali{
T_{(II,B1)} &= - T_{(II,B1a)} - T_{(II,B1b)} \\
T_{(II,B1a)} &= \int_{\R^n} \de_h P_{\leq I} v^{i_2}(x)\left[  \pr_{i_3} P_{\leq I} v^{i_3}(x+h) \pr_{i_2} P_{\leq I} v^{i_1}(x + h)  f(x+h)\right] \pr_{i_1} K(h) dh \label{eq:termIsActually0}\\
&+ \int_{\R^n} \de_h P_{\leq I} v^{i_3}(x) \left[ \pr_{i_3} P_{\leq I} v^{i_2}(x+h) \pr_{i_2} P_{\leq I} v^{i_1}(x + h)  f(x+h)\right] \pr_{i_1} K(h) dh \\
T_{(II,B1b)} &= \int_{\R^n} \de_h P_{\leq I} v^{i_2}(x) \de_h P_{\leq I} v^{i_3}(x) \left[\pr_{i_2} P_{\leq I} v^{i_1}(x + h)  f(x+h)\right] \pr_{i_3} \pr_{i_1} K(h) dh
}
These terms are under control by Proposition (\ref{prop:canTake2matDvsNow}) after the factors of $\de_h P_{\leq I} v$ are expressed using the Fundamental Theorem of Calculus.  (Of course, the term (\ref{eq:termIsActually0}) actually vanishes.)

The estimates for the terms $\left[\fr{D_{\leq I}}{\pr t}, T_{(I)}\right][f]$ and $\left[\fr{D_{\leq I}}{\pr t}, T_{(III)}\right][f]$ from the decomposition (\ref{eq:goodDecomposition}) involve the same techniques, but also use additional assumptions.  The assumption that $\a > 1/3$ comes into play in order to estimate the second material derivative in the term
\[ 
\int_{\R^n} \de_h \fr{D_{\leq I}^2}{\pr t^2}P_{\leq I} v^{i_1}(x) f(x+h) \pr_{i_1} K(h) dh
 \]
which arises in the expansion of $\left[\fr{D_{\leq I}}{\pr t}, \right]T_{(I)}[f]$ after commuting the advective derivative with $\de_h$.  The assumption (\ref{ineq:thirdMomentK}) on $\| |h|^3 \nab^3 K \|_{L^1}$ comes into play in order to estimate the term %third derivative of the kernel that appears in the term 
\[ \int_{\R^n} \de_h P_{\leq I} v^{i_1}(x) \de_h P_{\leq I} v^{i_2}(x) \de_h P_{\leq I} v^{i_3}(x) f(x{+}h) \pr_{i_3} \pr_{i_2} \pr_{i_1} K(h) dh, \]
which arises in the expansion of $\left[\fr{D_{\leq I}}{\pr t}, \right]T_{(III,2)}[f]$ in (\ref{eq:alreadyHasTwoDeh}) after integrating by parts.
\end{proof}

As a corollary of Proposition~\ref{prop:commutingWithOperators}, we have the following commutator estimates:
\begin{lem}\label{lem:weNeedTheseCommutators2}
Under the assumptions and notation of Proposition (\ref{prop:commutingWithOperators}), we have the estimates
\ali{
\sup_t \left\| \nab^D \left[\fr{D_{\leq I}}{\pr t},\right]^r (\De^{-1} \nab^A P_k) \right\| &\leq C_{D,A} 2^{(D + r(1-\a)) I} \cdot 2^{(A - 2)k} \ctdcxa{v}^r \label{ineq:commutForPk} \\
\sup_t \left\| \nab^D \left[\fr{D_{\leq I}}{\pr t},\right]^r (\nab^A P_{\lk}) \right\| &\leq C_{D,A} 2^{(D + r(1-\a)) I} \cdot 2^{A k} \ctdcxa{v}^r \label{ineq:commutForPlk}
}
for $r = 0, 1, 2$, all integers $k$ and all $I \geq k$.  Also, for $k \geq k_0(\T^n)$ we have
\ali{
\sup_t \left\| \nab^D \left[\fr{D_{\leq I}}{\pr t},\right]^r (\De^{-1} \nab^{2+A} P_{\leq k}) \right\| &\leq C_{D,A} (1 + |k - k_0(\T^n)|) 2^{(D + r(1-\a)) I} \cdot 2^{A k} \ctdcxa{v}^r \label{ineq:commutForPlqk}
}
If $1/3 < \a < 1$, all the above estimates hold as well for $r = 3$.  The estimate also holds with a different constant if $P_{\leq k}$ is replaced by a comparable projection $P_{\leqc k} = P_{\leq k + a}$ provided that $|a|$ is bounded.  Similarly, one can replace $P_k$ in (\ref{ineq:commutForPk}) with any comparable operator
\[ P_{\approx k} = P_{[k_1, k_2]} \]
provided $|k_1 - k|$ and $|k_2 - k|$ are bounded.
\end{lem}

These estimates all follow from Proposition (\ref{prop:commutingWithOperators}) after multiplying the operators by the appropriate constant.  For example, Proposition (\ref{prop:commutingWithOperators}) applies to $T = C^{-1} (1 + | k - \kotn|)^{-1} \De^{-1}  \nab^2 P_{\leq k} )$ and $T = C^{-1} 2^{(2 - A)k}\De^{-1} \nab^A P_{k} )$ if $C$ is sufficiently large.

\subsection{Second order material derivatives of pressure increments and regularity in time for the pressure gradient } \label{sec:pressIncPressGrad}

With the necessary commutator estimates in hand, we now begin the proof of Theorem (\ref{thm:theMatRegWeProveNow}).  As in the proof of Corollary (\ref{cor:thm13}), the proof will proceed by considering the pressure increments defined in Section (\ref{sec:timeRegPress}).  From this point onward we will no longer record the dependence of the constants on $\T^n$ and the fixed $\a < 1$.

The main Lemma that enables us to access higher order advective derivatives despite being unable to differentiate the velocity field itself is the following fact.
\begin{lem} \label{lem:weakConverge} Suppose $T$ is a continuous tensor field on $I \times \T^n$ and $v$ is a continuous vector field on $I \times \T^n$.  Then if $v_{(k)} \to v$ and $T_{(k)} \to T$ uniformly as $k \to \infty$, then we have weak convergence
\ali{
 \pr_t T_{(k)} + \pr_j( v_{(k)}^j T_{(k)} ) \rightharpoonup \pr_t T + \pr_j( v^j T ) \label{weakConvT}
}
in the sense of distributions.
\end{lem}
Our strategy for proving regularity for higher order material derivatives such as $\pr_t \nab p + \pr_j(v^j \nab p) = \fr{D^2 v}{\pr t}$ will be to first apply Lemma (\ref{lem:weakConverge}) with $v_{(k)} = P_{\lk} v$ and, say, $T_{(k)} = \nab p_{(k)} \to \nab p$, and then to upgrade the weak convergence in (\ref{weakConvT}) to convergence in H\"{o}lder spaces.  

We start with some preliminary estimates for the pressure increments.
\begin{prop} \label{prop:pressIncSummary2matdv} Let $\de p_{(k)}$ and $p_{(k)} = \sum_{I = k_0(\T^n)}^k \de p_{(I)}$ be defined as in Section (\ref{sec:timeRegPress}).  If $1/3 < \a < 1$ and $0 \leq r \leq 2$, the following bounds hold
\ali{
\co{ \nab^D  \fr{D_{\lk}^r}{\pr t^r} \de p_{(k)} } &\leq C_D (1 + | k - k_0(\T^n)|) 2^{(D + r(1-\a) - 2 \a) k} \ctdcxa{v}^{2+r} \label{ineq:2matdvdepk} \\
\co{ \nab^D  \fr{D_{\lk}^r}{\pr t^r} \nab^2 p_{(k)} } &\leq C_D (1 + | k - k_0(\T^n)|) 2^{(D + 2 + r(1-\a) - 2 \a) k} \ctdcxa{v}^{2+r} \label{ineq:2matdvnab2pk} \\
\co{ \nab^{D+1}  \fr{D_{\lk}^r}{\pr t^r} \nab p_{(k)} } &\leq C_D (1 + | k - k_0(\T^n)|) 2^{(D + 2 + r(1-\a) - 2 \a) k} \ctdcxa{v}^{2+r} \label{ineq:2matdvnabpk}
}
\end{prop}

\begin{proof}
The bounds (\ref{ineq:2matdvdepk}) and (\ref{ineq:2matdvnab2pk}) stated in Proposition (\ref{prop:pressIncSummary2matdv}) for $r = 0, 1$ were already established in Section (\ref{sec:timeRegPress}) without any assumptions on $\a$.  The estimate (\ref{ineq:2matdvnabpk}) is equivalent to (\ref{ineq:2matdvnab2pk}) when $r = 0$ and follows from (\ref{ineq:2matdvnab2pk}) and (\ref{ineq:2matdvnabpk}) by induction from the cases $r = 0, 1$ after commuting the spatial and material derivatives.  It therefore suffices to prove (\ref{ineq:2matdvdepk}) and (\ref{ineq:2matdvnab2pk}) for $r = 2$.

We start with (\ref{ineq:2matdvnab2pk}) since this quantity involves the least number of terms and suffices to illustrate all the main ideas.  We begin by writing
\ali{
\nab^2 p_{(k)} &= \De^{-1} \nab^2 P_{\leq k}[ \pr_l P_{\lk} v^j \pr_j P_{\lk} v^l ] \\
\fr{D_{\lk}}{\pr t} \nab^2 p_{(k)} &= \left[\Dlkdt, \De^{-1} \nab^2 P_{\leq k} \right] [ \pr_l P_{\lk} v^j \pr_j P_{\lk} v^l ] \label{eq:commutermDlkprtnb2pk}\\
&+ \De^{-1} \nab^2 P_{\leq k}\left[ \Dlkdt[\pr_l P_{\lk} v^j] \pr_j P_{\lk} v^l + \pr_l P_{\lk} v^j \Dlkdt[\pr_j P_{\lk} v^l] \right] 
}
One should regard the differentiation above as an application of the ``product rule'' for three terms, where the commutator term (\ref{eq:commutermDlkprtnb2pk}) is what arises when $\Dlkdt$ ``hits'' the operator $\De^{-1} \nab^2 P_{\leq k}$.  Taking a second material derivative gives a representation
\ali{
\fr{D_{\lk}^2}{\pr t^2} \nab^2 p_{(k)} &= \left[\Dlkdt, \left[\Dlkdt, \De^{-1} \nab^2 P_{\leq k} \right]\right] \left( \pr_l P_{\lk} v^j \pr_j P_{\lk} v^l \right) \\
&+ 2 \left[\Dlkdt, \De^{-1} \nab^2 P_{\leq k} \right]\left( \Dlkdt[\pr_l P_{\lk} v^j \pr_j P_{\lk} v^l]\right) \\
&+  \De^{-1} \nab^2 P_{\leq k} \left(\fr{D_{\lk}^2}{\pr t^2}[\pr_l P_{\lk} v^j \pr_j P_{\lk} v^l]\right) \label{eq:termWithMostMatdv2}
}
Since we have assumed $\a > 1/3$, Proposition (\ref{prop:canTake2matDvsNow}) now guarantees that the terms
\ali{
\left\co{  \fr{D_{\lk}^r}{\pr t^r}[\pr_l P_{\lk} v^j \pr_j P_{\lk} v^l ] \right} &\leq C_D 2^{(r(1-\a) + 2(1-\a))k} \ctdcxa{v}^{2+r}
}
obey the expected bounds for $r = 0, 1, 2$.  The bound (\ref{ineq:2matdvnab2pk}) follows from the Lemma (\ref{lem:weNeedTheseCommutators2}), which guarantees estimates for the smoothing operators
\ali{
\sup_t \left\| \nab^D \left[\Dlkdt,\right]^r \De^{-1} \nab^2 P_{\leq k} \right\| &\leq C_{D, \T^n} (1 + |k - k_0(\T^n)|) 2^{(D + r(1-\a))k} \ctdcxa{v}^r
}
for $r = 0, 1, 2$.  Here we have used the notation $[X,]^rT$ from Section (\ref{sec:commuteWithConv}) to denote operator obtained by commuting $T$ with $X$ repeatedly $r$ times.

The proof of estimate (\ref{ineq:2matdvdepk}) is essentially the same, drawing from the decomposition
\ali{
\de p_{(k)} &= \de p_{(k), LL} + \de p_{(k), HL} + \de p_{(k), HH} 
}  
from (\ref{de:dePLLshort})-(\ref{eq:dePHHshort}), but with two new features.  First, there are high frequency terms similar to $\fr{D_{\lk}^r}{\pr t^r}P_{k + 1} v$, which are bounded using Proposition (\ref{prop:canTake2matDvsNow}).  There are also operators which project to high frequencies of the type
\[ \nab^D \left[\Dlkdt,\right]^r (\De^{-1} \nab^A P_{\approx k}) \] 
with $r = 0, 1, 2$ and $A = 0, 1$, which are bounded by
\[ \sup_t \left\| \nab^D \left[\Dlkdt,\right]^r (\De^{-1} \nab^A P_{\approx k}) \right\| \leq C_D 2^{(D + r(1-\a) + A)k} \ctdcxa{v}^r \] 
according to Lemma (\ref{lem:weNeedTheseCommutators2}).
\end{proof}

We now define the frequency increments for the material derivative of the pressure gradient.
\ali{
\de_{(k)} \fr{D_{\lk}}{\pr t} \nab p_{(k)} &= \fr{D_{\leq k+1}}{\pr t} \nab p_{(k+1)} - \fr{D_{\leq k}}{\pr t} \nab p_{(k)} \\
&= P_{k+1} v \cdot \nab \nab p_{(k+1)} + \fr{D_{\leq k}}{\pr t} \nab \de p_{(k)}  \label{eq:freqIncMatdvNabp}
}
According to Lemma (\ref{lem:weakConverge}), we have
\ali{
\pr_t \nab p + \pr_j( v^j \nab p) &= \sum_{k = \kotn}^\infty \de_{(k)} \fr{D_{\lk}}{\pr t} \nab p_{(k)}
}
as distributions whenever $\a > 1/2$.  Our aim is to prove that the summation converges in the appropriate H\"{o}lder norms when $\a > 2/3$.  This convergence will follow from the following estimates, which are an immediate consequence of Proposition (\ref{prop:pressIncSummary2matdv}) and the formula (\ref{eq:freqIncMatdvNabp}).  The important point to observe is that the low frequency parts of $p_{(k)}$ always appear with at least two derivatives so that the bounds of Proposition (\ref{prop:pressIncSummary2matdv}) apply.
\begin{cor}[Bounds for frequency increments of $\nab p$ and $\fr{D}{\pr t} \nab p$]\label{cor:freqIncNabpMatNabp}  If $1/3 < \a < 1$, then
\ali{
\co{ \nab^D \fr{D_{\lk}}{\pr t} \nab \de p_{(k)}} &\leq C_D (1 + |k - \kotn| )2^{(D + (1 - \a) + (1 - 2 \a )) k} \ctdcxa{v}^3 \\
\co{ \nab^D \de_{(k)} \fr{D_{\lk}}{\pr t} \nab p_{(k)} } &\leq C_D (1 + |k - \kotn| )2^{(D + (1 - \a) + (1 - 2 \a )) k} \ctdcxa{v}^3 \\ 
\co{ \nab^D \fr{D_{\lk}}{\pr t} \de_{(k)} \fr{D_{\lk}}{\pr t} \nab p_{(k)} } &\leq C_D (1 + |k - \kotn| )2^{(D + 2(1 - \a) + (1 - 2 \a) ) k} \ctdcxa{v}^4
}
\end{cor}

We can now prove Theorem (\ref{thm:theMatRegWeProveNow}).
\begin{proof}[Proof of Theorem (\ref{thm:theMatRegWeProveNow})]
From Corollary (\ref{cor:freqIncNabpMatNabp}) we can interpolate with the estimate 
\ali{
 \pr_t \nab \de p_{(k)} &= (\pr_t + P_{\leq k} v \cdot \nab) \nab \de p_{(k)} - P_{\leq k} v \cdot \nab \nab \de p_{(k)} \\
\Rightarrow \co{\pr_t \nab \de p_{(k)} } &\leq C  (1 + |k - \kotn| )2^{2(1-\a) k} \| v \|_{C_tC_x^\a} \ctdcxa{v}^2 \label{ineq:prtnabdepk}
} 
to obtain
\ali{
\| \nab \de p_{(k)} \|_{C_{t,x}^\b} &\leq C (1 + |k - \kotn| ) 2^{\b k} ( 1 + \| v \|_{C_t C_x^\a})^\b \cdot 2^{(1 - 2 \a) k} \ctdcxa{v}^2
}
which implies that $\nab p \in C_{t,x}^\b$ for $\b < 2 \a - 1$ whenever $1/2 < \a < 1$.

Similarly, interpolating the bounds in Corollary~\ref{cor:freqIncNabpMatNabp} for $\a > 2/3$ also yields
\ali{
\| \de_{(k)} \fr{D_{\lk}}{\pr t} \nab p_{(k)} \|_{C_{t,x}^\b} &\leq C (1 + |k - \kotn| ) 2^{\b k} (1 + \| v \|_{C_tC_x^\a})^\b 2^{(2 - 3 \a) k} \ctdcxa{v}^3
}
from which it follows that $\pr_t \nab p + \pr_j ( v^j \nab p) \in C_{t,x}^\b$ for all $\b < 3\a - 2$ provided $\a > 2/3$.  

To establish H\"{o}lder regularity for $\pr_t p$, we can use the formula
\ali{
\pr_t^2 \de p_{(k)} &= \pr_t[ (\pr_t + P_{\lk} v \cdot \nab) \de p_{(k)} - P_{\leq k} v \cdot \nab \de p_{(k)} \\
&= (\pr_t + P_{\lk} v \cdot \nab)^2 \de p_{(k)} - P_{\lk} v \cdot \nab (\pr_t + P_{\lk} v \cdot \nab) \de p_{(k)} \\
&- (\pr_t + P_{\lk} v \cdot \nab)[P_{\leq k} v \cdot \nab \de p_{(k)}] + P_{\leq k} v \cdot \nab [P_{\leq k} v \cdot \nab \de p_{(k)}]
}
which implies
\ali{
\co{ \pr_t^2 \de p_{(k)} } &\leq C (1 + |k - \kotn| ) \|v\|_{C_t C_x^\a}^2 2^{(2 - 2 \a) k} \ctdcxa{v}^2
}
Together with the estimate (\ref{ineq:prtnabdepk}), we have
\ali{
 \| \pr_t \de p_{(k)} \|_{C_{t,x}^\b} \leq C (1 + |k - \kotn| ) (1 + \|v\|_{C_t C_x^\a})^\b 2^{\b k} \cdot 2^{(1-2\a)k} \|v\|_{C_t C_x^\a} \ctdcxa{v}^2 
}
and hence $\pr_t p \in C_{t,x}^\b$ for all $\b < 2 \a - 1$ when $\a > 1/2$.

The method above also applies to the frequency increments for the second order material derivative of $\nab p$, which are defined as
\ali{
\de_{(k)} \fr{D_{\leq k}^2}{\pr t^2} \nab p_{(k)} &= \fr{D_{\leq k+1}^2}{\pr t^2} \nab p_{(k+1)}-\fr{D_{\leq k}^2}{\pr t^2} \nab p_{(k)} \\
&= \fr{D_{\leq k+1}^2}{\pr t^2} \nab \de p_{(k)} + \left(\fr{D_{\leq k+1}^2}{\pr t^2} - \fr{D_{\leq k}^2}{\pr t^2}\right) \nab p_{(k)} \\
&= \fr{D_{\leq k+1}^2}{\pr t^2} \nab \de p_{(k)} + \left(\fr{D_{\leq k+1}}{\pr t} - \fr{D_{\leq k}}{\pr t}\right)\fr{D_{\leq k+1}}{\pr t} \nab p_{(k)} \notag \\
&+ \fr{D_{\lk}}{\pr t}(\fr{D_{\leq k+1}}{\pr t} - \fr{D_{\leq k}}{\pr t}) \nab p_{(k)} \label{eq:usedProductRuleink} \\
&= \fr{D_{\leq k+1}^2}{\pr t^2} \nab \de p_{(k)} + P_{k+1} v \cdot \nab \fr{D_{\leq k+1}}{\pr t} \nab p_{(k)}+ \fr{D_{\lk}}{\pr t}[ P_{k+1} v \cdot \nab \nab p_{(k)} ] \label{eq:gooddekDDdtnabpk}
}
From the bounds 
\[ \co{ \de_{(k)} \fr{D_{\leq k}^2}{\pr t^2} \nab p_{(k)}} \leq C (1 + |k - \kotn| ) 2^{( 2(1-\a) + 1-2\a)k} \ctdcxa{v}^2, \] 
we see that $\fr{D^2}{\pr t^2} \nab p \in C^0$ if $3/4 < \a$.  These observations together suffice for the proof of Theorem (\ref{thm:theMatRegWeProveNow}).  The bounds also give some H\"{o}lder regularity in space for $\fr{D^2}{\pr t^2} \nab p$, but proving convergence in $C_{t,x}^\b$ will require higher order estimates.
\end{proof}

Before moving on to establish the general higher order estimates for material derivatives, we examine the regularity that can be established already for the pressure itself.  The regularity for the pressure appears to be slightly more subtle than the regularity for the velocity field and pressure gradient stated in Theorem~\ref{thm:theMatRegWeProveNow}.

\subsection{Regularity in time for the pressure} \label{sec:firstRegTimePress} 

Using the methods in Section (\ref{sec:pressIncPressGrad}) and a few additional bounds, we can also establish the following regularity results for the pressure and its material derivatives.

Here we use the notation
\[ (x)_+ = \begin{cases} x &\tx{if } x \geq 0 \\ 0 &\tx{if } x < 0 \end{cases} \]

\begin{thm}\label{thm:2timeRegp}
If $1/3 < \a < 1$, then the distribution $\fr{D}{\pr t} p = \pr_t p + \pr_j( p v^j)$ belongs to
\ali{
\fr{D}{\pr t} p &\in C_{t,x}^\b 
}
for all $0 \leq \b < (1 - \a) + (1 - 2 \a)_+$.

If $1/2 < \a < 1$, we also have $\fr{D^2}{\pr t^2} p = \pr_t \fr{D}{\pr t}p + \pr_j(v^j  \fr{D}{\pr t}p) \in C^0$.

\end{thm}

Observe that here we are unable to show that $\nab \fr{D}{\pr t} p \in C^0$ even when $v \in C_tC_x^1$, whereas Theorem (\ref{thm:theMatRegWeProveNow}) guarantees that $\fr{D}{\pr t} \nab p \in C^0$ whenever $\a > 2/3$.

We begin the proof of Theorem (\ref{thm:2timeRegp}) by stating a few extra preliminary estimates.
\begin{lem}\label{lem:justForDdtp} If $\a \neq 2/3$, then
\ali{
\co{\nab p_{(k)}} &\leq C (1 + |k - k_0(\T^n)|) 2^{(1-2 \a)_+ k} \ctdcxa{v}^2 \label{eq:basicPressBoundpk} \\
\co{ (\pr_t + P_{\leq k} v \cdot \nab) \nab p_{(k)} } &\leq C (1 + |k - k_0(\T^n)|) 2^{(2 - 3 \a)_+k} \ctdcxa{v}^3 \label{eq:secondBasicDdtNabpk}
}
\end{lem}
The estimates follow from the arguments of Proposition \ref{cor:morePressIncEstimates}.  The bound (\ref{eq:basicPressBoundpk}) is obtained by summing by parts the bounds for the pressure increments
\[ \co{\nab^D \de p_{(I)}} \leq C (1 + |I - k_0(\T^n)|) 2^{(D - 2 \a)I}\ctdcxa{v}^2 \]
from $I = \kotn$ to $k$.  The new point here is that when $\a > 1/2$, the most we can say is that $\nab p_{(k)}$ is bounded, rather than decaying at the rate of $2^{(1-2 \a) k} \ctdcxa{v}^2$ that dimensional analysis would suggest.

The same technique above was used to establish (\ref{eq:secondBasicDdtNabpk}) in Proposition~\ref{cor:morePressIncEstimates} by summing the bounds for 
\ALI{
\de_{(I)} \fr{D_{\leq I}}{\pr t} \nab p_{(I)} &= (\pr_t + P_{\leq I+1} v \cdot \nab) \nab p_{(I+1)} - (\pr_t + P_{\leq I} v \cdot \nab) \nab p_{(I)} \\
&= P_{I+1} v \cdot \nab \nab p_{(I+1)} + (\pr_t + P_{\leq I} v \cdot \nab) \nab \de p_{(I)} 
} 
from $I = \kotn$ to $k$.  There we used an extra summation by parts in $I$ when $\a < 2/3$ -- when $\a = 2/3$ the method leads
to an extra factor of $(1 + |k - k_0(\T^n)|)$ in the estimate.  When $\a > 2/3$, we have a decaying geometric series, so the main term is the first term, which is bounded (in particular, the factor $(1 + |k - k_0(\T^n)|)$ does not actually appear in this case).  The bound (\ref{eq:secondBasicDdtNabpk}) has not been used to establish any of the results proven so far, but we will need it for Theorem \ref{thm:2timeRegp}.

With these bounds in hand we can estimate the frequency increments for $\fr{D}{\pr t} p$ and $\fr{D^2}{\pr t^2} p$.
\begin{lem} \label{lem:DdtpandDDdtp}
Define the frequency increments
\ali{
\de_{(k)} \fr{D_{\lk}}{\pr t} p_{(k)} &= \fr{D_{\leq k+1}}{\pr t} p_{(k+1)} - \fr{D_{\leq k}}{\pr t} p_{(k)} \\
&=  \fr{D_{\leq k}}{\pr t} \de p_{(k)} + P_{k+1} v \cdot \nab p_{(k+1)}  \label{form:dekDkpk}
}
and, following (\ref{eq:gooddekDDdtnabpk}),	
\ali{
\de_{(k)} \fr{D_{\lk}^2}{\pr t^2} p_{(k)} &= \fr{D_{\leq k+1}^2}{\pr t^2} p_{(k+1)} - \fr{D_{\leq k}^2}{\pr t^2} p_{(k)} \\
&= \fr{D_{\leq k+1}^2}{\pr t^2} \de p_{(k)} + P_{k+1} v \cdot \nab \fr{D_{\leq k+1}}{\pr t} p_{(k+1)}+ \fr{D_{\lk}}{\pr t}\left[ P_{k+1} v \cdot \nab p_{(k+1)} \right]  \notag \\
&= \fr{D_{\leq k+1}^2}{\pr t^2} \de p_{(k)} + P_{k+1} v^i \pr_i P_{\leq k+1} v^j \pr_j p_{(k+1)} + P_{k+1} v^i \fr{D_{\leq k+1}}{\pr t} \pr_i p_{(k+1)}  \notag \\
&+ (\fr{D_{\lk}}{\pr t}P_{k+1} v) \cdot \nab p_{(k+1)} + P_{k+1} v^i \fr{D_{\lk}}{\pr t}\pr_i p_{(k+1)} \label{form:dekDDktpk}
}
Then we have the estimates
\ali{
\co{ \nab^D \de_{(k)} \fr{D_{\lk}}{\pr t} p_{(k)} } &\leq C (1 + |k - \kotn|) 2^{(D + (1-\a) - 1 + (1-2 \a)_+)k } \ctdcxa{v}^3 \\
\co{ \nab^D \fr{D_{\lk}}{\pr t} \de_{(k)} \fr{D_{\lk}}{\pr t} p_{(k)} } &\leq C (1 + |k - \kotn|) 2^{(D + 2(1-\a) - 1 + (1-2 \a)_+)k } \ctdcxa{v}^4 \\
\co{ \nab^D \de_{(k)} \fr{D_{\lk}^2}{\pr t^2} p_{(k)} } &\leq C (1 + |k - \kotn|) 2^{(D + 2(1-\a) - 1 + (1-2 \a)_+)k } \ctdcxa{v}^4 
}
\end{lem}
Lemma \ref{lem:DdtpandDDdtp} follows by applying the bounds in Proposition \ref{prop:pressIncSummary2matdv} and Lemma \ref{lem:justForDdtp} to the formulas (\ref{form:dekDkpk}) and (\ref{form:dekDDktpk}).  In every case, the dominant terms are the ones with pressure gradients that are not differentiated, where we apply the bound $\co{\nab p_{(k)} } \leq (1+ |k - \kotn|) 2^{(1-2\a)_+ k} \ctdcxa{v}^2$.

Theorem \ref{thm:2timeRegp} now follows from Lemma \ref{lem:DdtpandDDdtp} by interpolation as in the arguments of Section \ref{sec:pressIncPressGrad}.

Having proven Theorems \ref{thm:2timeRegp} and \ref{thm:theMatRegWeProveNow} we now move on to the proof of the general Theorems \ref{thm:genThm} and \ref{thm:matDvPCts}, which require estimates for higher order material derivatives.

\section{Higher order material derivatives} \label{sec:theGeneralCase}

We now begin the proof of Theorems (\ref{thm:genThm}) and (\ref{thm:matDvPCts}), which summarize the H\"{o}lder regularity of all material derivatives $\fr{D^r}{\pr t^r} v$ and $\fr{D^r}{\pr t^r} p$ in space and time given that $v \in L_t^\infty C_x^\a$.  The proof proceeds by generalizing the proof of Theorem (\ref{thm:theMatRegWeProveNow}) to allow for higher order advective derivative estimates within a framework that is well-suited for induction.

%Compared to the proof of Theorem (\ref{thm:theMatRegWeProveNow}), there is essentially only one new idea necessary to establish the estimates for Theorem (\ref{thm:genThm}) which is needed to obtain bounds for higher order material derivatives of $R_{\leq k}^{jl}$, $P_k p$ and $\nab^2 P_{\leq k} p$.  We will therefore concentrate on the idea that goes into estimating the forcing terms.

We start by summarizing the notation we will be using in the rest of the proof (much of which has already been introduced), and by stating some preliminary lemmas.

\subsection{Notation and Preliminaries} \label{sec:notationPrelims}

In this Section we recall some notation that has been introduced during the course of the proof and will be used more heavily in what follows.  We also state some algebraic lemmas and conventions that we will follow in the remainder of the proof.

\subsection{Algebraic Conventions and Commutator Identities}

Let $X, Y$ and $Z$ belong to a noncommutative ring of operators.  For operators $Y_1, Y_2, \ldots, Y_n$, we use the notation
\[ \prod_{i=1}^n Y_i = Y_1 Y_2 \cdots Y_{n-1} Y_n \]
to denote the product of the operators taken from left to right.  An empty product is equal to $1$.

We use the notation
\[ \left[ X, \right] Y = [X,Y] = X Y - YX \]
to denote the commutator of $Y$ with $X$, and we let
\[ \left[ X, \right]^r Y \]
denote the operator obtained by commuting $Y$ with $X$ repeatedly $r$ times.  For example, $[X,]^2Y = [X,[X,Y]] = X(XY - YX) - (XY-YX)X$.

We will often use the following product rule for the commutator
\ali{
[X,](YZ) &= [X,]YZ + Y[X,]Z \label{eq:prodRule}\\
&= (XYZ - YXZ) + (YXZ - YZX)
}
Here and in what follows, we employ the convention that the commutator $[X,]$ precedes the operator multiplication in the order of operations.

We record here the formula
\ali{
\left[X^r,\right]Y &= \sum_{s = 1}^r \binom{r}{s} \left[X,\right]^s Y X^{r - s} \label{eq:commuteWithPowers}
}
Formula (\ref{eq:commuteWithPowers}) can be obtained from (\ref{eq:prodRule}) by induction on $r$ using Pascal's rule and the expression
\ALI{
\left[X^{r+1},\right]Y &= - [Y,](X X^r) = - [Y,X] X^r - X [Y,X^r] \\
&= [X,]Y X^r + [X,][X^r,]Y + [X^r,]Y X
}
or by comparing the coefficients of $t^r$ in the generating function
\[ e^{t X} Y = e^{t [X,]} Y e^{tX}. \]
The identity $e^{tX} Y e^{-tX} = e^{t [X,]} Y$ used above follows from uniqueness of solutions for ODEs.

As a consequence of (\ref{eq:commuteWithPowers}), it is possible to express the power of a sum of noncommutative operators in the form
\ali{
(X + Y)^n &= \sum_{\ell =0}^n \sum_{r_1, \ldots, r_\ell, m} C_{r_1, \ldots, r_\ell} \left( \prod_{i=1}^\ell \left[X,\right]^{r_i} Y \right) X^m \label{eq:sumExpand}
}
for some non-negative integers $C_{r_1, \ldots, r_\ell}$, where the sum runs over non-negative indices satisfying $r_1 + \ldots + r_\ell + \ell + m = n$.

In the applications below, we will always take the operator $X$ in the formulas (\ref{eq:commuteWithPowers}), (\ref{eq:sumExpand}) to be an operator of the form $X = \fr{D_{\leqc k}}{\pr t}$ as defined in (\ref{eq:dlqckDef}) below.

\subsection{Coarse scale material derivatives and notation} 

We denote by $\fr{D_{\leq k}^r}{\pr t^r}$ the $r$-times repeated, coarse scale advective derivative
\ali{
 \fr{D_{\leq k}^r}{\pr t^r} &= (\pr_t + P_{\leq k} v \cdot \nab)^r \label{def:theCoarseScaleMatdv}
}
We denote by $P_{\approx k}$ any operator of the form
\[ P_{\approx k} = P_{[k_1, k_2]} \]
for which the differences $|k_1 - k|$ and $|k_2 - k|$ are bounded.  Thus, operators of the form $P_{\approx k}$ are supported on a frequency shell $C^{-1} 2^k \leq |\xi| \leq C 2^k$, $\xi \in {\hat \R}^n$ with $C$ a constant which will depend only on the number $\a < 1$, which is fixed in the remainder of the proof.

Similarly, we denote by $P_{\leqc k}$ any operator of the form $P_{\leq k + a}$ where $|a| \leq C$ for some constant $C$.  Thus, ``projections'' $P_{\leqc k}$ restrict to frequencies $|\xi| \leq C 2^k$, and the difference between any two such operators has the form
\[ P_{\leqc k} - P_{\leqc k} = P_{\approx k} \]

Generalizing (\ref{def:theCoarseScaleMatdv}) we denote by $\fr{D_{\leqc k}}{\pr t}$ any operator of the form
\ali{
\fr{D_{\leqc k}}{\pr t} = (\pr_t + P_{\leqc k} v \cdot \nab) \label{eq:dlqckDef}
}

\subsection{The Main Lemma}

The Main Lemma used to establish Theorem (\ref{thm:genThm}) is the following

\begin{lem}[Main Lemma I]\label{lem:theMainLem} Suppose that $(v,p)$ are solutions to the incompressible Euler equations and fix $0 < \a < 1$.  Then for all $r(1-\a) - 2 \a < 0$, we have the estimates
\ali{
\co{ \nab^A \Dkdt{r} R_{\lk}} + \co{ \nab^A \Dkdt{r} P_{k} p } &\leq C_A 2^{(A + r (1-\a) - 2 \a )k} \ctdcxa{v}^{2+r} \label{ineq:stressAndLPpressbds} \\
\co{ \nab^A \Dkdt{r} \nab^2 P_{\leq k} p } &\leq C_A 2^{(A + r(1-\a) + (2 - 2 \a) ) k} \ctdcxa{v}^{2+r} \label{ineq:Plkpressbds}
}
and for $0 \leq s \leq r+1$,
\ali{
\co{ \nab^A \Dkdt{s} P_{\approx k} v } &\leq C_A 2^{(A + s(1-\a) - \a )k} \ctdcxa{v}^{1+s} \label{ineq:pApproxKboundsGen} \\
\co{ \nab^A \Dkdt{s} \nab P_{\leqc k} v } &\leq C_A 2^{(A + s(1-\a)+ (1 - \a) )k} \ctdcxa{v}^{1+s} \label{ineq:pleqcKboundsGen}
}
Furthermore, the vector fields $Z_s(t,\cdot) : \T^n \to \R^n$ obtained by commuting
\ali{
 \left[ \Dkdt{}, \right]^s ( P_{\approx k} v \cdot \nab ) = Z_s \cdot \nab \label{eq:commuteTermDefs}
}
have coefficients $Z_s(t,x)$ obeying the bounds
\ali{
\co{  \nab^A \Dkdt{q} Z_s } &\leq C_A 2^{(A + (q + s)(1- \a) - \a)k} \ctdcxa{v}^{q+s+1} \label{eq:boundForCommuteTermsDif}
}
provided $q + s \leq r+1$.

Also, for any operator of the form $Tf(x) = \int f(x+h) K(h) dh$ whose kernel $K(h)$ satisfies
\ali{
\| \nab^A K \|_{L^1_h} + \| |h| \nab^{1+A} K \|_{L^1_h} + \| |h|^2 \nab^{2+A} K \|_{L^1_h} + \ldots + \| |h|^{r+2} \nab^{r+2+A} K \|_{L^1_h} &\leq 2^{A k} \label{eq:kernelBounds}
}
for $A = 0, 1, \ldots, M$, we have the commutator estimates 
\ali{
\sup_t \left \| \nab^A \left[ \Dkdt{}, \right]^s T \right \| &\leq C_A 2^{(A + s(1-\a))k} \ctdcxa{v}^s  \label{ineq:genCommutEstimates}
}
for $0 \leq s \leq r+2$ and $A = 0,1,\ldots, M$.
\end{lem}

In the Sections~\ref{sec:introMainLem}-\ref{sec:highOrderAdvecForcing} below, we will give the proof of Lemma \ref{lem:theMainLem}.  Here we will outline how Lemma \ref{lem:theMainLem} implies Theorems \ref{thm:genThm} and \ref{thm:matDvPCts}, starting with the proof of Theorem \ref{thm:genThm}.

\subsection{Proof of Theorem~\ref{thm:genThm} on the time regularity of the velocity field} \label{sec:timeRegOfVelocField}

Here we show how Lemma~\ref{lem:theMainLem} can be used to establish Theorem~\ref{thm:genThm}.  The key idea is to prove estimates for the following frequency increments for the velocity field and its higher advective derivatives
\ali{
\de_{(k)} \fr{D_{\leq k}^r}{\pr t^r} P_{\leq k} v &= \fr{D_{\leq k+1}^r}{\pr t^r} P_{\leq k+1} v - \fr{D_{\leq k}^r}{\pr t^r} P_{\leq k} v \label{eq:defOfVelocFreqInc}
}
These frequency increments are defined so that 
\ali{
 \sum_{k = \kotn}^\infty \de_{(k)} \fr{D_{\lk}^r}{\pr t^r} P_{\leq k} v = \fr{D^r}{\pr t^r} v  \label{eq:sumVelocFreqInc}
}
when the summation converges uniformly.  %For $r = 0$, these frequency increments are simply the Littlewood-Paley projections of $v$, and the uniform convergence is clear.

Using Lemma~\ref{lem:theMainLem} we obtain the following bounds on the frequency increments defined in \eqref{eq:defOfVelocFreqInc}
\begin{lem}[Velocity increment bounds]\label{lem:velocIncBounds2} For all $r(1-\a) - 2 \a < 0$ and all $q + s \leq r+1$, the frequency increments defined in \eqref{eq:defOfVelocFreqInc} satisfy the bounds
\ali{
\co{\nab^A \fr{D_{\leq k}^q}{\pr t^q} \de_{(k)}\fr{D_{\leq k}^s}{\pr t^s} P_{\leq k} v  } &\leq C_A 2^{(A + (q+s)(1-\a) - \a)k} \ctdcxa{v}^{q+s+1} \label{eq:velocIncBoundsGood}
}
\end{lem}
\begin{proof}
Letting $r$ be fixed, we proceed by induction on $s$.  For $s = 0$, the frequency increments defined in \eqref{eq:defOfVelocFreqInc} are simply the Littlewood-Paley projections of $v$, and the estimate \eqref{eq:velocIncBoundsGood} follows for all $q \leq r + 1$ from the bound \eqref{ineq:pApproxKboundsGen}.  

We assume now by induction that \eqref{eq:velocIncBoundsGood} has been established for some $s \geq 0$ and consider the bound \eqref{eq:velocIncBoundsGood} for $s + 1$.  Observe that we can write each frequency increment for the $s+1$'st advective derivative in the form
\ali{
\de_{(k)} \fr{D_{\leq k}^{s+1}}{\pr t^{s+1}} P_{\leq k} v &= \fr{D_{\leq k+1}}{\pr t} \left[ \fr{D_{\leq k+1}^s}{\pr t^s} P_{\leq k+1} v \right] - \fr{D_{\leq k}}{\pr t} \left[ \fr{D_{\leq k}^s}{\pr t^s} P_{\leq k} v \right] \\
&= P_{k+1} v \cdot \nab \fr{D_{\leq k+1}^s}{\pr t^s} P_{\leq k+1} v + \fr{D_{\leq k}}{\pr t} \de_{(k)}\fr{D_{\leq k}^{s}}{\pr t^{s}} P_{\leq k} v \label{eq:inductStructureVeloc}
}
The first term in \eqref{eq:inductStructureVeloc} and its first $q \leq r + 1 - s$ advective derivatives can be estimated using Lemma~\ref{lem:theMainLem}.  These estimates follow by commuting the spatial derivative in $P_{k+1} v \cdot \nab$ onto the term $P_{\leq k} v$ using the formula \eqref{eq:commuteWithPowers}.  All of the terms generated by this commutation obey bounds of the form \eqref{eq:velocIncBoundsGood} by \eqref{ineq:pApproxKboundsGen}-\eqref{ineq:pleqcKboundsGen}.  %One can alternatively commute with the operator $P_{k+1} v \cdot \nab$ and use \eqref{eq:commuteWithPowers} and \eqref{eq:boundForCommuteTermsDif} to estimate the commutator terms.  

The second term in \eqref{eq:inductStructureVeloc} %and its advective derivatives of order $q \leq r + 1 - s$ 
obeys an estimate of the form \eqref{eq:velocIncBoundsGood} by our induction hypothesis, which concludes the proof of Lemma~\ref{lem:velocIncBounds2}.
\end{proof}

We can now conclude that the series \eqref{eq:sumVelocFreqInc} converges in $C^0$ for all $0 \leq r < \fr{\a}{1 - \a}$ by induction on $r$ using Lemma~\ref{lem:weakConverge} and the case $q = 0$ of the estimate \eqref{eq:velocIncBoundsGood}.  Theorem~\ref{thm:genThm} now follows as in the argument of Section~\ref{sec:endpointCase} by using the cases $A = 1, q = 0$ and $A= 0, q = 1$ of inequality~(\ref{eq:velocIncBoundsGood}) to bound the first spatial and temporal derivatives of the frequency increments.

\subsection{Proof of Theorem~\ref{thm:matDvPCts} on the time regularity of the pressure}

Our proof of Theorem~\ref{thm:matDvPCts} on the regularity in time of the pressure will require an analysis of frequency increments for the pressure which generalizes the analysis of Section~\ref{sec:pressIncPressGrad}.

The most basic estimates on the pressure increments are provided by Lemma~\ref{prop:pressIncBoundsGeneral} below, which is deduced from Lemma \ref{lem:theMainLem}. %by generalizing the arguments of Proposition \ref{prop:pressIncBds1} and Section \ref{sec:pressIncPressGrad}.
\begin{lem}[Pressure Increment bounds]\label{prop:pressIncBoundsGeneral}
Let $\de p_{(k)}$ be defined as in Definition \ref{def:pressInc}.  Then for all $r(1-\a) - 2 \a < 0$ and all $s + q \leq r+1$
\ali{
\co{ \nab^A \Dkdt{s} \de p_{(k)} } &\leq C_A (1 + | k - \kotn|) 2^{(A + s(1-\a) -2\a) k} \ctdcxa{v}^{2+s} \label{ineq:genBdDdtdepk} \\
\co{ \nab^A \Dkdt{s} \nab^2 p_{(k)} } &\leq C_A (1 + | k - \kotn|) 2^{(A + s(1-\a) + (2 - 2\a) ) k} \ctdcxa{v}^{2+s} \label{ineq:genDkdtNab2p}\\
\co{ \nab^{A} \nab \Dkdt{s} \nab p_{(k)} } &\leq C_A (1 + | k - \kotn|) 2^{(A + s(1-\a) + (2 - 2\a) ) k} \ctdcxa{v}^{2+s} \label{ineq:freqIntBoundNabNabpk} \\
\co{ \nab^{A} \Dkdt{q} \nab \Dkdt{s} \nab p_{(k)} } &\leq C_A (1 + | k - \kotn|) 2^{(A + (q+s)(1-\a) + (2 - 2\a) ) k} \ctdcxa{v}^{2+q+s} \label{eq:commuteOneNabnbpk}
}
\end{lem}
\begin{proof}

The proof of (\ref{ineq:genBdDdtdepk}) is a simple modification of the proof of (\ref{ineq:genDkdtNab2p}) given below, which generalizes the arguments of Proposition \ref{prop:pressIncBds1} and Section \ref{sec:pressIncPressGrad}.  The difference is that one uses formulas (\ref{de:dePLLshort})-(\ref{eq:dePHHshort}) for the pressure increments.

To prove (\ref{ineq:genDkdtNab2p}), we use the formula
\ali{
\Dkdt{s} \nab^2 p_{(k)} &= \Dkdt{s} \De^{-1} \nab^2 P_{\leq k} [ \pr_j P_{\leq k} v^l \pr_l P_{\leq k} v^j ] \\
&= \sum_{q = 0}^s C_q \left[ \Dkdt{}, \right]^q \De^{-1} \nab^2 P_{\leq k} \Dkdt{s-q} (\pr_j P_{\leq k} v^l \pr_l P_{\leq k} v^j)
}
which is deduced from the rule (\ref{eq:commuteWithPowers}).  Lemma (\ref{lem:theMainLem}) guarantees that $s \leq r+1$ material derivatives of $\nab P_{\leq k} v$ obey the desired estimates.  The operator $\left[ \Dkdt{}, \right]^q \De^{-1} \nab^2 P_{\leq k}$ is estimated as in the proof of Lemma~\ref{lem:weNeedTheseCommutators2} by applying the commutator estimates in \eqref{ineq:genCommutEstimates} to the operator
\[C^{-1} (1 + |k - \kotn|)^{-1} \De^{-1} \nab^2 P_{\leq k} = C^{-1} (1 + |k - \kotn|)^{-1} \sum_{I = \kotn}^k \De^{-1} \nab^2 P_{I} \] 
Here $C$ is a universal constant chosen sufficiently large so that the estimate \eqref{eq:kernelBounds} holds for this operator.  Such a choice of $C$ is possible because, by scaling, each term $C^{-1} \De^{-1} \nab^2 P_{I}$ in this decomposition satisfies \eqref{eq:kernelBounds} with $2^k$ replaced by $2^I$ once $C$ is chosen appropriately.

Inequality (\ref{ineq:freqIntBoundNabNabpk}) is the special case $q = 0$ of Inequality (\ref{eq:commuteOneNabnbpk}).

We obtain Inequality (\ref{eq:commuteOneNabnbpk}) by induction on $s \leq r+1$.  The base case $s = 0$ is exactly Inequality (\ref{ineq:genDkdtNab2p}), which has been proven.  Assuming Inequality (\ref{eq:commuteOneNabnbpk}) for $s$, if $\nab_i$ is any spatial derivative, we have
\ali{
\nab_i \Dkdt{s+1} \nab p&= \Dkdt{}[ \nab_i \Dkdt{s} \nab p ] - \nab_i P_{\leqc k} v^j \nab_j \Dkdt{s} \nab p
}
By Inequality (\ref{eq:commuteOneNabnbpk}) for $s$ and (\ref{ineq:pleqcKboundsGen}), we have the desired estimates for $\Dkdt{q}$ for both terms provided $q + (s+1) \leq r+1$.
\end{proof}

In order to deduce Theorem (\ref{thm:matDvPCts}) from Lemma (\ref{lem:theMainLem}), we must also supplement the bounds in Lemma (\ref{prop:pressIncBoundsGeneral}) with the estimate
\begin{lem} For all $r(1-\a) - 2 \a < 0$ with $\a \neq 1/2$ and all $q \leq r+1$, we have
\ali{
\co{ \nab^A \Dkdt{q} \nab p_{(k)} } &\leq C_A (1 + |k - \kotn|) 2^{(A + q(1-\a) + (1- 2 \a)_+)k}\ctdcxa{v}^{q + 2} \label{ineq:theAnnoyingBound} \\
\co{ \nab^A \nab \Dkdt{q} p_{(k)} } &\leq C_A (1 + |k - \kotn|) 2^{(A + q(1-\a) + (1- 2 \a)_+)k}\ctdcxa{v}^{q + 2} \label{ineq:theAnnoyingBound2}
}
\end{lem}
\begin{proof}
We proceed by induction on $q$.  For $q = 0$, we have
\ali{
\nab p_{(k)} &= \sum_{I = \kotn}^k \nab \de p_{(I)}
}
The estimate (\ref{ineq:theAnnoyingBound}) is obtained for $\a < 1/2$ by summing by parts the estimate \eqref{ineq:genBdDdtdepk}; the largest contribution comes from the last terms in the series.  For $\a > 1/2$ and $A = 0$, the main terms come from the beginning of the series, giving (\ref{ineq:theAnnoyingBound}).  The estimate for larger values of $A$ is already recorded in Proposition~\ref{prop:pressIncBoundsGeneral}.  Assuming the bound (\ref{ineq:theAnnoyingBound}) for $q$, we prove the bound for $q+1$ by writing
\ali{
\de_{(k)} \fr{D_{\leq k}^{q+1}}{\pr t^{q+1}} \nab p_{(k)} &= P_{k+1} v \cdot \nab \fr{D_{\leq k+1}^{q}}{\pr t^{q}} \nab p_{(k+1)} + \fr{D_{\leq k}}{\pr t} \de_{(k)} \fr{D_{\leq k}^{q}}{\pr t^{q}} \nab p_{(k)}
}
and applying (\ref{ineq:freqIntBoundNabNabpk}) and our induction hypothesis on  (\ref{ineq:theAnnoyingBound}).

The bound (\ref{ineq:theAnnoyingBound2}) follows from (\ref{ineq:theAnnoyingBound}) and
(\ref{ineq:pleqcKboundsGen}) by inductively commuting spatial and advective derivatives, as in the proof of (\ref{ineq:freqIntBoundNabNabpk})-(\ref{eq:commuteOneNabnbpk}).
\end{proof}

We now define the frequency increments for material derivatives of the pressure
\ali{
 \de_{(k)} \fr{D_{\leq k}^s}{\pr t^s} p_{(k)} &= \fr{D_{\leq k+1}^s}{\pr t^s} p_{(k+1)} - \fr{D_{\leq k}^s}{\pr t^s} p_{(k)} 
}
Theorem (\ref{thm:matDvPCts}) will be deduced from Lemma~\ref{lem:pressIncs} below.
\begin{lem}\label{lem:pressIncs} For all $r(1-\a) - 2 \a < 0$, $\a \neq 1/2$, and all $s + q \leq r + 1$ we have
\ali{
\co{ \nab^A \Dkdt{q} \de_{(k)} \fr{D_{\leq k}^s}{\pr t^{s}} p_{(k)} } &\leq C_A (1 + | k - \kotn|) 2^{(A + (s+q)(1-\a) - 1 + (1-2\a)_+) k} \ctdcxa{v}^{2+q+s} \label{eq:genFreqIntBoundsDdtpk}
}
\end{lem}
\begin{proof}
%For $s = 1$, we have 
%\ali{
%\de_{(k)} \fr{D_{\leq k}}{\pr t} p_{(k)} &= P_{k+1} v \cdot \nab p_{(k+1)} + \fr{D_{\leq k}}{\pr t} \de p_{(k)}
%}
%So the bound (\ref{eq:genFreqIntBoundsDdtpk}) follows from Lemma (\ref{prop:pressIncBoundsGeneral}) together with (\ref{ineq:theAnnoyingBound}).  

We proceed by induction on $s$.  For $s = 0$, the bound (\ref{eq:genFreqIntBoundsDdtpk}) follows immediately from (\ref{ineq:genBdDdtdepk}), as we have $-2\a \leq -1 + (1 - 2 \a)_+ = - \min \{ 2\a, 1 \} $.  Assuming the bound (\ref{eq:genFreqIntBoundsDdtpk}) for $s$, we write
\ali{
\de_{(k)} \fr{D_{\leq k}^{s+1}}{\pr t^{s+1}} p_{(k)} &= P_{k+1} v \cdot \nab \fr{D_{\leq k}^{s}}{\pr t^{s}} p_{(k)} + \fr{D_{\leq k}}{\pr t} \de_{(k)} \fr{D_{\leq k}^{s}}{\pr t^{s}} p_{(k)}
}
which gives (\ref{eq:genFreqIntBoundsDdtpk}) after applying the induction hypothesis and bounding the first term with (\ref{ineq:theAnnoyingBound2}).
\end{proof}

Applying the case $q = 0$ of Lemma~\ref{lem:pressIncs} and iterating Lemma~\ref{lem:weakConverge}, we have that the series
\ali{
\fr{D_{\leq k}^s}{\pr t^s} p &= \sum_{k = \kotn}^\infty \de_{(k)} \fr{D_{\leq k}^s}{\pr t^s} p_{(k)}
}
converges uniformly for all $s < \fr{\min \{2\a, 1\} }{1-\a} = \fr{1 - (1 - 2 \a)_+}{1 - \a}$.  As in the arguments of Sections \ref{sec:timeRegPress} and \ref{sec:pressIncPressGrad}, we obtain the H\"{o}lder regularity in time and space for $\fr{D^s}{\pr t^s} \nab p$ stated in Theorem \ref{thm:genThm} by interpolating the bounds in Lemma \ref{lem:pressIncs} for first spatial and temporal derivatives ($A = 1, q = 0$ and $A = 0, q = 1$) with the case $q = 0, A = 0$ to conclude the proof. % for \[ \fr{D_{\lk}}{\pr t} \de_{(k)} \fr{D_{\lk}^s}{\pr t^s}  p_{(k)}\] and 
%\[ \nab \de_{(k)} \fr{D_{\lk}^s}{\pr t^s}  p_{(k)}\] 
%gives the H\"{o}lder regularity in time and space for $\fr{D^s}{\pr t^s} \nab p$ stated in Theorem \ref{thm:genThm} as desired. 

\subsection{Proof of the Main Lemma, Intro} \label{sec:introMainLem}

We now turn to the proof of the Main Lemma (\ref{lem:theMainLem}).  The proof proceeds by induction on $r$, so we will assume that Lemma (\ref{lem:theMainLem}) has been proven for $r \leq n$, and we will prove that Lemma (\ref{lem:theMainLem}) also holds for $r = n+1$.  The base cases $r = 0, 1$ have been established in Sections (\ref{sec:boundsLPpieces}) through (\ref{sec:secondMatDvPressIncs}). %, and these proofs contain most of the ideas necessary for the general case of Lemma (\ref{lem:theMainLem}).  

We start the presentation by showing how the case $r = n+1$ of Lemma (\ref{lem:theMainLem}) can be reduced to establishing the bound (\ref{ineq:Plkpressbds}) for $r = n+1$ using the cases $r \leq n$ of Lemma (\ref{lem:theMainLem}) as an inductive hypothesis.  The main step that requires a new trick is to prove the estimate (\ref{ineq:Plkpressbds}) for $r = n+1$.  

\subsection{Reducing to the forcing term estimates} \label{sec:reduceToForcing}

%Once (\ref{ineq:Plkpressbds}) has been established for $r = n+1$, the rest of the Lemma follows from (\ref{ineq:Plkpressbds}) and the cases $r \leq n$ of Lemma (\ref{lem:theMainLem}) as we now explain.

In this Section we assume that Lemma (\ref{lem:theMainLem}) has been proven for $r \leq n$ and furthermore that the bound (\ref{ineq:stressAndLPpressbds}) has been established for $r = n+1$.  In this section, we show how the rest of the statements in the case $r = n+1$ of Lemma (\ref{lem:theMainLem}) excluding \eqref{ineq:stressAndLPpressbds} follow from these assumptions.

\paragraph{The estimate (\ref{ineq:Plkpressbds}). }

We can obtain the estimate (\ref{ineq:Plkpressbds}) for $r = n+1$ by decomposing into frequency increments
\ali{
\fr{D_{\leq k}^{n+1}}{\pr t^{n+1}} \nab^2 P_{\leq k} p &= \sum_{I = \kotn}^{k-1} \de_{(I)} \fr{D_{\leq I}^{n+1}}{\pr t^{n+1}} \nab^2 P_{\leq I} p  \label{eq:kotnFrIntDdtIp} \\
\de_{(I)} \fr{D_{\leq I}^{n+1}}{\pr t^{n+1}} \nab^2 P_{\leq I} p &= \left(\fr{D_{\leq I + 1}^{n+1}}{\pr t^{n+1}} \nab^2 P_{\leq I+1} p - \fr{D_{\leq I}^{n+1}}{\pr t^{n+1}} \nab^2 P_{\leq I} p\right).
}
The bound (\ref{ineq:Plkpressbds}) then follows from the bound 
\ali{
\co{ \nab^A \de_{(I)} \fr{D_{\leq I}^{n+1}}{\pr t^{n+1}} \nab^2 P_{\leq I} p } &\leq 2^{(A + (n+1)(1-\a) + (2-2\a))I} \ctdcxa{v}^{n+3} \label{ineq:freqIncsDdtn1PlqIp}
}
for the frequency increments, because the estimate for the terms in \eqref{eq:kotnFrIntDdtIp} grows geometrically, with the main term coming from the last term $I = k-1$.  

We now focus our attention on proving the inequality \eqref{ineq:freqIncsDdtn1PlqIp}. % The estimate \eqref{ineq:freqIncsDdtn1PlqIp} follows, in turn, by an induction argument as in the proof of Lemma~\ref{lem:velocIncBounds2} above.

\paragraph{Inequality \eqref{ineq:freqIncsDdtn1PlqIp} }
Inequality (\ref{ineq:freqIncsDdtn1PlqIp}) can be proven quickly by induction on $n$ in a manner similar to the proofs of Lemmas \ref{lem:velocIncBounds2} and \ref{lem:pressIncs}.  Here we unwind the induction to give a more direct proof.

Using the ``product rule'' for $\de_{(I)}$, we decompose
\ali{
\de_{(I)} \fr{D_{\leq I}^{n+1}}{\pr t^{n+1}} \nab^2 P_{\leq I} p &= \Ga_I + \fr{D_{\leq I}^{n+1}}{\pr t^{n+1}} \nab^2 P_{I+1} p \label{eq:nplus1partOfFreqInc} \\
\Ga_I &= (P_{I+1} v \cdot \nab)\fr{D_{\leq I+1}^{n}}{\pr t^{n}} \nab^2 P_{\leq I + 1} p + \fr{D_{\leq I}}{\pr t}(P_{I+1} v \cdot \nab)\fr{D_{\leq I+1}^{n-1}}{\pr t^{n-1}} \nab^2 P_{\leq I + 1} p  \notag\\
&+ \ldots + \fr{D_{\leq I}^{n}}{\pr t^{n}} (P_{I+1} v \cdot \nab) \nab^2 P_{\leq I+1} p \\ 
&= \sum_{j = 0}^{n} \fr{D_{\leq I}^j}{\pr t^j} \left( P_{I+1} v \cdot \nab \right) \fr{D_{\leq I+1}^{n-j}}{\pr t^{n-j}} \nab^2 P_{\leq I+1} p \label{eq:formForSiI}
}
The $P_{I+1} p$ term separated from the series in (\ref{eq:nplus1partOfFreqInc}) can be estimated by the case $r = n+1$ of (\ref{ineq:stressAndLPpressbds}).  Since at most $n$ material derivatives fall on $\nab^2 P_{\leq I} p$, the series (\ref{eq:formForSiI}) can be estimated by the $r \leq n$ case of Lemma (\ref{lem:theMainLem}) once each term has been expanded using the commutator rules (\ref{eq:commuteWithPowers}) and (\ref{eq:sumExpand})
\ali{
\fr{D_{\leq I}^j}{\pr t^j} &\left( P_{I+1} v \cdot \nab \right) \fr{D_{\leq I+1}^{n-j}}{\pr t^{n-j}} \nab^2 P_{\leq I+1} p = \fr{D_{\leq I}^j}{\pr t^j} \left( P_{I+1} v \cdot \nab \right) \left( P_{I+1} v \cdot \nab +  \fr{D_{\leq I}}{\pr t}\right)^{n-j}\nab^2 P_{\leq I+1} p  \label{eq:preCommuteNab2pTerm}\\
&= \sum_{1 \leq \ell \leq (n-j)+1} \sum_{r_1, \ldots, r_\ell, m} C_{j, r_1, \ldots, r_\ell} \left(\prod_{i = 1}^\ell \left[ \fr{D_{\leq I}}{\pr t}, \right]^{r_i} \left( P_{I+1} v \cdot \nab \right)\right) \fr{D_{\leq I}^m}{\pr t^m} \nab^2 P_{\leq I+1} p \label{eq:theBigFormForNab2PlqIp}
}
Here the sum only runs over non-negative indices with $r_1 + \ldots + r_\ell + \ell + m = n + 1$ and $\ell \geq 1$.  Therefore, at most $r_1 + \ldots + r_\ell + m = n + 1 - \ell \leq n$ advective derivatives appear in each term of \eqref{eq:theBigFormForNab2PlqIp}.  Applying the $r \leq n$ case of Lemma (\ref{lem:theMainLem}) gives the bound (\ref{ineq:Plkpressbds}).

\paragraph{Estimates (\ref{eq:boundForCommuteTermsDif}) and (\ref{ineq:genCommutEstimates}) for the commutators.}

The estimates (\ref{eq:boundForCommuteTermsDif}) follow from (\ref{ineq:pApproxKboundsGen})-(\ref{ineq:pleqcKboundsGen}) as follows.  For $s = 0$, the bound (\ref{eq:boundForCommuteTermsDif}) is identical to the bound (\ref{ineq:pApproxKboundsGen}).  For $s \leq r+1$, the bound (\ref{eq:boundForCommuteTermsDif}) follows by induction on $s$ from the recursive formula
\ali{
Z_{s+1} \cdot \nab &= \left[ \fr{D_{\leqc k}}{\pr t}, \right] Z_s \cdot \nab \\
&= \left( \fr{D_{\leqc k}}{\pr t} Z_s - Z_s \cdot \nab P_{\leqc k} v \right) \cdot \nab
}
once the bounds (\ref{ineq:pApproxKboundsGen}), (\ref{ineq:pleqcKboundsGen}) are established.

The proof of Proposition (\ref{prop:commutingWithOperators}) explains in detail how to estimate the commutator (\ref{ineq:genCommutEstimates}) using the bounds (\ref{ineq:pApproxKboundsGen})-(\ref{ineq:pleqcKboundsGen}) in the cases $r \leq 1$ and $s \leq 3$.  These cases already contain all the necessary ingredients for the general case.

%Besides introducing more notation and carrying out a straightforward induction, the only additional idea in the proof of Lemma (\ref{lem:theMainLem}) goes into obtaining the higher order material derivative estimates (\ref{ineq:stressAndLPpressbds}) - (\ref{ineq:Plkpressbds}) for the Littlewood-Paley components of the pressure and the Reynolds stress.  Actually, 
\subsection{ Higher order advective derivatives of forcing terms } \label{sec:highOrderAdvecForcing}

To complete the induction, it now remains to show that the estimate (\ref{ineq:stressAndLPpressbds}) holds for $r = n+1$ assuming the cases $r \leq n$ of Lemma (\ref{lem:theMainLem}).  We concentrate first on the bound for the Reynolds stress, as the bounds for Littlewood-Paley projections of the pressure are similar.  First we recall the decomposition obtained in (\ref{eq:RleqkTrichotomy2}), (\ref{eq:RHLgood}), (\ref{eq:RLLgood}) and (\ref{eq:formWeUseHighHigh})
\ali{
R_{\leq k}^{jl} &= R_{\leq k, HH}^{jl} + R_{\leq k, HL}^{jl} + R_{\leq k, LL}^{jl} \label{eq:RleqkTrichotomy2} \\
R_{\leq k, LL}^{jl} &= \int \de_h P_{\leq k} v^j(x) \de_h P_{\leq k} v^l(x) \eta_{\leq k}(h) dh \label{eq:lastLowLowTerm} \\
R_{\leq k, HL}^{jl} &= \int P_{[k, k+2]} v^j(x+h) \de_h P_{\leq k}v^l(x) \eta_{\leq k}(h) dh \notag \\
&+ \int \de_h P_{\leq k} v^j(x) P_{[k, k+2]}v^l(x+h) \eta_{\leq k}(h) dh \label{eq:lastHighLowTerm} \\
R_{\leq k, HH}^{jl} &= \int ( v^j(x+h) - P_{\leq k} v^j(x+h) ) (v^l(x+h) - P_{\leq k}v^l(x+h) ) \eta_{\leq k}(h) dh \\
&= P_{\leq k} [( v^j - P_{\leq k} v^j ) (v^l - P_{\leq k}v^l ) ] \\
&= \sum_{I \geq k} P_{\leq k} ( P_I v^j P_{\approx I} v^l )
}
The last decomposition follows from the bandlimited property of Littlewood-Paley projections.

We already saw in the proofs of Propositions (\ref{prop:firstMatdvRlk}) and (\ref{prop:commutingWithOperators}) one way to estimate material derivatives $\fr{D_{\leqc k}^s}{\pr t^s}$ for terms of the type (\ref{eq:lastLowLowTerm}) and (\ref{eq:lastHighLowTerm}).  It is straightforward to see that we have the desired estimates 
\ali{
\co{ \nab^A \fr{D_{\leqc k}^s}{\pr t^s} R_{\lk, LL} } &\leq C 2^{(A + s(1-\a) - 2 \a)k} \ctdcxa{v}^{s+2} \\
\co{ \nab^A \fr{D_{\leqc k}^s}{\pr t^s} R_{\lk, HL} } &\leq C 2^{(A + s(1-\a) - 2 \a)k} \ctdcxa{v}^{s+2} 
}
for $s \leq n+1$, since the case $r = n$ of Lemma (\ref{lem:theMainLem}) allows us to take up to $n+1$ material derivatives of $P_{\approx k} v$ and $\nab P_{\leqc k} v$ provided $n(1-\a) - 2 \a < 0$.

The restriction $(n+1)(1-\a) - 2 \a < 0$ becomes important for summing the estimates in the High-High terms.  Here our goal is to estimate
\ali{
\Dkdt{n+1} R_{\leq k, HH}^{jl} &= \sum_{I \geq k} \Dkdt{n+1} P_{\leq k} ( P_I v^j P_{\approx I} v^l ) \label{eq:Dkdtnplus1RlkHH} \\
&= \sum_{I \geq k} \Dkdt{n+1} R_{\leq k, HH, I}^{jl}
}
using our bounds for $\fr{D_{\leq I}^{n+1}}{\pr t^{n+1}} P_{\approx I} v$.  

The main idea will be to take advantage of the fact that the bandlimited properties of Littlewood Paley projections allow us to express any operator $\fr{D_{\leqc k}^{n+1}}{\pr t^{n+1}} P_{\leqc k}$ in the form
\ali{
\Dkdt{n+1}& P_{\leqc k} &= {\sum_{\ell=0}^{n+1} \sum_{r_1, \ldots, r_{\ell}, m} C_{r_1, \ldots, r_{\ell}}} \prod_{i = 1}^\ell\left[ \Dkdt{} , \right]^{r_i} \left( P_{\leqc k} P_{\approx k} v \cdot \nab P_{\leqc k} \right) \prod_{j = 1}^m \left( P_{\leqc k} \fr{D_{\leq I}}{\pr t} P_{\leqc k} \right) P_{\leqc k} \label{eq:whatTheProductLooksLike}
}
where the sum runs only over indices $r_1 + \cdots + r_\ell + \ell + m = n+1$.  The frequencies of the projections $P_{\leqc k}$ remain bounded by $C \cdot 2^k$ because the number of factors $n+1$ is bounded in terms of the fixed $\a < 1$.

The starting point for the representation (\ref{eq:whatTheProductLooksLike}) is that we can use the bandlimited property of Littlewood-Paley projections to express
\ali{
\Dkdt{} P_{\leqc k} &= \left( P_{\leqc k} \Dkdt{} P_{\leqc k} \right) P_{\leqc k} \label{eq:stayLowFreq}
}
for some Littlewood-Paley projection $P_{\leqc k}$.  The point here is that each material derivative can only increase the overall frequency support by at most a factor of $C \cdot 2^k$.  

Using the bandlimited property again, the operator (\ref{eq:stayLowFreq}) can then be expressed as
\ali{
P_{\leqc k} \Dkdt{} P_{\leqc k} &=  P_{\leqc k} \left( \fr{D_{\leq I}}{\pr t} - P_{\approx k} v \cdot \nab \right) P_{\leqc k} \label{eq:eachMatDvIsLowPartOfHigh}
}
where the preponderance of the intermediate frequencies between $2^k$ and $2^I$ do not contribute to $P_{\approx k} v$ thanks to the  leftmost projection operator $P_{\leqc k}$.  Our first explicit example of this technique appeared already in the formula \eqref{eq:formWeUseHighHigh} above.

The decomposition (\ref{eq:whatTheProductLooksLike}) is achieved by induction on $n$.  First we use (\ref{eq:eachMatDvIsLowPartOfHigh}) to write
\ali{
\Dkdt{n+1} P_{\leqc k}  &= T_1 - T_2 \\
T_1 &= \Dkdt{n} \left( P_{\leqc k} \fr{D_{\leq I}}{\pr t} P_{\leqc k} \right) P_{\leqc k} \label{eq:nPowerLotsOfOps} \\
T_2 &=   \Dkdt{n} \left(P_{\leqc k} P_{\approx k} v \cdot \nab P_{\leqc k} \right) P_{\leqc k} \label{eq:needToCommuteABunch}
}
The term (\ref{eq:nPowerLotsOfOps}) has the form (\ref{eq:whatTheProductLooksLike}) after the leftmost factor of $\Dkdt{n} P_{\leqc k}$ in (\ref{eq:nPowerLotsOfOps}) has been expressed in the form (\ref{eq:whatTheProductLooksLike}) using the induction hypothesis.

For the term (\ref{eq:needToCommuteABunch}), we commute the material derivatives using the rule (\ref{eq:commuteWithPowers})
\ali{
T_2 &= \sum_{q = 0}^n \binom{n}{q} \left[ \Dkdt{}, \right]^{q} \left(P_{\leqc k} P_{\approx k} v \cdot \nab P_{\leqc k} \right) \Dkdt{n-q} P_{\leqc k} \label{eq:theAboutkTerm}
}
These terms all have the form (\ref{eq:whatTheProductLooksLike}) after the factor $\Dkdt{n-q} P_{\leqc k}$ has been expressed in the form (\ref{eq:whatTheProductLooksLike})  using the induction hypothesis.

 %we expand $\left( P_{\leqc k} \Dkdt{} P_{\leqc k} \right)^{n}$ in the from (\ref{eq:whatTheProductLooksLike}) from the induction hypothesis.  For the term (\ref{eq:needToCommuteABunch}) we use the formula (NEEDREF) to commute.

The formula (\ref{eq:whatTheProductLooksLike}) allows us to expand each term in the series (\ref{eq:Dkdtnplus1RlkHH}) as
\ali{
&\Dkdt{n+1} R_{\leq k, HH, I}^{jl} = \Dkdt{n+1} P_{\leq k}\left[ P_I v^j P_{\approx I} v^l \right] \\
&= {\sum_{\ell=0}^{n+1} \sum_{r_1, \ldots, r_{\ell}, m } C_{r_1, \ldots, r_{\ell}}} \prod_{i = 1}^\ell\left[ \Dkdt{} ,\right]^{r_i} \left( P_{\leqc k} P_{\approx k} v \cdot \nab P_{\leqc k} \right)  \prod_{j=1}^m \left( P_{\leqc k} \fr{D_{\leq I}}{\pr t} P_{\leqc k} \right) P_{\leq k}\left[ P_I v^j P_{\approx I} v^l \right] \label{eq:bigRHHExpand}
}
with $r_1 + \cdots + r_\ell + \ell + m = n+1$.  After fully expanding the commutators using (\ref{eq:commuteTermDefs}) and (\ref{eq:commuteWithPowers}), this decomposition and the cases $r \leq n$ of Lemma (\ref{lem:theMainLem}) give the estimate
\ali{
\co{ \nab^A \Dkdt{n+1} R_{\leq k, HH, I}^{jl} } &\leq C_A 2^{Ak} 2^{((n+1)(1-\a) - 2 \a)I} \ctdcxa{v}^{n+3}, \label{eq:boundForRlkHHI}
}
where the worst estimate arises from the terms of the form $\fr{D_{\leq I}^{n+1}}{\pr t^{n+1}}[P_I v^j P_{\approx I} v^l]$, which appear in the case $\ell = 0$ and $m = n+1$.  Note that the leftmost operator on every term in (\ref{eq:bigRHHExpand}) has the form $\left[ \Dkdt{}, \right]^s P_{\leqc k}$ for some $s \leq n+1$, which ensures that spatial derivatives never cost more than $C 2^k$ by the localization to frequencies $\lesssim 2^k$.  
% the case $r \leq n$ for the bounds (\ref{eq:boundForCommuteTermsDif}) in Lemma (\ref{lem:theMainLem}).  
For $(n+1)(1-\a) - 2 \a < 0$, the bound (\ref{eq:boundForRlkHHI}) can be summed over $I \geq k$ to give
\ali{
\co{ \nab^A \Dkdt{n+1} R_{\leq k, HH}^{jl} } &\leq C 2^{( A + (n+1)(1-\a) - 2 \a)k} \ctdcxa{v}^{n+3},
}
which concludes the proof of estimate (\ref{ineq:stressAndLPpressbds}) for the Reynolds stress.  The bound for the pressure in (\ref{ineq:stressAndLPpressbds}) is proven in essentially the same way using the analogous trichotomy decomposition achieved in Proposition (\ref{prop:DdtnabDPkp}).  From this bound, the Lemma (\ref{lem:theMainLem}) for $r = n+1$ follows from the discussion in Section (\ref{sec:reduceToForcing}), which completes the inductive proof of Lemma (\ref{lem:theMainLem}).

\section{Smoothness of trajectories} \label{sec:smoothTraject}

Here we show how the results of Section \ref{sec:theGeneralCase} can be used to prove the existence of smooth particle trajectories.  We consider the setting of Theorem \ref{thm:smoothTraject} and we now assume that the velocity field $v(t,x)$ has borderline regularity $v(t,x) \in \cap_{\a < 1} L_t^\infty C_x^\a$.  Our goal is to construct a smooth particle trajectory $X(t)$ through $x_0$ at time $0$; that is, a $C^\infty$ solution to
\ali{
X(t) &= x_0 + \int_{0}^t v(s,X(s)) ds \label{eq:uniqueFlow}
}
For the coarse scale velocity fields $P_{\leq k} v$, we have well-defined particle trajectories $X_{(k)}(t)$ satisfying
\ali{
X_{(k)}(t) &= x_0 + \int_{0}^t P_{\leq k} v(s, X_{(k)}(s)) ds \label{eq:approxTrajectory}
}
From the identity \eqref{eq:approxTrajectory} we can see that the curves $X_{(k)}(t)$ are Lipschitz in $t$ uniformly in $k$, and therefore form an equicontinuous family of functions mapping $t \in I$ into $\T^n$.  Thus, the sequence $X_{(k)}(t)$ has a subsequence (not relabeled) converging uniformly on compact sets to some limit as $k \to \infty$.  Let $X(t)$ be any such uniform limit; then $X(t)$ satisfies \eqref{eq:uniqueFlow} by passing to the limit in \eqref{eq:approxTrajectory}. %Fix such a subsequence and %, and we have assumed a unique solution to \eqref{eq:uniqueFlow}.  It follows by a simple contradiction argument that $X_{(k)}(t,x_0) \to X(t,x_0)$ uniformly on compact subsets of $I$.  

Now note that the particle trajectories $X_{(k)}$ are smooth in time by the results of Section \ref{sec:theGeneralCase}, with
\ali{
\fr{d^{r+1}}{dt^{r+1}} X_{(k)}(t) &= \fr{D_{\leq k}^r}{\pr t^r} P_{\leq k} v(t, X_{(k)}(t))
}
In particular, we have Taylor's formula for all $t_0, t_1 \in I$
\ali{
\label{eq:TaylorsFormula}
\begin{split}
X_{(k)}(t_1) &= X_{(k)}(t_0) + \sum_{r=1}^N \fr{D_{\leq k}^{r-1}}{\pr t^{r-1}} P_{\leq k} v(t, X_{(k)}(t_0)) \fr{(t_1 - t_0)^r}{r!} \\
&+ \fr{(t_1 - t_0)^{N+1}}{N!}\int_0^1 \fr{D_{\leq k}^{N}}{\pr t^{N}} P_{\leq k} v(t_s, X_{(k)}(t_s)) (1 - s)^N ds 
\end{split}
\\
t_s &= t_0 + s(t_1 - t_0) \notag
}
By the results of Section \ref{sec:timeRegOfVelocField}, we have furthermore that $\fr{D_{\leq k}^r}{\pr t^r} P_{\leq k} v$ converges uniformly on $I \times \T^n$ to $\fr{D^r}{\pr t^r} v$ for any $r \geq 0$.  This observation and the uniform convergence of $X_{(k)}(t) \to X(t)$ along the subsequence allow us to pass to the limit in Taylor's formula \eqref{eq:TaylorsFormula} and conclude that $X(t)$ is smooth with $\fr{d^{r+1}}{dt^{r+1}} X(t) = \fr{D^r}{\pr t^r}v(t, X(t))$.  

If we assume only that $v \in L_t^\infty C_x^\alpha$ for some $0 < \alpha < 1$, then the same argument yields a trajectory of class $C^r$ for all integers $r < \frac{1}{1-\alpha}$.  Note that the argument only takes advective derivatives of functions that are $C^1$, in which case the weak advective derivative is the same as the classical one.  Thus we have proven Theorem \ref{thm:smoothTraject}.

\section{Concluding Remarks} \label{sec:conclusion}

Several parts of the analysis in this paper give a new point of view on convex integration constructions of Euler flows and the pursuit of Onsager's conjecture and its still open variants.  One point that the analysis clarifies is that some of the special estimates for material derivatives in these constructions are forced by the Euler equations, rather than being artifacts of the constructions.  These bounds give another point of view on the constraints one faces for a type of scheme that could be used to approach the still open variants of Onsager's conjecture such as the two-dimensional case.  For example, the bounds on material derivatives (inequality (\ref{bd:matDvOfLpPiece2}) in particular) show that the natural time scale associated to frequency $\la \approx 2^k$ is on the order $\la^{-(1-\a)} \ctdcxa{v}^{-1}$.  In particular, any time cutoffs employed in a construction at frequency $\la$ should have a lifespan at least least $\sim \la^{-(1-\fr{1}{3})}$ in order to be compatible with the desired spatial regularity $1/3$.  This constraint, which is satisfied in the solution of Onsager's conjecture \cite{isettOnsag}, rules out certain short-time methods that are employed in the construction of $(1/5-\ep)$-H\"{o}lder solutions in \cite{isett, deLSzeBuck}.  See \cite{Buckmaster, buckDeLSzeOnsCrit} for examples of schemes that partially avoid the restrictions on short time methods to make progress on Onsager's conjecture in weaker topologies.

Theorem \ref{thm:energyReg} on the regularity of the total energy suggests some further questions regarding the energy profiles of Euler flows.  De Lellis and Sz\'{e}kelyhidi have shown \cite{deLSzeHoldCts} that the energy profile of an Euler flow constructed by convex integration can be essentially any smooth, positive function (see also \cite{deLSzeBuck}).  Refinements of this result tailored to the initial value problem show that uniqueness for the initial value problem for the Euler equations in H\"{o}lder spaces cannot be restored by many natural ``entropy criteria'' one might propose (see \cite{deLSzeAdmiss, Daneri}). %There is a restriction in \cite{deLSzeHoldCts, deLSzeBuck} that the energy profile is bounded below by a positive constant, but this restriction may be purely technical.  
It is reasonable to suspect that the energy profile can also be made rough as well, and it would be interesting to see whether the regularity in Theorem (\ref{thm:energyReg}) is sharp since the proof of Theorem (\ref{thm:energyReg}) is closely related to the proof of energy conservation in \cite{CET}.  For exponents $\a < 1/5$, the sharpness of regularity for the energy profile has now been proven in \cite{isettOhNonpd}.  %(This result would simultaneously exhibit that the velocity fields obtained from the construction are indeed no better than the claimed exponent.)

It would be of further interest to show that irregularity of the energy profile is a generic behavior for solutions with regularity strictly below $1/3$.  That is, for an Euler flow that is generic in a space such as $C_tC_x^\a$ with $\a < 1/3$, we expect that the energy profile will not belong to any space with better regularity than $C_t^{\fr{2 \a}{1 - \a}}$, and furthermore should fail to be of bounded variation on every time interval.  In particular, a small perturbation of Euler flows in $C_tC_x^\a$ for $\a < 1/3$ should generically lead to an irregular energy profile that fails to be monotonic on every time interval, in contrast to the discussion of the case $\a = 1/3$ in Section~\ref{sec:timeRegEnergy}, which shows the stability of energy dissipation in the critical space.  Such a result would indicate that energy dissipation at regularity below $1/3$, while possible, is an unstable phenomenon, so that the $1/3$ law \eqref{law:oneThird} would be the only possible law for velocity fluctuations that is compatible with the dissipation of energy in a robust sense.  %See \cite{isettOhNonpd, isettOhKinEn} for results in support of this conjecture in the range $\a < 1/5$. 
\bibliographystyle{alpha}
\bibliography{newEulerRegularity}

\end{document}